\documentclass{ecgd-l}

\usepackage{amssymb,amsmath}
\usepackage{url}
\usepackage{amscd}
\usepackage{texdraw}
\usepackage[all]{xy}
\usepackage{fancyhdr}
\usepackage{amsmath}
\usepackage{graphicx}
\usepackage{subcaption}
\DeclareMathOperator{\sech}{sech}
\usepackage{amscd,color}
\usepackage[shortlabels]{enumitem}

\theoremstyle{plain}
\newtheorem{theorem}{Theorem}

\newtheorem{corollary}{Corollary}[section]
\newtheorem{lemma}{Lemma}[section]
\newtheorem{remark}{Remark}[section]
\newtheorem{proposition}{Proposition}[section]
\newtheorem{definition}{Definition}[section]
\newtheorem*{lemma1}{Lemma}
\input txdtools

\newcommand{\CC}{{\mathbb C}}

\newcommand{\RR}{{\mathbb R}}

\newcommand{\ZZ}{{\mathbb Z}}

\newcommand{\Rt}{{\mathcal R}_t}
\newcommand{\R}{{\mathcal R}}

\begin{document}
\title[Cycle Doubling and Merging]
{Cycle Doubling, Merging And Renormalization in the Tangent Family}

\author{Tao Chen, Yunping Jiang, and Linda Keen}

\address{}
\email{}

\thanks{}


\begin{abstract} 
In this paper we study the transition to chaos  for  the restriction  to the real and imaginary axes of the tangent family $\{ T_t(z)=i t\tan z\}_{0< t\leq \pi}$.  Because tangent maps have no critical points but have an essential singularity at infinity and two symmetric asymptotic values,  there are new  phenomena:  as $t$ increases we find single instances of  ``period quadrupling'',   ``period splitting'' and standard ``period doubling'';   there follows  a general pattern of ``period merging'' where two attracting cycles of period $2^n$ ``merge'' into one attracting cycle of period $2^{n+1}$, and ``cycle doubling'' where an attracting cycle of period $2^{n+1}$ ``becomes'' two attracting cycles of the same period. 

We  use renormalization to prove the existence  of these bifurcation parameters.  
The uniqueness of the cycle doubling  and cycle merging parameters is quite subtle and requires a new approach.  
To prove the cycle doubling and merging parameters are, indeed, unique,  we apply the concept of ``holomorphic motions'' to our context.   

In addition,  we prove that there is an ``infinitely renormalizable'' tangent map  $T_{t_\infty}$.   It has no attracting or parabolic cycles.  Instead, it has a strange attractor contained in the real and imaginary axes which is forward invariant and minimal under  $T^2_{t_\infty}$. 
The intersection of this strange attractor with the real line consists of two binary Cantor sets and 
the intersection  with the imaginary line is  totally disconnected, perfect  and unbounded.
\end{abstract}

\maketitle

\section{Introduction}\label{intro}

In the 1970s, Feigenbaum~\cite{Fe1,Fe2}, and independently, Coullet and Tresser~\cite{ACT,CT}, discovered an interesting phenomenon in physics  called  {\em period doubling} that showed how a sequence of dynamical systems with stable dynamics can converge to one with chaotic dynamics (see e.g.~\cite{JBook}).  They began with the quadratic family $Q_{t}(x)=-(1+t)x^{2} +t$ parameterized by $t \in [0,1]$.  For all $t$, $Q_t$ maps the interval $[-1,1]$ into itself, fixing $-1$.   While for very small values of $t$, every point  inside the interval $(-1,1)$  is attracted by a fixed point inside the interval, eventually one encounters a strictly increasing sequence $\{t_{n} \}_{n=1}^{\infty}$ such that for $t \in (t_{n-1},t_{n})$, $Q_{t}$ has a  repelling cycle $C_{k,t}$ of period $2^{k}$ for $k=1, \ldots, n-1$ and an attracting cycle $C_{n,t}$ of period $2^n$. That is, as $t$ passes through each $t_n$, the period of the attracting cycle is doubled. 

Since the $t_n$'s form a bounded increasing sequence, they have a limit $t_{\infty}$.   The limit map $Q_{t_{\infty}}$  is a quadratic polynomial with repelling cycles $C_{k,t_{\infty}}$ of period $2^k$ for all positive integers $k$.   It has no attracting or parabolic cycle, but the orbit of the critical point $0$,   attracts all points that do not land on one of  the repelling cycles.   This attractor is homeomorphic to the standard $1/3$ Cantor set and $Q_{t_{\infty}}$ acts on it as an adding machine; thus, it is minimal in the sense that the orbit of every point is dense.   

Following that early work,  period doubling was found to occur in many branches of mathematics, physics, chemistry, and biology, etc..
 In this paper, we exhibit an analogous phenomenon, 
which we  call {\em  cycle doubling}, and a new phenomenon, which we  call {\em cycle merging}, that occurs for the tangent family
$$
\{ T_{t} (z)=it \tan z\}_{0< t\leq \pi}.
$$
Before we get to what that means, we need to fill in some background about the
tangent family.

Each $ T_{t}$ is a meromorphic map of the complex plane with poles at  $\{k\pi+\pi/2\}_{k\in \mathbb{Z}}$. 
Unlike the quadratic maps, it has no critical points. It does have an essential singularity at infinity with a symmetric pair of asymptotic values, $\{-t, t\}$ that are limits of $T_t$ along paths tending to infinity in the directions of the positive and negative imaginary axes respectively.  In this sense, the asymptotic values can be thought of as ``virtual images" of infinity.   As a general principle, the dynamics of a system generated by iterating a map are controlled by the orbits of the points where the map is not a regular covering: for quadratic maps, this set is the orbit of the critical value, and for the tangent map, it is the orbits of the two symmetric asymptotic values. 

For real values of $t$,  $T_{t}$ maps the real line $\Re$ to the imaginary line $\Im$ and maps $\Im$ to $\Re$ so that $f_t=T_t^2$ is a self map of both $\Re$ and $\Im$.   In the same sense that the asymptotic values are virtual images of infinity under $T_t$, they are virtual images of the poles of $T_t$ under $f_t$.  
We study the dynamics of $T_t$ and $f_t$ restricted to these axes and  show that  as $t$ moves from left to right in $(0, \pi)$, we see first a ``period quadrupling'', next a ``period splitting'' and then a ``period doubling'' like that for quadratic maps.  Unlike quadratic maps, however, this period doubling occurs only for a single value of $t$.  Afterwards,  as $t$ increases,  a  general pattern occurs either of ``cycle merging'', in which two attracting cycles of period $n$ merge to form a cycle of period $2n$, or ``cycle doubling'', where instead of seeing one new cycle of period $2n$ form from two cycles of period $n$, we see two new attracting cycles  of period $n$ form from the single cycle of period $n$. The result is a pair of interleaved strictly increasing sequences where these phenomena happen, $\{\alpha_n=t_n^{double}\},\{\beta_n= t_n^{merge}\}$, that have a common limit $t_{\infty}$ (see Figures~\ref{fig1} and~\ref{fig10}).  

We prove that, like the limit in the quadratic case, $T_{t_{\infty}}$ has   no attracting or parabolic cycle. Instead, it has an attractor $C$ contained in  the real  and  imaginary lines and it attracts almost all points on the these lines.  In the real line,  it consists of two binary Cantor sets and is forward invariant and minimal under $f_{t_{\infty}}$. On the imaginary line, it is a forward invariant, unbounded, totally disconnected,  perfect subset. 

\begin{figure}[ht]
    \includegraphics[width=4in]{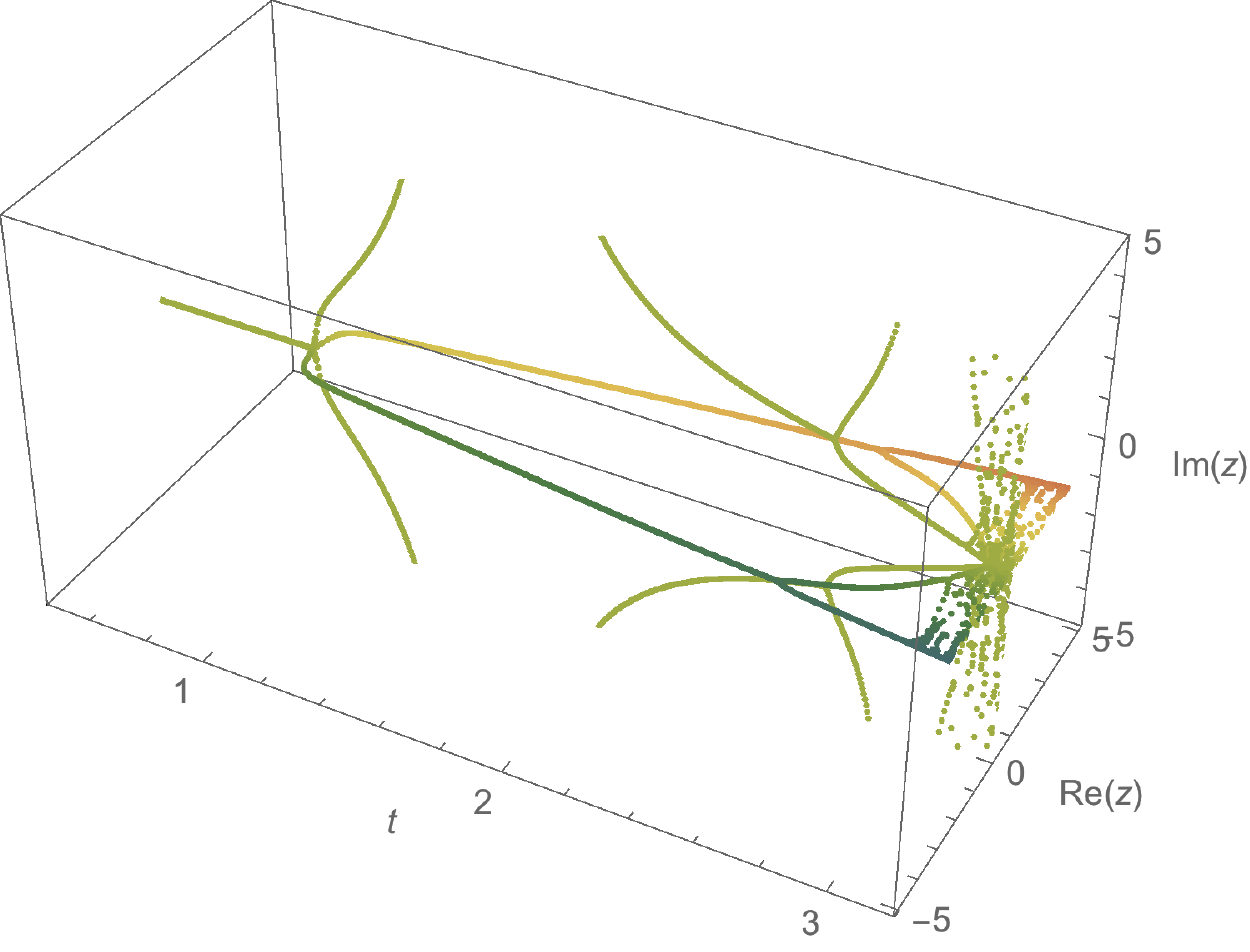}
    \caption{This is a computer generated picture illustrating the bifurcation diagram for $\{ T_{t} (z)=it \tan z\}_{0< t\leq \pi}$.}~\label{fig1}
\end{figure}

Our proofs involve modifying, for the tangent family,  standard techniques for real and complex dynamical systems. The proof that cycle doubling occurs at $\alpha_n$ uses fairly standard results about bifurcation near parabolic periodic points.  {\em Cycle merging, however, is inherently a phenomenon for meromorphic maps with symmetric asymptotic values.}   It depends on adapting the  ``renormalization'' process for polynomials to the tangent family.  For the quadratic family,  at each $t_n$ a new map,  its ``$n^{th}$-renormalization'', is defined as $Q_{t_n}^{2^n}$ restricted to a subinterval of $[-1,1]$ that contains the critical point and on which the iterate is a unimodal map. (See e.g. \cite{JBook,MSBook}).  In  the tangent family, for each $\beta_n$,  we define the ``$n^{th}$-renormalization'' to be the iterate $f_t^{2^n}$ restricted to a pair of intervals each containing a pole, and each bounded by certain pre-poles.   The renormalized map is ``tangent-like", that is, continuous and monotonic except at the poles. We prove that  the map $T_{t_{\infty}}$ is infinitely renormalizable and the orbits of the asymptotic values form an attractor $C$.  (See ~\cite{MilA} for the definition of an attractor). 

The renormalization process gives a complete solution to the existence   of cycle doubling    and cycle merging parameters  in the tangent family. The  uniqueness of the cycle doubling and cycle merging parameters is quite subtle and requires a new approach.     
To prove the cycle doubling and merging parameters are, indeed, unique, we adapt ideas used in \cite{LSS} to   study entropy of folding maps.   In particular, we apply the concepts of ``holomorphic motions'' and ``transversality'' to our context.

We note that just as renormalization for quadratics can be extended to the complex parameter plane of the quadratic and other polynomial families, (see e.g. \cite{Dou1,Dou2,JBook,McM}), renormalization exists in the complex $t$ plane of the tangent family where there is cycle doubling and merging.   In fact, renormalization also exists in dynamically natural slices (see \cite{FK}) of more general families of meromorphic functions.  We see cycle doubling and merging in those slices where there are two asymptotic values that behave symmetrically. For example, we see these phenomena in  the family $\lambda \tan^3 z$,  whose functions have two symmetric asymptotic values and one superattractive fixed point.   We leave these generalizations for another paper.

The paper is organized as follows. In \S\ref{basic},  we  describe some basic facts about the tangent family $it \tan z$ with $t \in (0, \pi]$. 
In \S\ref{quadrupling}, we show that as $t$ increases an attracting fixed point ``quadruples'' into  a period $4$ attracting cycle. Then
in \S\ref{sec:period splitting}, we show  that the  period $4$ cycle ``splits'' to become two period $2$ attracting cycles. 
In \S\ref{doubling-ren}, we show how the two period $2$ attracting cycles ``period double'' to become two period $4$ attracting cycles.  In the same section, 
we define the first renormalization; it is the paradigm for the higher order renormalizations needed to show the  general pattern.  
In \S\ref{merging1}, we give the first example of the ``cycle merging'' phenomenon:    two period $4$ attracting cycles merge to become a single period $8$ attracting cycle.
In \S\ref{std doubling}, we give the first example of the ``cycle doubling'' phenomenon:  the single period $8$   cycle doubles to become two period $8$ attracting cycles.
In \S\ref{merging2}, we show the next example of ``cycle merging'':  the two period $8$  cycles merge into one period $16$ attracting cycle.  We will see that 
this second merging is somewhat different from the first merging.
 In \S\ref{gen pattern}, we state and prove our first main result (Theorem~\ref{bif}) which says  the sequences of  cycle doubling and cycle merging phenomena exist and are part of a  general pattern for the tangent family. The proof is by induction and the main tool in the induction step is renormalization. 
In  \S\ref{tran} we introduce holomorphic motions and use them to  prove transversality at the cycle merging parameters $\beta_n$ (Theorem 3). Because  our family is restricted to  the real and imaginary axes, and depends on  real parameters, we obtain positive transversality. This is what we need for  the uniqueness of the $\beta_n$, which in turn, gives us the uniqueness of the $\alpha_n$.
 In \S\ref{inf ren}, we show there exists an  infinitely renormalizable  map $T_{t_{\infty}}$ and prove  the second main result (Theorem~\ref{infrenorm}) 
which describes the  Cantor-like structure of its strange attractor.  The  appendix contains the proof of a standard lemma on parabolic bifurcation.

\vspace{30pt}
\noindent {\em Acknowledgement.}
All three authors are partially supported by awards from PSC-CUNY. The second author is partially supported by  grants from the Simons Foundation [grant number 523341], the NSF [grant number DMS-1747905), and the NSFC [grant number 11571122]. This work was partially done when the second author visited the National Center for Theoretical Sciences (NCTS) and he would like to thank NCTS for its hospitality. We would like to thank Professors Enrique Pujals, Charles Tresser, and Tomoki Kawahira for helpful discussions.  We also thank Professor Kawahira for his help in creating some of the figures.  We would also like to thank Jonathan Brezin for proofreading and improving the exposition in our paper.

\vspace{30pt}

\section{Facts about Tangent Maps}\label{basic}
\subsection{Basic Facts}
Let $\mathbb{R}$ be the real line, $\mathbb{C}$ the complex plane, and $\widehat{\mathbb{C}}=\mathbb{C}\cup \{\infty\}$ the Riemann sphere. The tangent map is defined as
$$
\tan (z) =\frac{1}{i} \frac{e^{iz} -e^{-iz}}{e^{iz}+e^{-iz}}.
$$
Let 
\begin{equation}~\label{ttan}
T_{t} (z) =it \tan (z) =t  \frac{e^{iz} -e^{-iz}}{e^{iz}+e^{-iz}}.
\end{equation}
The family we consider in this paper is the subfamily of the tangent family
\begin{equation}~\label{tf}
{\mathcal T}=\{ T_{t} (z)=it \tan z: \mathbb{C}\to \widehat{\mathbb{C}}\}_{0<t\leq \pi}.
\end{equation}

We use $\Im=\{ iy\;|\; y\in {\mathbb R}\}$ to denote the imaginary line in the complex plane. Let $\Im^{+}=\{ iy\;|\; y>0\}$ and $\Im^{-}=\{ iy\;|\; y<0\}$ be the positive and negative rays in $\Im$, respectively. 
From the definition we have
$$
T_{t} (iy) = t \frac{e^{-y} -e^{y}}{e^{-y}+e^{y}}.
$$
Thus, 
\begin{equation}~\label{av}
\lim_{y\to \infty} T_{t}(iy) =- t \quad \hbox{and}\quad \lim_{y\to -\infty} T_{t}(iy) = t.
\end{equation}
 Since $T_{t}$ is continuous on $\Im$, we see that the restriction of $T_{t}$ to $\Im$ is a map onto $(-t, t)$ and it is a strictly decreasing function of $y$, where the ordering on $\Im$ is given by the rule: $iy >ix$ if $y>x$. 
 
The family ${\mathcal T}$ has the following properties: 
\begin{enumerate}[(a)]
\item Each $T_t$ in ${\mathcal T}$  has no critical points.
\item
The values $\{ -t, t\}$ are both omitted by all maps $T_{t}$ in ${\mathcal T}$; they are the 
{\em asymptotic values} of $T_{t}$. 
\item All  maps $T_t$ in ${\mathcal T}$ have poles at the same set of  points,  $z_{k}= k\pi + \pi/2$, $k \in \ZZ$. They are all periodic with  period $\pi$; that is, 
$$
T_{t} (z+\pi) =T_{t}(z),\quad  \forall \;\; z\in {\mathbb C}.
$$ 
 \item
 Every map $T_t$ in ${\mathcal T}$ is an odd function; that is $T_{t}(-z)=-T_{t}(z)$ and so  $T_t$  has a fixed point at $0$. 
 The multiplier at $0$ is 
 \begin{equation}~\label{ev0}
 \lambda_{0} = (T_{t})' (0) =it.
 \end{equation}
 \item 
For  every $T_t$,  the preimages of $0$ are the points $k\pi$,  $k \in \ZZ$,   on the real line. In each fundamental interval $(k\pi -\pi/2, k\pi +\pi/2)$, $T_{t}$ is a continuous strictly increasing function onto the imaginary line $\Im$. 
\end{enumerate}
 
 If we compose $T_{t}$ in ${\mathcal T}$ with itself, we obtain an odd periodic map of period $\pi$ from $\RR$ to itself, 
\begin{equation}~\label{ffunc}
f_{t}(x)=T_{t}^{2}(x)=-t \tanh (t \tan x) = -t \Big[\frac{e^{t\tan x}-e^{-t \tan x}}{e^{t \tan x}+e^{-t \tan x}} \Big],
\end{equation}
whose derivative is an even periodic function
\begin{equation}~\label{deriv}
f_{t}'(x) =-t^{2}\sech^{2} (t\tan x) \sec^{2} x=- \frac{4t^{2}\sec^{2} x}{(e^{t\tan x}+e^{-t\tan x})^{2}}.
\end{equation}
Let $\Delta x= x-(k\pi +\pi/2)<0$ (or $=x-(k\pi-\pi/2)>0$). Since  
$$
\lim_{\Delta x\to 0} -\tan x= \lim_{\Delta x\to 0}\frac{\cos \Delta x}{\sin \Delta x} =\lim_{\Delta x\to 0}\frac{1}{\Delta x} \mbox{ and }
$$
$$
\lim_{\Delta x\to 0} -\sec^{2}x = \lim_{\Delta x\to 0} \frac{1}{\sin^{2}\Delta x} = \lim_{\Delta x\to 0} \frac{1}{(\Delta x)^{2}},
$$  
\begin{equation}~\label{approxderiv}
f_{t}'(x) = -4t^{2} \frac{e^{-\frac{2t}{|\Delta x|}}}{(\Delta x)^{2}} +o(|\Delta x|).
\end{equation} 

\begin{figure}[ht]
    \includegraphics[width=4in]{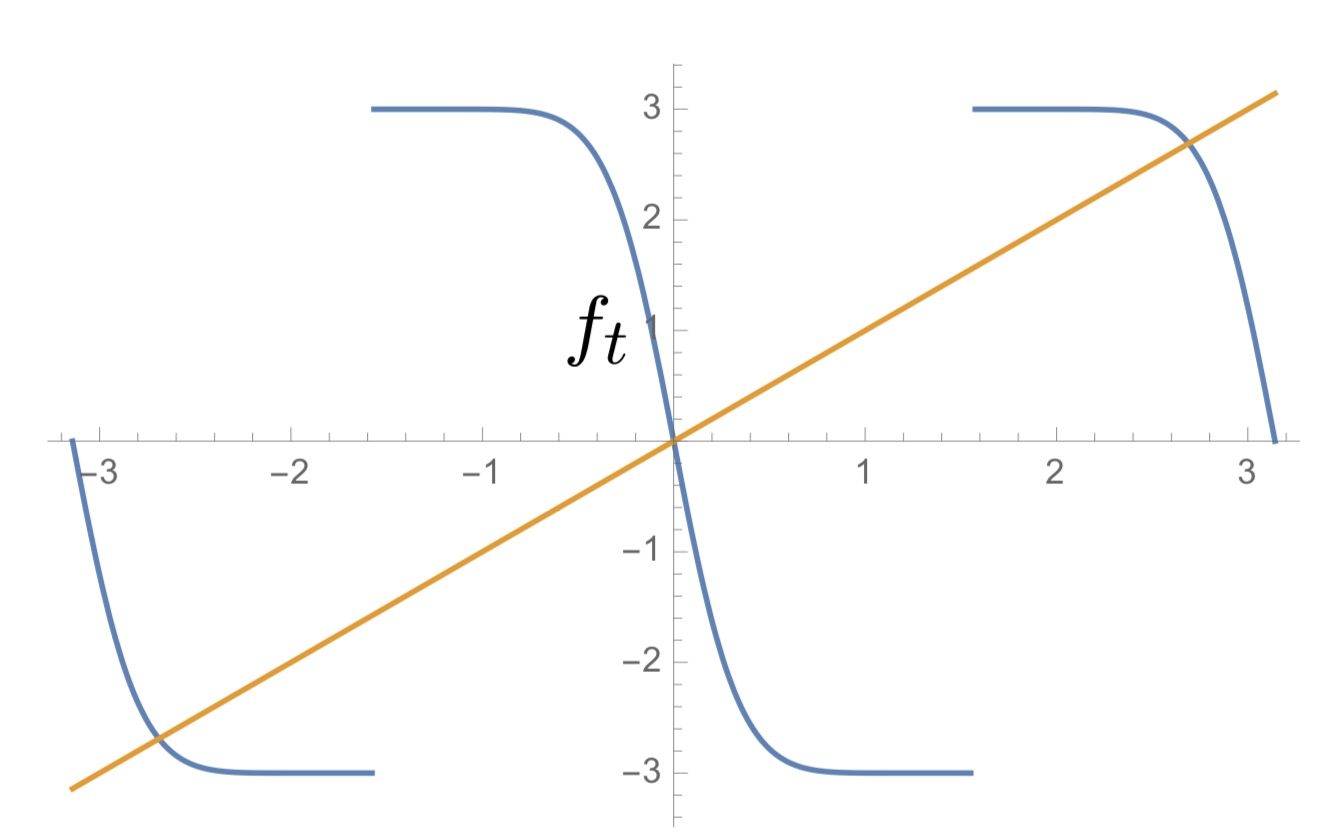}
    \caption{Graph of $f_{t}$}~\label{fig2}
\end{figure}

The map $f_{t}$ is an odd function with countably many discontinuities at the points $k\pi +\pi/2$ for $k \in \ZZ$.
It is a strictly decreasing smooth map from each fundamental interval $(k\pi -\pi/2, k\pi+\pi/2)$ onto $(-t, t)$ (see Figure~\ref{fig2}). 
The points $k\pi +\pi/2$  constitute the set of  poles of $T_t$;  by abuse of notation, we will call them the {\em poles of $f_t$.}    Similarly, since the points $x_k$ such that $f_t^m(x_k)=k\pi +\pi/2$, $m=0,1,2,\ldots, $ are   called pre-poles of $T_t$ of order $2m+1$, we will call them {\em pre-poles of $f_t$.}  

In what follows, and throughout the rest of the paper,  we use the notation $y^{\pm}$ for the upper and lower limits of $\lim_{x \to y^{\pm}} x$ and the notation  $f_t(y^{\pm})$ for the upper and lower limits  $\lim_{x \to y^{\pm}} f_t(x)$.  With this notation, 
 the function $f_{t}$ can be extended to the closed interval $[k\pi-\pi/2, k\pi+\pi/2]$ 
 continuously with image $[-t, t]$ by setting 
$$
f_t\big( (k\pi-\frac{\pi}{2})^{+}\big) =t 
$$
and 
$$
f_t\big( (k\pi+\frac{\pi}{2})^{-}\big)=-t. 
$$
Again, by abuse of notation we use $f_{t}$ for this extension as well. 
From (\ref{approxderiv}), for every $n\geq 1$, 
the $n^{th}$-derivatives satisfy $f_{t}^{(n)} ((k\pi+\pi/2)^-)=0$ and $f_{t}^{n} ((k\pi-\pi/2)^+)=0$.  
Thus the points $k\pi \pm \pi/2$ are flat critical points of $f_{t}$. 

 The Schwarzian derivative of $f_{t}$ is, by definition, 
$$
S(f_{t}) =\frac{f_{t}'''}{f_{t}'} -\frac{3}{2} \Big(\frac{f_{t}''}{f_{t}'}\Big)^{2}.
$$
Since $f_{t}$ is the composition of $-t\tanh x $ and $t\tan x$ and since $S(t\tan) =2$ and $S(-t\tanh) =-2$, we have
$$
S(f_{t}) =S(-t\tanh )\circ \tan \cdot (t \sec^{2} )^{2} +S(t \tan) = 2(- t^{2} \sec^{4} +1)
$$
Thus $S(f_{t}) <0$ when $t>1$.

\subsection{Some basic dynamics for tangent maps}

A set of points \\ $C=\{ z_{1}, \cdots, z_{n} \}$ is called a {\em period} $n$ cycle if $T^n_t(z_i)=z_i$, and $T_{t} ^{k}(z_{i})=z_{(i+k)\pmod{n}} \not=z_{i}$ for all $1\leq k\leq n-1$. Let $\lambda_{C}=(T_{t}^{n})'(z_{1})$ be the multiplier 
of $C$. We say $C$ is {\em attracting}, {\em parabolic}, or {\em repelling} if $|\lambda_{C}|<1$, $\lambda_{C}=e^{2p\pi i/q}$ or $|\lambda_{C}|>1$, respectively.
As in the theory of dynamics of rational maps (see e.g. ~\cite{Mil}), it was proved in~\cite{K,KK1,KK2}  that  the immediate basin of every attracting or parabolic cycle contains at least one asymptotic value.  By the symmetry of the tangent maps, if a given $T_t$ has an attracting cycle or parabolic cycle $C$, then $-C=\{-z_{1}, \cdots, -z_{n} \}$ is also attracting or parabolic with the same multiplier.  This means that either $C$ and $-C$ each attracts one asymptotic value or $C=-C$ and it attracts both asymptotic values.   It follows that $T_t$ can have no other attracting or parabolic cycles.

We will need the following basic lemmas. 
 
\medskip
\begin{lemma}\label{allevenper}
    For all $t>1$, the period of  any  attracting or parabolic  cycle of $T_t$ is even. 
\end{lemma}

\begin{proof}
For real $x$ the map $T_t (x)=it\tan(x)$ maps the real line to the  imaginary line; it is  continuous and strictly monotonically  increasing on each fundamental interval. 
The map $T_t (iy)= it\tan (iy)=-t\tanh (y)$ maps the imaginary line onto the  real interval $(-t,t)$; it is  continuous and strictly  monotonically decreasing.  If $T_t$ has an attracting or parabolic periodic cycle, then one of the asymptotic values $\pm t\in \mathbb{R}$ must be attracted to it.  The orbit of the asymptotic value must alternate between the real and imaginary axes so the points in this attracting  periodic cycle must lie in the union of the real  and  imaginary lines.  It follows that the period of any attracting or parabolic periodic cycle other than $0$ must be even.
\end{proof}

\medskip
\begin{lemma}\label{interval}
Suppose $t>1$ and    suppose $T_t$ has a period $p$ attracting cycle. Then the intersection of each component of the immediate basin of this attracting cycle with the real line or the imaginary line must be an interval.
\end{lemma}

\begin{proof}  By the lemma above, all points in the cycle are in the real and imaginary axes.  
Let $D$ be a component of the immediate basin of the cycle  that intersects  the real axis.   Let $x\in \overline{D}$ (the closure of $D$) be the point of the cycle. 
Then $T^p_t (x)=x$ is in or on the boundary of $D$.  It was shown in \cite{DK} that such a $D$ must be simply connected.  
    We claim  that $D$ is symmetric with respect to complex conjugation, that is, $\bar{D}=D$. This together with $D \cap {\mathbb R} \neq \emptyset$  will prove the lemma.

Now let us prove the claim. For any $z \in D$, $T^{pn}_t(z)\to x$ as $n\to \infty$.  By the symmetry of  $T_t^2$, $\overline{T^{2}_t(z)}=T^{2}_t(\bar{z})$. This implies that $T^{pn}_t(\bar{z})\to \bar{x}=x$ so that $\bar{z}\in D$.
\end{proof}

As an immediate corollary we have
\begin{corollary}
Suppose $t>1$ and suppose $T_{t}$ has an attracting cycle. Then any boundary point of  a component of the immediate basin of this attracting  cycle  
on the real or imaginary axis  is either a repelling periodic point  or a pre-pole  of $T_t$.
\end{corollary}

If the function $T_{i\lambda}=\lambda \tan z$, $\lambda \in \mathbb C \setminus \{0\}$ has an attracting cycle it is called {\em hyperbolic}.  It was proved in \cite{KK1,FK} that the hyperbolic components of the $\lambda$-plane are universal covering spaces of the punctured disk.   This immediately implies

\begin{proposition}\label{monomult}The intersection of any hyperbolic component with the real line $\lambda=t$ is an interval $I$.
Moreover, the multiplier of the attracting cycle  for $t \in I$ is a monotonic function with absolute value between $0$ and $1$. 
\end{proposition}

\section{Period Quadrupling: One period one to one period four}\label{quadrupling}

In this section, we  describe the  {\em period quadrupling phenomenon} in ${\mathcal T}$ that occurs as  $t$ increases through the point $1$ (see Figure~\ref{fig1} and Figure~\ref{fig10}).   That is, we see that the attracting fixed point of $T_t$ for $t<1$ becomes parabolic at $t=1$ and then repelling for $t>1$.  As $t$ increases past $1$, a new period four attracting cycle is born.

The following lemma is easy to prove.  

\medskip
\begin{lemma}~\label{Lem1}
For every $0<t< 1$,  $0$ is an attracting fixed point of  $T_{t}$; there are no other attracting cycles.
\end{lemma}

\begin{proof}
Since $T_{t}(0)=0$ and $|\lambda_{0}|= |T_{t}'(0)| =|it\sec^{2} (0)| =t<1$, $0$ is an attracting fixed point.
The asymptotic values of $T_{t}$ are $-t$ and $t$. One of them, say $t$, is attracted to $0$; that is, $T_{t}^{n} (t)\to 0$ as $n\to \infty$. Since $T_{t}$ is an odd function,   $T_{t}^{n}(-t)= -T_{t}^{n} (t) \to 0$ as $n\to \infty$ so that 
  both $t$ and $-t$ are attracted to $0$ under iterations of $T_{t}$.
Therefore $C=\{0\}$ is the only attracting cycle for $T_{t}$, and its period is $1$.
\end{proof}

\medskip
\begin{lemma}~\label{Lem2}
For $t=1$, $0$ is a parabolic fixed point and there are no other attracting or parabolic cycles.
\end{lemma}

\begin{proof}
 Because $T_{1}(0)=0$ and $ \lambda_{0}=T_{1}'(0) = i\sec^{2} (0) = i$, $0$ is a parabolic fixed point. The asymptotic values of $T_{1}$ are $1$ and $-1$.
As in  the proof of Lemma~\ref{Lem1}, both $-1$ and $1$ are attracted to $0$ under iterations of $T_{1}$ so  $0$ is the only parabolic cycle for $T_{1}$ and there are no attractive cycles. 
\end{proof}

As $t$ increases through $1$, we see a period  {\em  quadrupling phenomenon in $\mathcal T_t$}:  as the  single period $1$ cycle becomes repelling, a new  period $4$ cycle is born.
\medskip
\begin{lemma}
 As $t$ increases through $1$, $0$ becomes a repelling fixed point of $T_{t}$ and  a period $4$ attracting cycle forms near $0$. This attracting cycle persists for all $t \in (1,\pi/2)$. 
\end{lemma}

\begin{proof}
 Since  $t>1$ and $|\lambda_{0}|=|T_{t}'(0)| =t$, $0$ is a repelling fixed point of $T_{t}$.
Now we consider  the  function $f_{t}^{2} =T_{t}^{4}$.  It is an odd strictly increasing function on 
$(-\pi/2, \pi/2)$ with maximal value $f_{t}^2((\pi/2)^{-})=t\tanh (t\tan t))$ and minimal value $f_{t}^2((-\pi/2)^{+})=-t\tanh (t\tan t))$. 
Since $t \in (0,\pi/2)$, $f_{t}^2((\pi/2)^{-}) \in (0,\pi/2)$. Since $0$ is a repelling fixed point, by elementary calculus 
we see that $f_{t}^2$ has a fixed point $p_{2,t}\in (0, \pi/2)$. By  symmetry, it has also a fixed point 
$p_{4,t}= -p_{2,t} \in (-\pi/2, 0)$. Since $S(f_{t}^2) <0$, we see  that both these  fixed points have to be attracting and   that they are the only attracting fixed points in $(0, \pi/2)$ and in $(-\pi/2, 0)$, respectively. 

The point $p_{2,t}$ attracts $t$ and the point $p_{4,t}$ attracts $-t$. Since both asymptotic values are attracted by these fixed points of $f_{t}^2$, there are none available to be attracted to any other cycle so there are no other attracting or parabolic cycles.

Let $x=T_{t}^{2}(p_{2,t}) \in (-\pi/2, 0)$. Since $T_{t}^{4} (p_{2,t})=p_{2,t}$, $T^{4}_{t} (x) =T^{4}_{t} (T_{t}^{2} (p_{2,t}))= T^{2}_{t} (T^{4}_{t}(p_{2,t}))= T_{t}^{2}(p_{2,t}))=x$. 
Thus $x$ is a fixed point of $f_{t}^2$ in $(-\pi/2, 0)$. This implies that $x=p_{4,t}$. Hence the set $\{ p_{2,t}, p_{4,t}\}$ is an attracting period $2$ cycle of $f_{t}$ that attracts both asymptotic values, $\pm t$. 
Let $p_{3,t}=T_t(p_{2,t})\in \Im^{+}$ and let $p_{1,t}=T_t(p_{4,t})\in \Im^{-}$. Thus we see that 
$$
C_{4,t}=\{ p_{1,t}, p_{2,t}, p_{3,t}, p_{4,t}\}
$$
is a period $4$ cycle of $T_{t}$ which attracts both asymptotic values $\pm t$.  Therefore   $T_t(x)$ has no other attracting or parabolic cycles.
\end{proof}

 We denote the multiplier of  the cycle $C_{4,t}$ by 
$$
\lambda_{4, t}= (T_{t}^{4})' (p_{1,t}).
$$

\section{Period Splitting: One period four  to two period two}
\label{sec:period splitting}

As $t$ increases past $\pi/2$ we see a {\em period splitting} phenomenon in ${\mathcal T}$; that is, the attracting period four cycle becomes   two  attracting period two cycles (see Figure~\ref{fig1} and Figure~\ref{fig10}). 
 
Suppose $t=\pi/2$;  then   the extended function $f_{\pi/2}(x)$ satisfies 
$$
f_{\frac{\pi}{2}}(\pi/2^{-})=-\frac{\pi}{2} \quad \hbox{  and  } \quad  f_{\frac{\pi}{2}}(-\pi/2^{+})=\frac{\pi}{2}
$$ 
so that  as $|x|$ approaches $\pi/2$ from below, 
$\{ -\pi/2, \pi/2\}$ is a period 2 cycle of $f_{\pi/2}(x)$.  On the other hand 
$$
f_{\frac{\pi}{2}}(\pi/2^{+})=\frac{\pi}{2}\quad \hbox{and}\quad f_{\frac{\pi}{2}}(-\pi/2^{-})=-\frac{\pi}{2}
$$ 
so that, as $|x|$  approaches $\pi/2$ from above,   
$-\pi/2$ and $\pi/2$ are  both  fixed points of $f_{\pi/2}(x)$. 

We call the pair $\{-\pi/2, \pi/2\}$ (or the points $-\pi/2$ and $\pi/2$)  a {\em virtual cycle} of period $2$  (or  {\em virtual fixed points}) for $f_{\pi/2}$;  we call the parameter $\pi/2$ a {\em virtual cycle parameter}.  Since  the  limit of the multiplier from either side satisfies $(f_{\pi/2}^{2})'(\pm \pi/2^{\pm})=0$, we also call $\pi/2$ a {\em virtual center}.   
 
 The names virtual cycle and virtual center  are justified because, like the super attracting cycles at the {\em centers} of the hyperbolic components of the Mandelbrot set for quadratic polynomials, where the critical value belongs to the cycle and the multiplier is zero, the asymptotic value belongs to the virtual cycle and the limit multiplier is zero.  The virtual cycle, however, is not really a cycle because  the asymptotic value is only the ``image'' of infinity under $T_t$ in a limiting sense;  it is   a {\em virtual image}.  In the same sense, the asymptotic value is the {\em virtual image} of the pole under $f_t$.  See \cite{CK,FK,KK1} for a more detailed discussion of virtual centers and virtual cycle parameters;  in particular, it is proved in \cite{KK1} that for the tangent family, every virtual cycle parameter  is a  virtual center and vice versa.   

We now want to see what properties of $f_t$ are also properties of $T_{t}$  as  $t$  approaches  the virtual center $\pi/2$ from either  direction.

When $t<\pi/2$, $T_{t} (t) \in \Im^{+}$ and $T_{t}(-t)\in \Im^{-}$. Thus we have
 $$ 
 \lim_{t \to (\frac{\pi}{2})^-} T_{t} (t) = i\infty  \mbox{  and  }  \lim_{t \to (\frac{\pi}{2})^- } T_{t}(-t) = -i\infty.
 $$
 Hence  the set
 $$
 C_{4,(\pi/2)^{-}}=\{-\pi/2, -i\infty, \pi/2, i\infty\}
 $$ 
 is  the limiting cycle of $C_{4,t}$ as $t\to (\pi/2)^{-}$.  It is a period $4$  cycle whose (limit) mulitplier $\lambda_{4,(\pi/2)^{-}}=\lim_{t\to (\pi/2)^{-}}\lambda_{4,t}$ is easily seen to be $0$.
Continuing with our notation above, we  call the limiting cycle $C_{4,(\pi/2)^{-}}$ a {\em virtual  cycle}, this time with period $4$, and the points in the set  {\em virtual periodic points}.   
 
 When $t>\pi/2$,  $T_{t} (t) \in \Im^-$ and $T_{t}(-t)\in \Im^+$. Thus we have 
$$ 
 \lim_{t \to (\frac{\pi}{2})^+} T_{t} (t) = -i\infty  \mbox{  and  }  \lim_{t \to (\frac{\pi}{2})^+} T_{t}(-t) = i\infty.
 $$
The two sets 
$$
C_{2,(\frac{\pi}{2})^{+}}=\{ \frac{\pi}{2}, -i\infty\}\quad \hbox{and} \quad C_{2,(\frac{\pi}{2})^{+}}'=\{-\frac{\pi}{2}, i\infty\}
$$ 
are  virtual  period $2$ cycles whose respective   (limit) multipliers are $0$.

These    virtual  cycles  become actual period $2$ cycles for  $t>\pi/2$.  More precisely, 
 there are  two fixed points of $f_{t}$, $p_{2,t}\in (\pi/2, t)$ and $p_{2,t}'=-p_{2,t}\in (-t, -\pi/2)$.  (see Figure~\ref{fig3}).

\begin{figure}[ht]
\centering
    \includegraphics[width=4in]{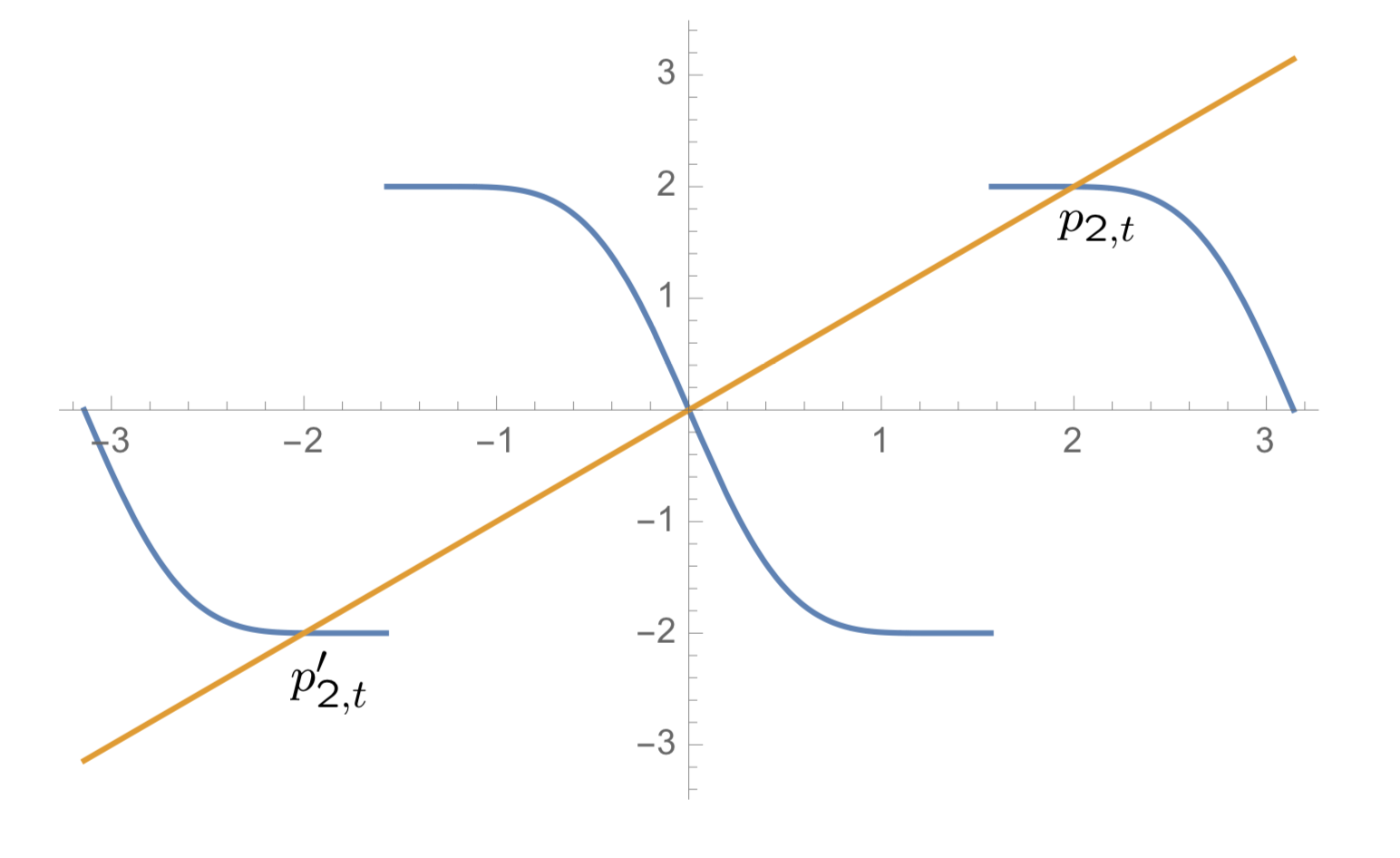}
    \caption{For  $t>\pi/2$, $f_{t}$ has two fixed points $p_{2,t}$ and $p_{2,t}'$.}~\label{fig3}
\end{figure}

Let  $p_{1,t}=T_{t}(p_{2,t})\in \Im^-$ and $p_{1,t}'=T_{t}(p_{2,t}') \in \Im^+$.  Then 
we have two period $2$ cycles for $T_{t}$, 
$$
C_{2,t}=\{ p_{1,t}, p_{2,t}\}\quad \hbox{and}\quad C_{2,t}'=\{ p_{1,t}', p_{2,t}'\}=-C_{2,t}
$$

Using the prime notation for both the derivative and the symmetric periodic point,  the multipliers of $C_{2,t}$ and $C_{2,t}'$ are
$$
\lambda_{2,t} = (T_{t}^{2})' (p_{1,t}) \quad \hbox{and}\quad \lambda_{2,t}' = (T_{t}^{2})' (p_{1,t}')
$$
  By Propsition~\ref{monomult}, for $t$ greater than, but close to $\pi/2$, 
$ \lambda_{2,t}=\lambda_{2,t}' \in (-1,0)$ and is monotonic in some interval to the right of $\pi/2$, $(\pi/2,\alpha_1)$.  Since  $\lim_{t \to \pi/2^+} \lambda_{2,t} =0$,  $\lambda_{2, \alpha_{1}}=-1$ and $C_{2,\alpha_{1}}$ and $C_{2,\alpha_{1}}'$ are both parabolic cycles. 
Now $C_{2,(\pi/2)^{+}}$ and $C_{2,(\pi/2)^{+}}'$ are the limiting cycles of $C_{2,t}$ and $C_{2,t}'$, respectively. 
Moreover, since  the multiplier  function $\lambda_{4,t}$  of the cycle $C_{4,t}$ for $t <\pi/2$ has limit $0$ at $\pi/2^-$,  it can be extended continuously for $t > \pi/2$  by setting $\lambda_{4, t} = \lambda_{2,t}\lambda_{2,t}'$. It thus becomes  a continuous function taking values from $1$ at $t=1$ to $0$ at $t=\pi/2$ and then back to $1$ at  $\alpha_{1}$ (see also~\cite{K,KK1,KK2}). 

This proves 

\medskip
\begin{lemma}~\label{alpha1}
There exists   $\alpha_{1} \in (\pi/2,\pi)$ such that when $t \in (\pi/2,\alpha_{1})$,   $0$ is a repelling fixed point of $T_{t}$ and $T_{t}$ has two period $2$ attracting cycles, $C_{2,t}$ and $C_{2,t}'=-C_{2,t}$. 
At $\alpha_{1}$, $C_{2,\alpha_{1}}$ and $C_{2,\alpha_{1}}'$ are period $2$ parabolic cycles both of whose multipliers are $-1$.   For all $t \in (\pi/2, \alpha_1]$, the cycle $C_{2,t}$ attracts the asymptotic value $t$ and the cycle $C_{2,t}'$ attracts the asymptotic value $-t$;     $T_{t}$ has no other attracting or parabolic cycles.  
\end{lemma}

\section{Period Doubling and Renormalization}\label{doubling-ren}

In this section, we will see that as $t$ increases through $\alpha_{1}$, $T_{t}$ undergoes 
a standard period doubling phenomenon.  Because the multipliers of the parabolic cycles of $T_{\alpha_1}$ (parabolic fixed points of of $f_{\alpha_1}$) are $-1$ and the map $f_{t}$ has  negative Schwarzian derivative, 
 both period $2$ 
attracting cycles become  period $2$ repelling cycles and, at the same time, two new period $4$ attracting cycles with positive multiplier are born.
    Thus, ``the period is doubled'' and as $t$ increases, it moves into a new hyperbolic component, where, by Proposition~\ref{monomult}, the multipliers of the new doubled cycles decrease monotonically to $0$.  The next lemma describes what happens at the right endpoint $\beta_1$ of this interval where the multipliers become zero.   In particular, the discussion shows that $\alpha_1$ is the only point in the interval $(\pi/2, \beta_1)$ where there are parabolic fixed points of $f_t$. This  is the first step of a renormalization process.

\medskip
\begin{lemma}~\label{beta1}
 There exists  $\beta_{1} \in (\alpha_{1},\pi)$ such that for $t \in (\alpha_{1},\beta_{1})$,  
$0$ remains a repelling fixed point of $T_{t}$,  the period $2$ cycles 
$C_{2,t}$ and $C_{2,t}'$ persist, but  they are now  repelling and $T_{t}$ has two new period $4$ attracting cycles: 
$$ C_{4,t}=\{ p_{4,1,t}, p_{4,2,t}, p_{4,3,t}, p_{4,4,t}\}$$
with  $p_{4,4,t}<p_{4,2,t} \in (\pi/2, \pi), p_{4,3,t}>p_{4,1,t}\in  \Im^-$; and 
$$ C_{4,t}'=-C_{4,t}=\{ p_{4,1,t}', p_{4,2,t}', p_{4,3,t}', p_{4,4,t}'\} $$
with $p_{4,2,t}'<p_{4,4,t}'\in (-\pi, -\pi/2)$, $p_{4,3,t}'<p_{4,1,t}'\in  \Im^+$ (see Figure~\ref{fig4}). 
The map $T_{t}$ has no other attracting or parabolic periodic cycles. The multipliers $\lambda_{4,t}$ and $\lambda_{4,t}'$ 
of these new period $4$ cycles are equal, real, positive and   decrease monotonically from $1$ to $0$ as $t$ increases from $\alpha_{1}$ to $\beta_{1}$. 
\end{lemma}

\begin{figure}[ht]
\centering
    \includegraphics[width=2in]{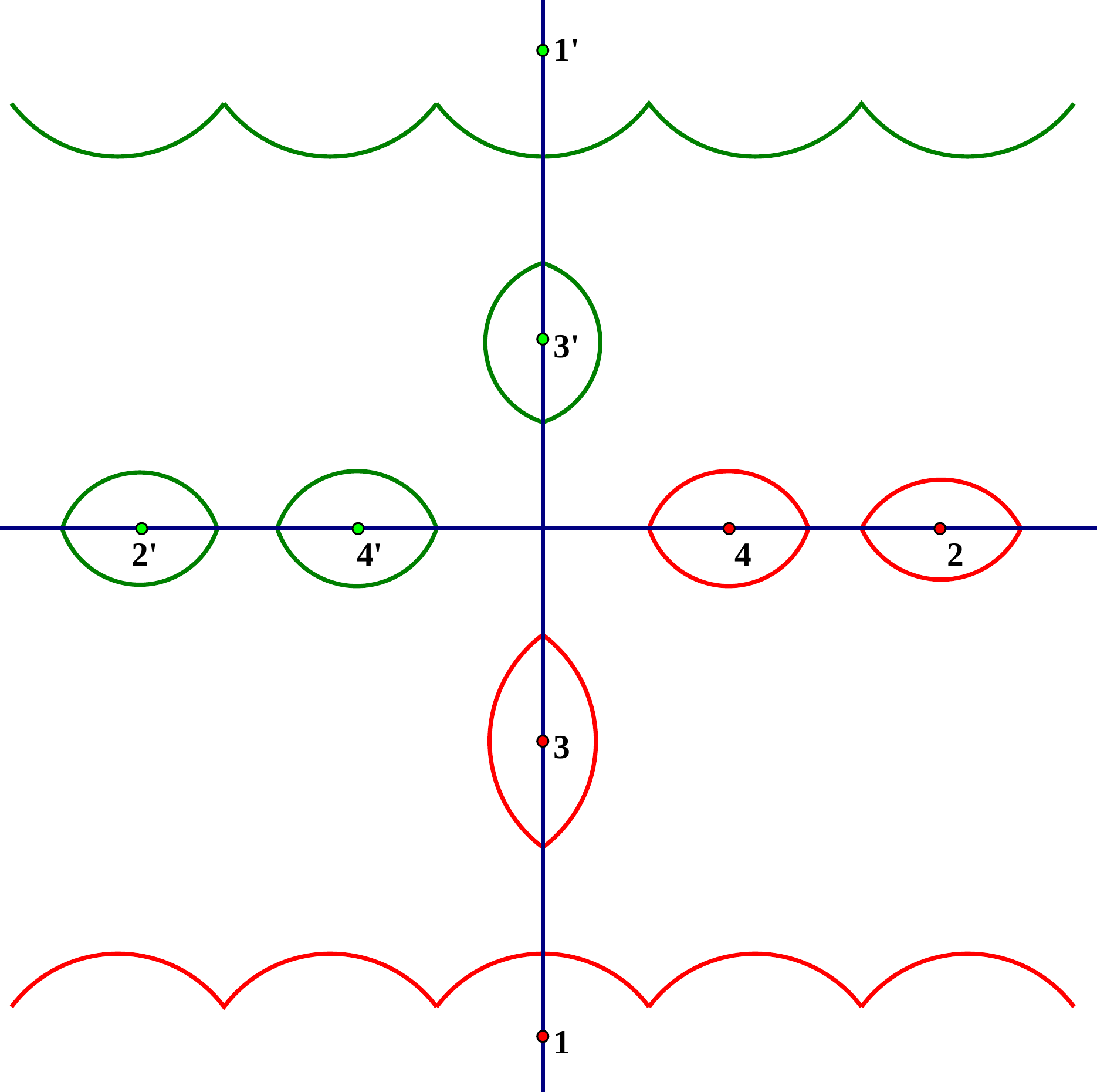}
    \caption{Two period $4$ attracting cycles $C_{4,t}$ (red) and $C_{4,t}'$ (green). }~\label{fig4}
\end{figure}

\begin{proof}
For $t \in (\pi/2, \pi]$, let $a_{t}\in (0, \pi/2)$ and $b_{t}\in (\pi/2, \pi)$  be   pre-poles such that $f_{t}(a_{t})=-\pi/2$ and $f_{t}(b_{t})=\pi/2$. Then $-a_{t}\in  (-\pi/2, 0)$ and $-b_{t}\in (-\pi, -\pi/2)$ are also pre-poles with $f_{t} (-a_{t})=\pi/2$ and $f_{t}(-b_{t})=-\pi/2$. Moreover, by symmetry and periodicity, $a_{t}+b_{t}=\pi$.   Let $I_{t} =[-b_{t}, -a_{t}]\cup [a_{t}, b_{t}]$ and consider $f^{2}_{t}|_{I_{t}}$. It is  strictly increasing  and continuous on the following set of intervals (see Figure~\ref{fig:firstrenorm})
$$
f_{t}^2: [-b_{t},-\pi/2]\to [-t, t\tanh (t\tan t)];
$$
$$
f_{t}^2: [-\pi/2,-a_{t}]\to [-t\tanh (t\tan t), t];
$$
$$
f_{t}^2: [a_{t}, \pi/2]\to [-t, t\tanh (t\tan t)];
$$
and
$$
f_{t}^2: [\pi/2, b_{t}]\to [-t\tanh (t\tan t), t].
$$
\begin{figure}[ht]
\centering
    \includegraphics[width=4in]{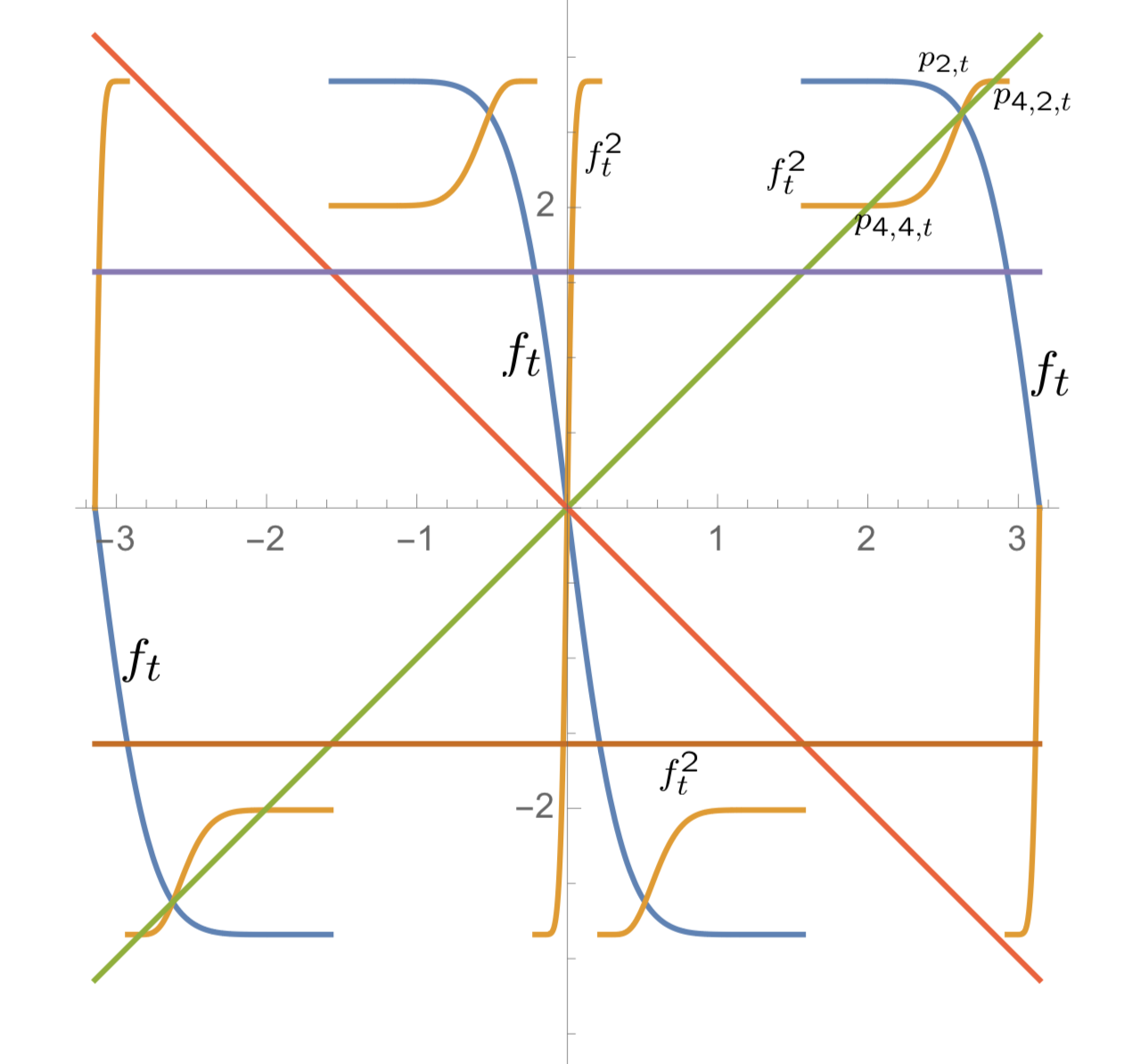}
    \caption{For  $t = 2.84 >\alpha_1$, $f_{t}$ has repelling fixed points $p_{2,t}$ and $p_{2,t}'=-p_{2,t}$ and $f_t^2$ has attracting fixed points $\{p_{4,4,t}, p_{4,2,t}\}$ and $\{-p_{2,4,t}, -p_{4,4,t}\}$. $f_t$ is blue and $f_{t}^2$ is blue.}~\label{fig:firstrenorm}
\end{figure}
 
When $t=\alpha_{1}$,  $p_{2,\alpha_1}$ and  $p_{2,\alpha_1}'=-p_{2,\alpha_1}$  are  parabolic fixed points of $f_{t}$   and both have multiplier $-1$.
Since the Schwarzian derivative of $f_{t}$ is negative, we see that for $t>\alpha_{1}$, but close to it, both of these fixed points are repelling  and near each an attracting period $2$ cycle appears.   The new cycles are $\{p_{4,2,t},p_{4,4,t}\}$ and  $\{p_{4,2,t}', p_{4,4,t}'\}$ and they are arranged as follows: 
$$-\pi< p_{4,2,t}'< p_{2,t}'<p_{4,4,t}' <-\pi/2 <0<\pi/2 <p_{4,4,t}<p_{2,t}<p_{4,2,t}<\pi.$$
Now  $p_{4,3,t}= T_{t}(p_{4,2,t})$ lies above $p_{4,1,t}= T_{t}(p_{4,4,t})$ in  $\Im^-$ and 
\begin{equation}~\label{two4}
C_{4,t}=\{ p_{4,1,t}, p_{4,2,t}, p_{4,3,t}, p_{4,4,t}\}
\end{equation}
is a period $4$ attracting cycle of $T_t$. 
Similarly $p_{4,3,t}'= T_{t}(p_{4,2,t}')$ lies below $p_{4,1,t}'= T_{t}(p_{4,4,t}')$ in  $\Im^+$ and 
\begin{equation}~\label{two4'}
C_{4,t}'=\{ p_{4,1,t}', p_{4,2,t}', p_{4,3,t}', p_{4,4,t}'\}
\end{equation}
is another period $4$ attracting cycle of $T_t$.  

  We now want to show that as $t$ increases,  it reaches a parameter $\beta_1$ where the cycles $C_{4,t}$ and $C_{4,t}'$ become  virtual cycles of period $4$;  that is,  as $t$ tends to $\beta_1$ from below,  the asymptotic values  tend  to pre-poles of $T_t$, the limit cycles contain  poles   and  the multipliers of the cycles have limit $0$.  To find $\beta_1$, 
 we  consider the continuous function $$c_{2}(f_t)= - t\tanh (t\tan t)= f_{t}^2((\pi/2)^{+}),$$ defined on in the interval $(\pi/2,\pi]$.  Note that $c_2(f_t)$ is the image of the asymptotic value $t$ under $f_t$.  In Section~\ref{sec:period splitting} we saw that  as a limit, the asymptotic value is the virtual image  of the pole $\pi/2$ under $f_t$; that is, $f_t(\pi/2^+)=t$.  
  Thus, for example, we can think of $c_2(f_t)=f_t^3(b_{1,t}^-)$ as the  ``virtual image  of a pre-pole"  of $f_t$.   
 
We claim that if $ t \in (\pi/2,\alpha_{1})$,  then  $c_{2}(f_t)>\pi/2$.  To see this is so, note that  for $t$ in this interval,  
the fixed point  $p_{2,t}$ of $f_t$ (and  $f_{t}^2) $  is attracting and greater than $\pi/2$. 
 Thus, $0<(f^2_{t})' (x)<1$ for $\pi/2 < x < p_{2,t}$ so that by the mean value theorem 
      $$p_{2,t}-c_{2}(f_t)= f_{t}^2(p_{2,t}) - f_{t}^2 (\pi/2) < p_{2,t}-\pi/2$$  and therefore $c_{2}(f_t)>\pi/2$.  

 Notice that   $c_{2}(f_\pi)=0$. 
Thus, by the intermediate value theorem,  there must be a number $\beta_1$ in $(\alpha_{1},\pi)$ satisfying  $c_{2} (f_{\beta_{1}})=\pi/2$;  that is, at $t=\beta_1$, the image of the asymptotic value is a pole and is therefore a point in a virtual cycle. 
A priori, there may be more than one solution, but by the transversality of Theorem~\ref{transversality}, $c_2$ is strictly monotonic at $\beta_1$ so the solution is locally unique.    

 Both $\pi/2$ and $-\pi/2$  are fixed points of $f^2_{\beta_{1}}$ and the limit of the multiplier at each of these fixed points is $0$.  Thus
$$
C_{4,\beta_{1}}=\{ \beta_{1}, p_{4,3,\beta_{1}},  \pi/2,  -i\infty\}
$$
and 
$$
C_{4,\beta_{1}}'=\{ -\beta_{1}, p_{4,3,\beta_{1}}',  -\pi/2,  i\infty\}
$$
are  virtual cycles of  period $4$.  By symmetry the multipliers of the cycles $C_{4,t}$ and $C_{4,t}'$ are equal.  By Proposition~\ref{monomult}  they decrease monotonically   from $1$ to $0$ as $t$ increases from $\alpha_{1}$ to $\beta_{1}$.     
In this interval, each of the cycles $C_{4,t}$ and $C_{4,t}'$  attracts one asymptotic value so $T_{t}$ has no other attracting or parabolic periodic cycles.  This  proves the lemma.  
\end{proof}

This lemma gives us the existence of $\beta_{1}$. The uniqueness of $\beta_{1}$ in the interval $(\alpha_{1}, \alpha_{2})$ follows from 
Corollary~\ref{uniqueness}. 

 \subsection{The first renormalization}
  We can now define the first renormalization of the function $f_t$ and to do so we  introduce two auxiliary functions. 
   We  set $c_1(f_t)=f_t((\pi/2)^{+})=t$, the positive asymptotic   value and as above, we set $c_2(f_t)=f_t(t) = f_t^2((\pi/2)^{+})$,  the image of the asymptotic value under $f_t$.  It is also  the virtual image of the pole of $f_{t}^2$ and thus the positive asymptotic value of $f_t^2$.  
 
\medskip

 Let $I_{1,t} =[-b_{1,t}, -a_{1,t}]\cup [a_{1,t}, b_{1,t}]=I_{1,t}^{-} \cup I_{1,t}^{+}$  denote the pair of intervals in Lemma~\ref{beta1}.  
 We introduce the index $1$ for future reference. 
 
 \medskip
 
 \begin{definition}~\label{1ren}
 We say that $f_{t}$ (or $T_{t}$) is renormalizable if the map $f_t^2$ has a unqiue pre-image of each of the poles of $f_t$, $ \pm\pi/2$, in each of the intervals composing $I_{1,t}$.    If this is true, we  call the map 
$$
{\mathcal R}_t= f_{t}^{2}|I_{1,t} =T_{t}^{4}|I_{1,t}
$$ 
 {\em the first renormalization of $f_{t}$ (or $T_{t}$)}.   The poles of $\Rt$ are $\pm a_{1,t}, \pm b_{1,t}$ and the asymptotic values $\pm t$ of $f_t$ are their virtual images;   the limits $\Rt(\pi/2^{\pm})$ are the asymptotic values of $\Rt$.    
   \end{definition}
   By Lemma~\ref{beta1}, since $c_1(f_t)=t > \pi/2$ and  $t > \beta_1$ implies $c_2(f_t) < \pi/2$, $\Rt$ is defined  for $t$ in some interval to the right of   $\beta_1$.  (See Figure~\ref{fig5}).
  As $t$ approaches $\pi$,  $\Rt(\pi/2^{-})=\Rt(\pi/2^{+})=0$ so we see that  $\Rt$ is defined for all $t \in (\beta_1, \pi)$.   
  
   \begin{figure}[ht]
\centering
    \includegraphics[width=4in]{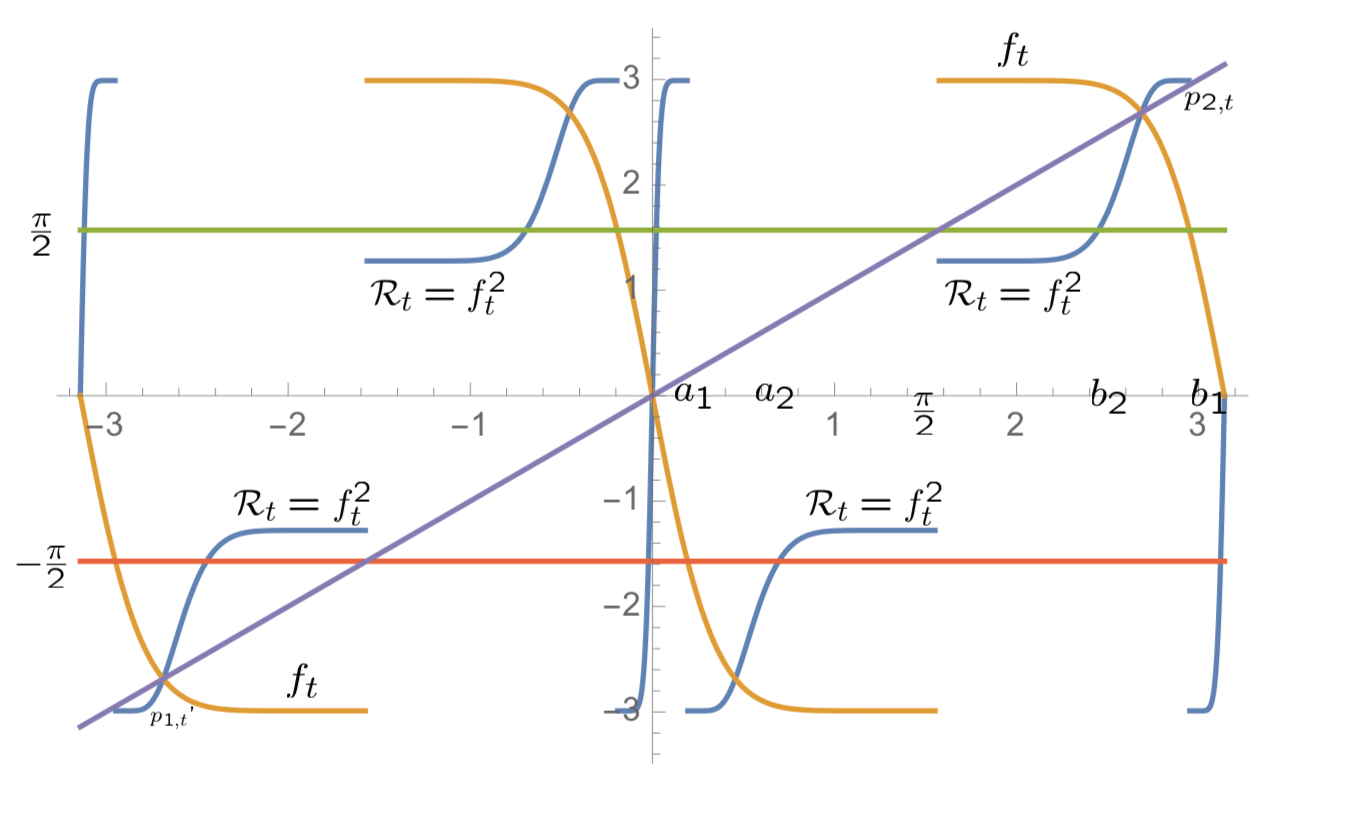}
    \caption{The First Renormalization: \\$t=2.99 > \beta_1 \approx 2.94$. 
  ($a_1=a_{1,t} , b_1=b_{1,t})$, $f_t$ is yellow and $f_t^2$ is blue.}~\label{fig5}
\end{figure}

   This is the first step of a renormalization process for tangent maps that will be defined in \S\ref{inf ren}.   The endpoints of the intervals of $I_{1,t}$ are pre-poles of $f_t$ and hence poles of $\Rt$;  each is divided into two by a pole of $f_t$.  As we saw above, $\Rt=f_t^2$ maps each subinterval of  $I_{1,t}$  continuously onto either $[c_2(f_t), c_1(f_t)]$ or its negative. 
    Since $c_1(f_t)=t > \pi/2$ and  $c_2(f_t) < \pi/2$, these image intervals each contain a pole of $f_t$, ($\pm \pi/2)$.    
  Otherwise said, in $I_t$, $f_t^2$ is monotonically strictly increasing  with discontinuities at $\pm \pi/2$ and, just as the renormalized quadratic map is unimodal where it is defined,  $\mathcal{R}_t$ is  ``tangent-like''  on $I_{1,t}$ (see Figure~\ref{fig5}).   To make the analogy complete, we should make an affine conjugation so that $\Rt$ is defined on $[\pi, -\pi]$ again, but to do this would make the notation even more complicated than it is. 

The following remark is important for  the discussion in next three sections of the period merging phenomenon.  

\medskip
\begin{remark}~\label{misp1}
In the family of quadratic polynomials, there is a notion of a {\em  full family} for a family of renormalizations (see~\cite{JBook}).  Roughly speaking,  this means that each renormalization is defined for an interval of parameters and these intervals nest as further renormalizations are made.  This is also so for renormalizations of the tangent family,  but because we don't make the affine conjugation, the  parameter intervals all have same right endpoint. 
\end{remark}

\section{The First Cycle Merging}\label{merging1}

In the quadratic family, as the polynomials pass through the center of a hyperbolic component where the critical point belongs to the attracting cycle and the multiplier is zero, the attracting cycle persists.   By contrast, in the tangent family $T_{t}$, as the parameter $t$ passes through the parameter $\beta_{1}$ described in Lemma~\ref{beta1}  where the limit multiplier is zero and the limit function has two virtual cycles of period $4$, we will see that  the two period $4$ attracting cycles merge into one period $8$ attracting cycle.   This is the first of a sequence of ``cycle merging phenomena'' that occur for this family.  

\subsection{Virtual cycles and  virtual centers}
In \S\ref{sec:period splitting} we saw that at $t=\pi/2$ the asymptotic values could be thought of as part of a virtual cycle.  Here we give the general definition of virtual cycles and virtual cycle parameters of arbitrary periods.

In this paper we always take $N$ to be of the form $2^n$ but the definitions make sense for any $N$.  

  \medskip
 \begin{definition}~\label{vcp}  
  If, for some $N \ge 2$ and $k \in \mathbb Z$, $T^{N-2}_t(t)= k\pi +\pi/2$,  then $t$ is called a {\em  virtual cycle parameter}.   If we want to emphasize the value of $N$,   we call it a {\em virtual cycle parameter  of period $N$}. For such a $t$, we call the  orbits of the asymptotic values    {\em virtual cycles}.
 \end{definition} 

Suppose $t$ is a virtual cycle parameter of period $N$.   
 Set $p_{2,t}=t$ and for $i=2,\cdots, N-2$, let $p_{i+1,t}=T^{i}_{t} (p_{2,t})$.  To define $p_{1,t}$ so that  $T_t(p_{1,t})=p_{2,t}$, we must take limits. Since 
 \[ \lim_{y \to \pm \infty} T_{t} (iy)= \mp t \]
 we take $p_{1,t} = - i \infty$.   Next since 
 \[ \lim_{x \to \pm(k\pi+ \pi/2)} T_t(x)=T_t(k\pi+ \pi/2)^{\pm} = {\mp} i \infty, \]
 if we set  $p_{N,t}=(k\pi +\pi/2)^+$, then $p_{N+1,t}=-i \infty =p_{1,t}$, and we have a virtual cycle of period $N$ containing  the asymptotic value $t$.    The other asymptotic value lies in a symmetric virtual cycle of period $N$ and we denote this pair of virtual cycles by 
  $$
 C_{N, t} =\{ p_{1,t}=-i\infty, \; p_{2,t}=t, \; \cdots,\; p_{N-1, t}, \; p_{N,t} =(k\pi+\pi/2)^{+}\}
 $$
 and 
 $$
 C_{N,t}' = \{ p_{1,t}'=i\infty,  \; p_{2,t}'=-t,  \; \cdots, \; p_{N-1,t}'=-p_{N-1,t}, \; \cdots, p_{N,t}'= (-k\pi-\pi/2)^{-}\}. 
 $$
 
If, however, we set   $p_{N,t}=(k\pi +\pi/2)^-$, we get $p_{N+1,t}=i \infty =-p_{1,t}$. Now, as a limit  $p_{N+2,t} = -t$, and  setting  $p_{N+1+i,t}=T^{i}_{t} (p_{N+1,t})$ for $i=2,\ldots, N-1$, we obtain a single virtual cycle of period $2N$ that contains the orbits of both asymptotic values.  We denote  this by 
$$
 C_{2N, t} =\{ p_{1,t}=-i\infty, \; p_{2,t}=t, \; \cdots,\; p_{N-2, t}, \; p_{N,t} =(k\pi+\pi/2)^{-}, \;  p_{N+1,t}'=i\infty,  \; 
 $$
 $$
 p_{N+2,t}'=-t,  \; \cdots, \; p_{2N-1,t},\; p_{2N,t}= (-k\pi-\pi/2)^{+}\}. 
 $$

\medskip

\begin{remark}\label{mergemult}
The multiplier of $C_{N,t}$ is given by the formula 
 \[ 
\lambda_{N,t}(C_{N,t})=(T_t^{N})'(p_{i,t})=(it)^{N}\prod_{k=1}^{N}\sec^2(p_{k,t}) =0.  
\] 
Note that for $z \in {\mathbb C}$, $\sec(z) \neq 0$.  

If we set 
\[
T_{t}' (\pm i \infty) =\lim_{y \to \pm \infty} (it) \sec^2(iy) =0,
\]   
 we can define the multiplier of the virtual cycle as a limit, 
 $ \lambda_{N,t}(C_{N,t})=0$.    Similarly, as limits,  we have  
 \[  \lambda_{N,t}(C_{2N,t})= \lambda_{N,t}(C_{N,t} )\lambda_{N,t}(C_{N,t}')=0.  
\]
\end{remark}
 
\begin{remark}
Virtual cycle parameters and virtual centers were first introduced in~\cite{K,KK1,KK2};  it was proved there that  for the family $T_t$,  every hyperbolic component has a virtual center and every virtual cycle parameter is the virtual center of two distinct hyperbolic components  tangent at the virtual  center.  This justifies calling the virtual cycle parameter a virtual center.  See~\cite{CK,FK,KK1} for  more general discussions of virtual cycle parameters and virtual centers. 
\end{remark} 

\medskip
\subsection{The first cycle merging}
Now we apply the discussion above to the period $4$ cycles in (\ref{two4}) and (\ref{two4'}). 
 For $t\in (\alpha_1, \beta_{1})$, as $t$ approaches $\beta_{1}$,  the  cycles $C_{4,t}$ and $C_{4,t}'$ approach 
cycles $C_{4,\beta_{1}}$ and $C_{4,\beta_{1}}'$ as follows:
\[ 
\lim_{t \to \beta_{1}^-} T_t(p_{4,4,t}) =-i\infty \mbox{   and  } \lim_{t \to \beta_{1}^-} T_t(p_{4,3,t}) =(\pi/2)^{+} \]
 and 
 \[ \lim_{t \to \beta_{1}^-} T_t(p_{4,4,t}') = +i\infty \mbox{   and  } \lim_{t \to \beta_{1}^-} T_t(p_{4,3,t}') = (-\pi/2)^{-}. \] 
 Thus in the limit  we have two symmetric virtual  cycles of period $4$, 
$$
C_{4,\beta_{1}}=\{ -i\infty, p_{4,2,\beta_{1}}, p_{4,3,\beta_{1}}, \frac{\pi}{2}\}
$$
and
$$
C_{4,\beta_{1}}'=\{ i\infty, p_{4,2,\beta_{1}}', p_{4,3,\beta_{1}}', -\frac{\pi}{2}\},
$$
whose multipliers are $0$. 

By Corollary~\ref{uniqueness}  and Theorem~\ref{transversality},   there is a unique  solution $\beta_{1}$  of $c_2(f_t)=\pi/2$ in $(\alpha_1,\pi)$ and  $c_{2}'(f_t)|_{\beta_1} \neq 0$.
Thus 
\[ \lim_{t \to \beta_{1}^{+}} T_t(p_{4,4,t}) =+i\infty \mbox{  and  } \lim_{t \to \beta_{1}^{+}} T_t(p_{4,3,t}) =(\pi/2)^{-} \]
 and 
 \[ \lim_{t \to \beta_{1}^+} T_t(p_{4,4,t}') = -i\infty \mbox{  and  } \lim_{t \to \beta_{1}^+} T_t(p_{4,3,t}') = (-\pi/2)^{+}. \] 
 
 We see that the continuations of the period $4$ virtual cycles  $C_{4,1,\beta_{1}^{-}}$ and $C_{4,2,\beta_{1}^{-}}$ exist as we approach $\beta_{1}$ from below and we denote them by $C_{4,1,\beta_{1}^{-}}$ and  $C_{4,2,\beta_{1}^{-}}$. Since $c_{2}'(f_t)|_{\beta_1} \neq 0$, when we take the limit as $t$ approaches  $\beta_{1}$ from above,  the same set of points are part of a single period $8$ cycle whose multiplier tends to $0$,
 $$
C_{8,\beta_{1}^{+}}=\{ p_{8,1,\beta_{1}^{+}}=-i\infty,\;  p_{8,2,\beta_{1}^{+}}=p_{4,2,\beta_{1}^{-}}=\beta_{1}, \; 
$$
$$
p_{8,3,\beta_{1}^{+}}=p_{4,3,\beta_{1}^{-}}, \; p_{8,4, \beta_{1}^{+}}= (\pi/2)^{-},
p_{8,5,\beta_{1}^{+}} =i\infty, \; 
$$
$$
p_{8,6,\beta_{1}^{+}}= -\beta_{1},\;  p_{8,7,\beta_{1}^{+}}= p_{4,3,\beta_{1}^{-}}', \; p_{8,8, \beta_{1}^{+}}= (-\pi/2)^{+} \}.
$$
(See Figure~\ref{fig6}). 

 The  cycle $C_{8,\beta_{1}^{+}}$ persists as $t$ increases beyond $\beta_{1}$. 
Thus by Proposition~\ref{monomult} we have

\medskip
 \begin{lemma}\label{1,8case}
  There exists  $\alpha_{2}$ in $(\beta_{1}, \pi)$ such that for $t \in (\beta_{1},\alpha_{2})$, 
  $0$ is a repelling fixed point and $C_{2,1,t}$ and $C_{2,2,t}$ are repelling period 2 cycles of $T_t$. 
In addition, $T_{t}$ has one attracting period $8$ cycle:
$$
C_{8,t}=\{ p_{8,1,t} , \; p_{8,2,t}, \; p_{8,3,t}, \; p_{8,4,t}, 
p_{8,5,t}, \; p_{8,6,t}, \; p_{8,7,t}, \; p_{8,8,t}\},
$$
where 
$$
-\pi< p_{8, 6,t}<-\pi/2<p_{8,8,t} <0 <p_{8, 4,t}<\pi/2<p_{8,2,t} <\pi
$$
and 
$$
-i\infty < p_{8,1,t} < p_{8,3,t} <0< p_{8,7,t} < p_{8,5,t}<i\infty.
$$ 
  The map $T_{t}$ has no other attracting or parabolic cycles.
\end{lemma}

\begin{figure}[ht]
\centering
    \includegraphics[width=3in]{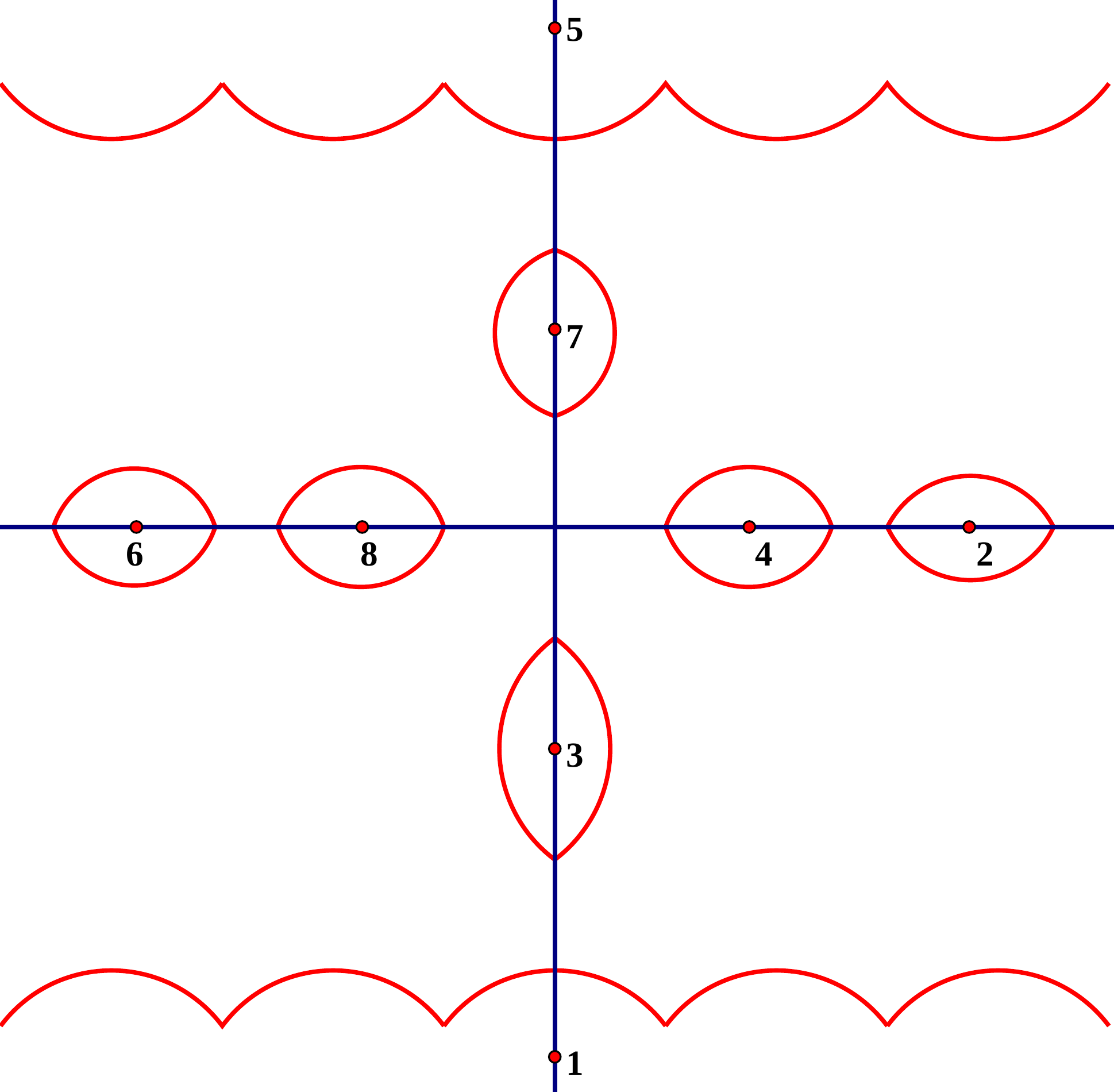}
  \caption{Period merging at $\beta_{1}$ from two period $4$ cycles to one period $8$ cycle.}~\label{fig6}
\end{figure}

 It follows from  Remark~\ref{mergemult}, that  the multipliers of the cycles $C_{4,t}$ and $C_{4,t}'$ tend to $0$ as $t$ tends to $\beta_{1}$  from either side.  The multiplier of $C_{8,t}$ is the product of the formulas for the derivatives at the points in $C_{4,t}$ and $C_{4,t}'$ so it has a limit of $0$ at $\beta_{1}$.  This implies  that $C_{8,t}$ is an attracting cycle attracting both asymptotic values  
and  $T_{t}$ can have no 
other attracting or parabolic cycles. Beyond $\beta_{1}$ the multiplier is monotone strictly increasing so 
this attracting property persists until $t$ reaches some $\alpha_2$ where $\lambda_{8,\alpha_2}=1$.  
This lemma gives us the existence of $\alpha_{2}$. The uniqueness of $\alpha_{2}$ follows from 
Corollary~\ref{uniquealpha}.

\section{The First Cycle Doubling}\label{standard doubling}\label{std doubling}

 In this section we will prove that the single period $8$ parabolic cycle for $t=\alpha_2$ ``doubles'' into two period $8$ attracting cycles as $t$ increases past $\alpha_{2}$. This phenomenon is somewhat different from the period doubling we observed for $t=\alpha_1$ because the multiplier of the parabolic cycle in that case was $-1$ and in this case it is $+1$.  It therefore needs to be described differently and hence we give it a different name, {\em  the cycle doubling phenomenon.}  We will see that for $ t \in (\beta_{1}, t_{\infty})$, where $t_{\infty}< \pi$ is to be defined, there is no more period doubling, but there is a sequence of  ``cycle doublings'' that starts with $\alpha_2$.  

The following lemma says that when $t>\beta_{1}$, the multiplier of any parabolic cycle is $+1$ and therefore, $T_{t}$ does not undergo a standard period doubling. 

\medskip
\begin{lemma}~\label{pos}
If $t>1$  the multiplier of any period $4n$ attracting or parabolic cycle of $T_t$ is positive.
\end{lemma}

\begin{proof} Since $T_t$ maps the real line to the imaginary line and vice-versa, and since the two asymptotic values are real, it follows  that any attracting or parabolic cycle is contained in the union of the real  and   imaginary axes. Suppose the cycle is $p_1, p_2, \cdots, p_{4n}$. Without of loss generality, 
we may assume $p_1$ is real. Then $p_{2k-1}$  is  real and $p_{2k}$ is pure imaginary for all $k=1,\cdots, 2n$. 
 Now the multiplier of the cycle is  
\[ (
T_t^{4n})'(p_1)=(it)^{4n}\prod_{k=1}^{2n}\sec^2(p_{2k-1})\sech^2(i p_{2k}) >0.  
\]
\end{proof}

The next lemma is important for our proof of Part d) in Theorem~\ref{bif}.

\medskip
\begin{lemma}~\label{inftyfs}
Suppose $t>1$ and suppose $T_{t}$ has an attracting or parabolic cycle. 
Then there exists an $r_t >0$ such that the intervals $(-i\infty, -r_t i)$ and $(r_t i, i\infty)$  are in the intersection of the immediate basin of this cycle with $\Im-$ and $\Im^+$,  respectively.  
\end{lemma}
 
\begin{proof}  Since $t>1$, $0$ is a repelling fixed point.   If $T_{t}$ has an attracting or parabolic cycle,  then either it is a symmetric cycle and both its asymptotic values $\pm t$  lie the immediate basin of the cycle or there are two symmetric cycles and one asymptotic value is in the immediate basin of each.   Assume for arguments sake there are two symmetric cycles; the argument in the other case will be clear.   Since $t$ is in the immediate basin, by Lemma~\ref{interval}, there is $s>0$ such that the interval $(s,t)$ is in the immediate basin of the cycle and the periodic point of the cycle is either inside the interval or is the point $s$.    Now 
the preimage of $[s, t)$ under $T_{t}$ is  $(-i\infty, -r_t i]\subset \Im^{-}$ for some $r_t>0$ and it contains a point of the cycle.   The same argument   says that the interval 
  $(-t, -s]$ and its preimage $(r_t i, i\infty)\subset \Im^{+}$ under $T_{t}$ are both in the immediate basin of the symmetric cycle.   Therefore, all four of  these intervals belong to the intersection of the immediate basin of the cycle with ${\mathbb R}^{-}$, $\Im^{-}$, ${\mathbb R}^{+}$, and $\Im^{+}$. 
   \end{proof}

We also need the following general result from complex dynamics about  parabolic cycles.  The proof uses standard techniques so we defer it to  the Appendix. 
\medskip
\begin{lemma}~\label{parabolicbur}
Suppose $f(z)=z+a_nz^n+o(z^n)$ is an analytic function defined on some neighborhood  of $0 \in \mathbb C$. 
\begin{enumerate}
\item   Suppose  $\lambda$ lies  inside a small disk, inside and tangent to the unit circle at the point $1$.   
Then $g_{\lambda} (z)= \lambda f(z)$ has one attracting fixed point $0$ and $(n-1)$ repelling fixed points  counted with multiplicity,  in a small neighborhood of $0$.
\item  Suppose  $\lambda$ lies  inside a small disk, outside and tangent to the unit circle at the point $1$. Then $g_{\lambda}(z)=\lambda f(z)$ 
has one repelling fixed point $0$ and $(n-1)$ attracting fixed points counted with multiplicity, in a small neighborhood of $0$.
\end{enumerate}
\end{lemma}
The next lemma provides the first instance of cycle doubling. 
\medskip
\begin{lemma}~\label{alpha2}
There exists  $\beta_{2}$  in  $(\alpha_{2}, \pi)$, 
such that for $t \in (\alpha_{2},\beta_{2})$,  
$0$ remains a repelling fixed point of $T_{t}$ and $C_{2,t}$, $C_{2,t}'$ remain  period $2$ repelling cycles of $T_{t}$,   the merged 
     period $8$ parabolic cycle $C_{8,t}$ at $\alpha_{2}$ becomes a repelling cycle and  a new pair of period $8$ attracting cycles are born.
\end{lemma} 

\begin{proof}
We consider the period $8$ parabolic cycle $C_{8, \alpha_{2}}$.  
As the limit of cycles for $t$ in the interval $(\beta_{1}, \alpha_{2})$, 
it attracts both asymptotic values $\pm \alpha_{2}$.  Since its multiplier is $1$, 
we  cannot prove the lemma by  the standard period doubling argument that we in used in \S\ref{sec:period splitting}.  Instead,
we will show that there are exactly two attracting petals at each point in the cycle $C_{8, \alpha_{2}}$. Then we can apply Lemma~\ref{parabolicbur} and Corollary~\ref{uniquealpha} to complete the proof.
This is equivalent to showing that each point in the cycle is the common boundary point of  two distinct components in its  immediate basin of attraction.  By Lemma~\ref{inftyfs}, it will suffice to show there are a pair of intervals that belong to distinct components of the  immediate basin  meeting at each point in the cycle.   Recall that because the immediate basin must contain the forward orbit of at least one asymptotic value and that these orbits lie in the real and imaginary axes, there are at most two distinct components at each point. 

For readability, we drop the index $8$ in the notation for the periodic points of the cycle. Then, following our convention (see Figure~\ref{fig6}),  $p_{1, \alpha_{2}}$ denotes the lowest point on  $\Im^-$
 and $p_{5, \alpha_{2}}$ 
  the highest point on $\Im^+$  in the cycle $C_{8, \alpha_{2}}$.
Since $T_{\alpha_{2}} (\Im^{-})= (0,\alpha_{2})$,  we have $T_{\alpha_{2}}((-i\infty, p_{1, \alpha_{2}}))=(p_{2, \alpha_{2}},\alpha_{2})$.  
The interval $(p_{2, \alpha_{2}},\alpha_{2})$ is therefore contained  
  in the intersection of an attracting petal lying to the right of $p_{2, \alpha_{2}}$ with $\Re^+$.  Pulling this petal back to $p_{1, \alpha_{2}}$ by $T_t$, we have a petal containing 
$(-i\infty, p_{1, \alpha_{2}})$.
By symmetry we obtain an attracting petal at $p_{5, \alpha_{2}}$ containing $(p_{5, \alpha_{2}}, i\infty)$ and one whose intersection with $\Re^-$ is an interval  to the left of $p_{6, \alpha_{2}}$ containing  $(-\alpha_2,p_{6, \alpha_{2}})$.  

Since $f_{t}$ is a monotonic strictly decreasing piecewise continuous function on the real line,  $f_{t}^2$ is a  monotonic strictly increasing piecewise continuous function.  
 Therefore if we apply $f_t$ to an interval in its   region of continuity  an even number of times, the image has the same orientation as the original.   
This implies that when we apply $f_{t}^{2}=T_{t}^{4}$ to the petal lying to the right of $p_{2,\alpha_{2}}$, we get an attracting petal at $p_{6, \alpha_{2}}$ which contains an interval to its right in $\Re^-$.  Since $p_{6, \alpha_{2}}$ is a parabolic periodic point, it is not in the Fatou set, so the intervals on each side of it are in  distinct petals.  Since there are at most two attracting petals 
 at each point  of $C_{8, \alpha_{2}}$ there are exactly two.    
  We can thus use the  local coordinate at each point,  
\[
T_{t}^{8}(z)=z+a_{3} z^3+o(z^3), \quad  a_{3}\neq 0.
\]
and apply  Lemma~\ref{parabolicbur} and Corollary~\ref{uniquealpha} to deduce that   when $t>\alpha_{2}$, the  cycle $C_{8, t}$ is 
  repelling  and there are two distinct  
period $8$ attracting cycles $C_{8, t,1}$ and $C_{8, t,2}$ near this period $8$ repelling cycle.
By Proposition~\ref{monomult}, these attracting cycles persist through some interval $(\alpha_2,\beta_2)$ in which, as $t$ increases, their multipliers decrease from $1$ to $0$;  at  $\beta_{2}$ the multiplier is $0$.   As $t$ approaches $\beta_2$ from below 
 the point $p_{1,t}$ increases and  limits at $i\infty$ and the point  $p_{5, t}=-p_{1,t}$ decreases and limits 
at $-i\infty$. 
\end{proof} 
  We will discuss $\beta_2$ further in the next section.
  
\section{The Second Period Merging and Renormalization}\label{merging2}

The second period merging  phenomenon is somewhat different from the first one  so we show how it works in this section. 
In the previous section, we saw that as $t$ increases through $\alpha_{2}$,  two  new period $8$ attracting cycles form. These persist until $t$ reaches a value $\beta_{2}$ where they become virtual cycles.  
It follows from the discussion of renormalization and Theorem~\ref{transversality}  that  cycle merging occurs at  $\beta_{2}$.

 Let us recall some of the notation we used for the first renormalization.  
We started with the 
  full family  of tangent maps,  
  \[  \{ f_{t}: [-\pi, \pi] \to [-t, t]\}_{t=\pi/2}^{\pi} \] 
  and obtained the full family of tangent-like maps 
   \[  \{ {\mathcal R}_t=f^{2}_{t}: I_{1,t}\to [-t, t \tanh (t\tan t)]\cup [-t\tanh(t\tan t), t]\}_{t=\beta_{1}}^{\pi} \} \]  
  where 
    \[ I_{1,t} = [-b_{1,t}, -\pi/2]\cup[-\pi/2, -a_{1,t}] \cup [a_{1,t}, \pi/2]\cup[\pi/2, b_{1,t}]  \]
    and $\pm a_{1,t}, \pm b_{1,t}$ are pre-poles of $f_t$ of order 2 closest to $\pm\pi/2$.  
Then we introduced the notation 
 \[ c_{1} (f_t) = f_t( (\pi/2)^+)=t  \mbox{ and }  c_{2}(f_t) = - f_{t}^2 ((\pi/2)^+)= -t\tanh (t\tan t) >0. \]
 Note that $c_1(f_t)$ is the asymptotic value of $f_t$ and $c_1(\Rt)=\Rt(\pi/2^+)=c_2(f_t)$ is the asymptotic value of $\Rt$.

\subsection{The second renormalization}
For each $t \in (\beta_{1}, \pi)$,  $c_2(f_t)< \pi/2$.  This implies  there is a unique solution $b_{2,t}$ of ${\mathcal R}_t(z)=\pi/2$ in $[\pi/2, b_{1,t}]$ and another solution $-a_{2,t}$ in $[-\pi/2,-a_{1,t}]$.  Using the symmetry about $0$ we find solutions $-b_{2,t}$ and $a_{2,t}$ of   ${\mathcal R}_t(z) =-\pi/2$.  By periodicity we see that $a_{2,t}+b_{2,t}=\pi$.  We now apply Definition~\ref{1ren} to obtain the renormalization $\Rt^2$ of $\Rt$. 

The points $\pm a_{2,t}, \pm b_{2,t}$  are poles of $\Rt^2=f_t^4$ and are points of discontinuity;  they divide the intervals of $I_{1,t}$ into subintervals where $\Rt^2$ is continuous.  We consider those that have the pole $\pm \pi/2$ as an endpoint and, taking advantage of the symmetry, we label them    as follows:
 \[ I_{2,t} =[-b_{2,t}, -\pi/2]\cup [-\pi/2, -a_{2,t}] \cup [a_{2,t}, \pi/2]\cup [\pi/2, b_{2,t}]. \]
 
 Below, for the sake of readability we suppress the dependence on $t$ and relabel the sub-intervals  of $I_{1,t}$ and $I_{2,t}$ as
  \[ I_{1,t} = I_{11-}\cup I_{12-}\cup I_{12+}\cup I_{11+}. \] 
 \[ I_{2,t}  = I_{21-}\cup I_{22-}\cup I_{22+}\cup I_{21+}.\] 
   
 We now set $c_{2}(\Rt)=\Rt^2(\pi/2^+)=|f_t^4(\pi/2^+)|$.  It denotes the positive asymptotic value of $\Rt^2$.  Because $\Rt^2$ is monotonic strictly increasing, $c_{2}(\Rt)>c_{1}(\Rt)$. 
 With this notation, the second renormalization ${\mathcal R}^{2}_t$ (see Figure~\ref{fig7}) is:
$$
\Rt^2=f_{t}^{4}: I_{22-} \to  [c_{1}(\Rt),c_{2}(\Rt)]
$$ 
$$
\Rt^2=f_{t}^{4}:  I_{21-}  \to [-c_{2}(\Rt), -c_{1}(\Rt)],
$$ 
$$
\Rt^2=f_{t}^{4}: I_{21+} \to[c_{1}(\Rt),c_{2}(\Rt)],
$$ 
and 
$$
\Rt^2=f_{t}^{4}: I_{22+} \to [-c_{1}(\Rt),-c_{2}(\Rt)].
$$

 \begin{figure}[ht]
\centering
    \includegraphics[width=4in]{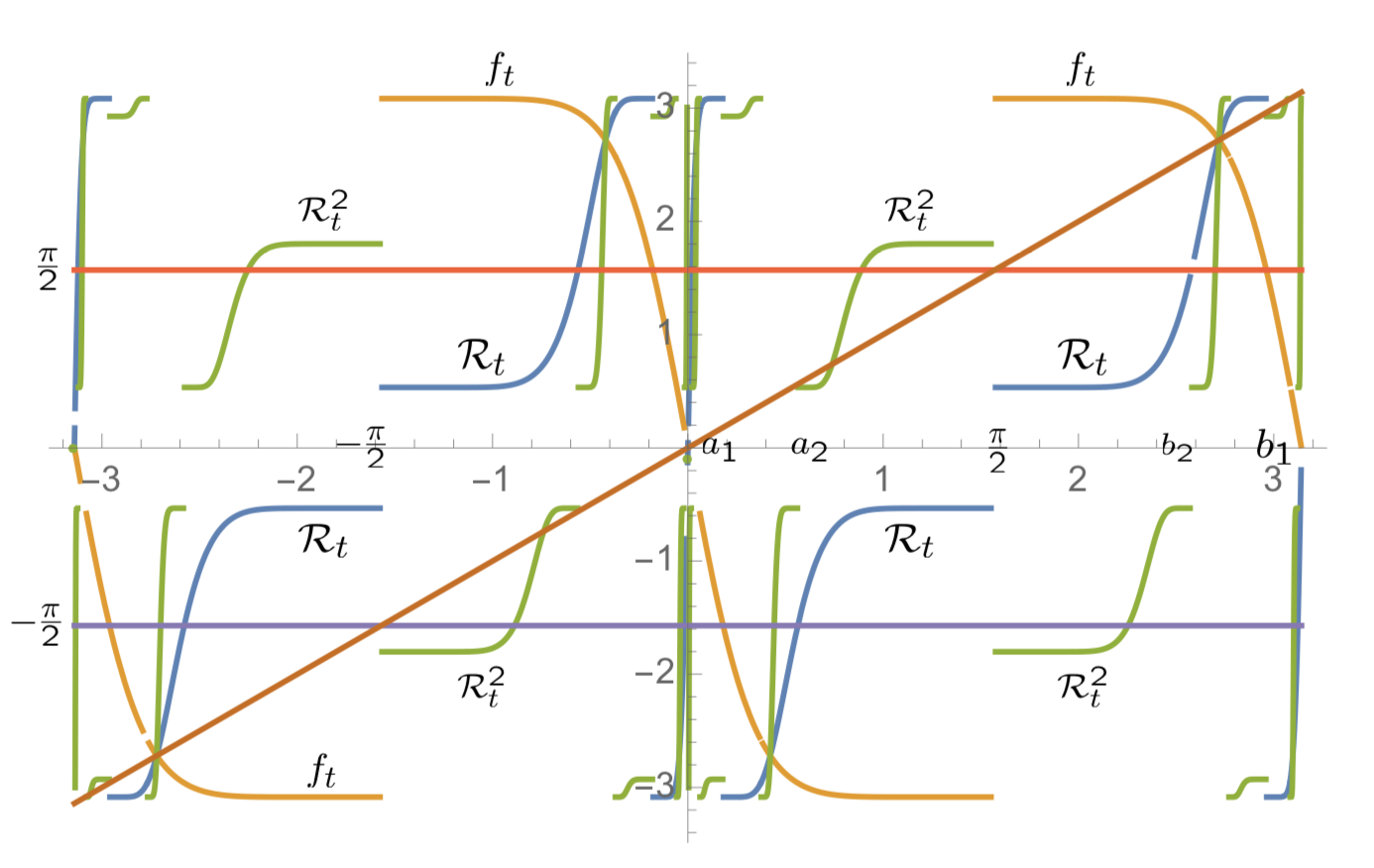}
    \caption{Second Renormalization for $t=3.085$. $f_t$ is yellow, $\R_t$ is blue and $\R_{t}^2$ is green. }~\label{fig7}  
    \end{figure}

  Arguing as we did in the proof of Lemma~\ref{beta1},  we see that  for all $t$ to the right of $\beta_{1}$,   $c_1(\Rt) < \pi/2$.  Also, 
   $c_1(\Rt)$ is continuous and decreases monotonically to $0$  so by the intermediate value theorem  there is a parameter $\beta_{2}$ in the interval $(\alpha_{2}, \pi)$  such that $c_{1} (\R(f_{t})) =a_{1, \beta_{2}}$, where $a_{1, t}$ is an endpoint of  one of the  intervals on which $\R(f_{t})$ is defined. At this $\beta_{2}$, we have $c_{2}(\R(f_{\beta_{2}}))=|\pi/2|$;  that is, the asymptotic value of $\R_{\beta_2}$ is a pole.   Again, as we saw in the proof of Lemma~\ref{beta1}, for $t \in (\beta_1, \beta_2)$ we have $c_{2}(\Rt) < \pi/2$. 
Moreover, by Theorem~\ref{transversality}, for $t$ in some interval to the right of $\beta_2$, $c_{2}(\Rt) > \pi/2$. 

Now we examine the behavior of the  periodic cycles (considered as cycles of $T_t$) to see how this merging occurs. 
 As  $t$ approaches   $\beta_{2}$, the asymptotic values become pre-poles and $\beta_{2}$ is a virtual cycle parameter.  When $t$ approaches $\beta_{2}$ from below,  the period $8$ attracting cycles become symmetric virtual cycles of the same period; 
 (see Figure~\ref{fig8}), 
$$
C_{8, \beta_{2}^{-}} =\Big\{  i\infty,\; -\beta_{2}, \; T_{\beta_{2}} (-\beta_{2}), \; \cdots,\; T^{5}_{\beta_{2}} (-\beta_{2}), \; \pi/2 \Big\}
$$
and 
$$
C_{8, \beta_{2}^{-}}' =\Big\{ -i\infty, \; \beta_{2}, \;T_{\beta_{2}} (\beta_{2}),\;  \cdots, \;T^{5}_{\beta_{2}} (\beta_{2}),\; -\pi/2 \Big\}.
$$
$$
C_{8, \beta_{2}^{-}} =\Big\{  i\infty,\; -\beta_{2}, \; T_{\beta_{2}} (-\beta_{2}), \; \cdots,\; T^{5}_{\beta_{2}} (-\beta_{2}), \; \pi/2. \Big\}
$$
Thus we see that $\beta_{2}$ is a parameter that satisfies Lemma~\ref{alpha2}.  

It follows from Corollary~\ref{uniqueness} that taking limits as $t$ approaches $\beta_{2} $ from above,  we have 
\[ \lim_{t\to \beta_{2}^+} T_{t}^{7} (-t) = -i\infty \mbox{  and  }  \lim_{t\to \beta_{2}^+} T_{t}^{7} (t) = i\infty. \]
 Thus, as a  limit from above,  $T^8(\beta_{2})=-\beta_{2}$.  This implies that at $t=\beta_{2}^+$, the two period $8$ virtual cycles $C_{8, \beta_{2}^{-}}$ 
and $C_{8, \beta_{2}^{-}}'$ merge into one symmetric period $16$ virtual cycle (see Figure~\ref{fig9}), 
$$
C_{16, \beta_{2}^+} =\Big\{   i\infty,\; -\beta_{2}, \; T_{\beta_{2}} (-\beta_{2}), \; \cdots,\; T^{5}_{\beta_{2}} (-\beta_{2}), \; \pi/2,\;  -i\infty, 
$$
$$
\beta_{2}, \;T_{\beta_{2}} (\beta_{2}),\;  \cdots, \;T^{5}_{\beta_{2}} (\beta_{2}),\; -\pi/2\Big\}.
$$
Taking the limit from the right, the multiplier of  the cycle $C_{16, \beta_{2}^+}$ is $0$.

\begin{figure}[ht]
\centering
    \includegraphics[width=3in]{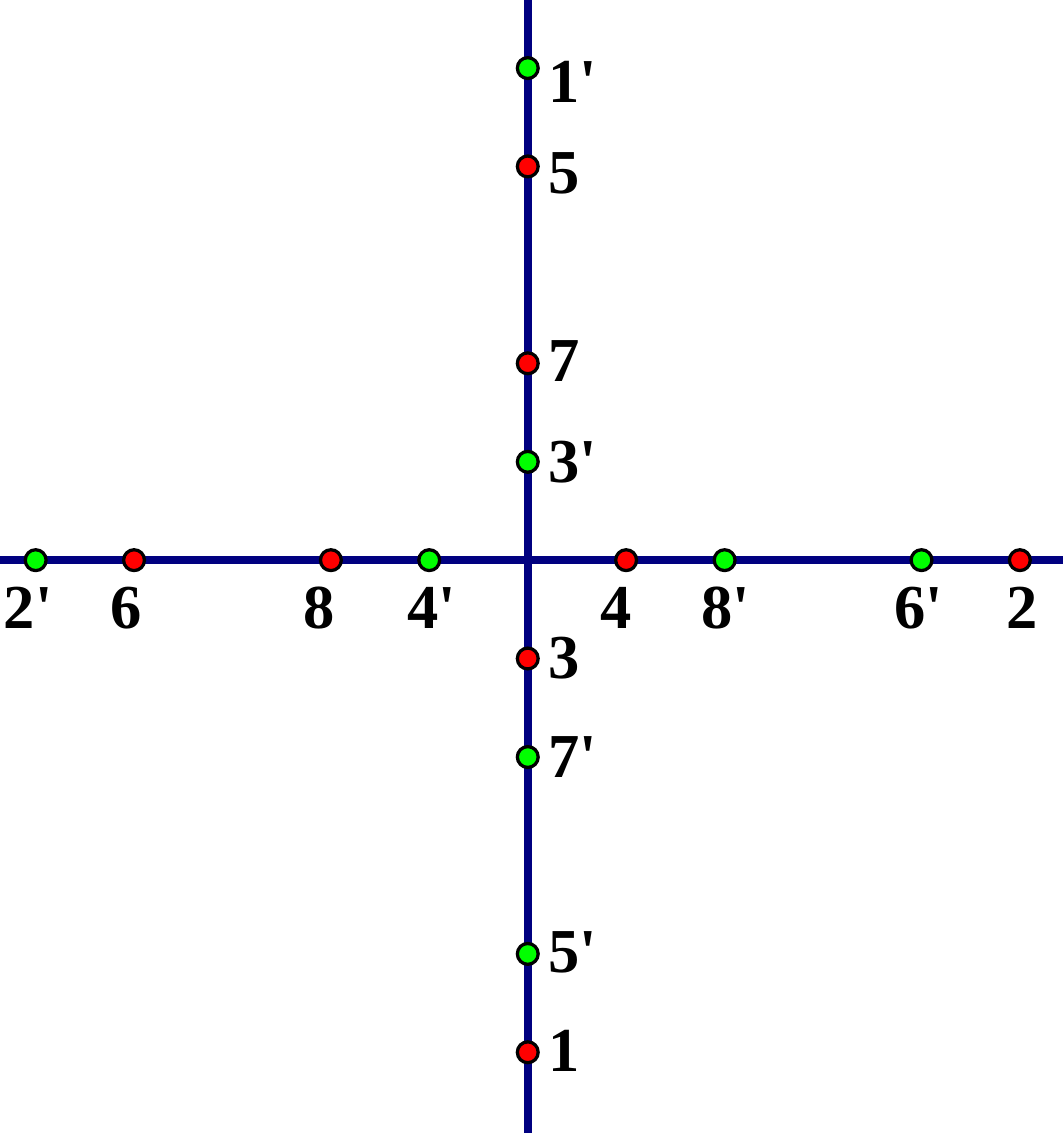}
    \caption{Two period $8$ virtual cycles}~\label{fig8}
  \end{figure}
  
Applying Proposition~\ref{monomult}, we see that as $t$ increases  beyond $\beta_{2}$ there is an interval $(\beta_{2},\alpha_3)$ in which the   cycle $C_{16, \beta_{2}^+}$ is an attracting cycle 
$C_{16, t}$whose multiplier varies from $0$ to $1$.   Thus $\beta_2$ is uniquely defined in the interval $(\alpha_1, \alpha_2)$.

\begin{figure}[ht]
\centering
    \includegraphics[width=3in]{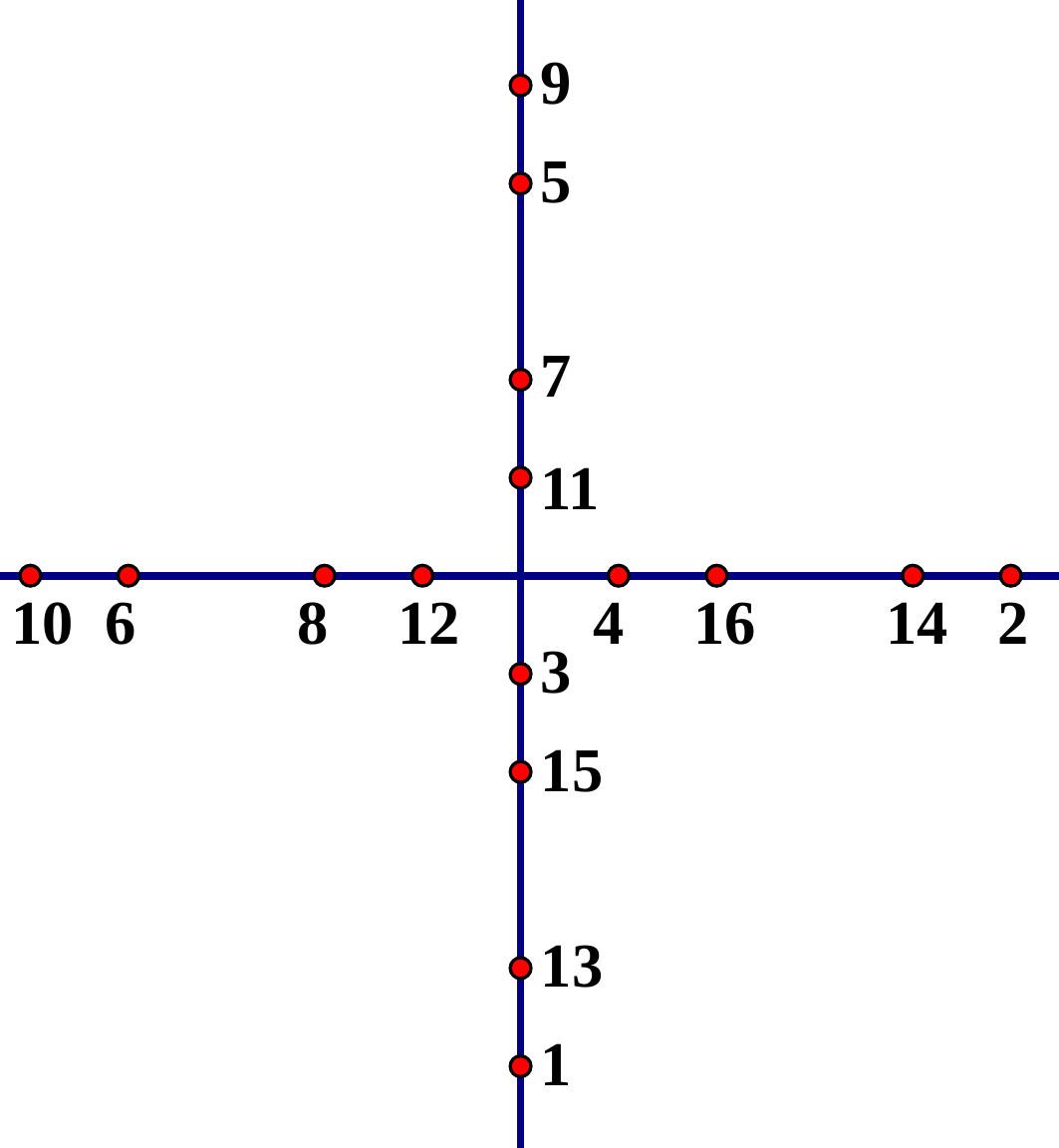}
    \caption{One period $16$ attracting cycle}~\label{fig9}
  \end{figure}

This discussion constitutes a proof of the following lemma asserting the existence of $\alpha_{3}$. The uniqueness of $\alpha_{3}$ follows from 
Corollary~\ref{uniquealpha}.

\medskip
\begin{lemma}\label{1-16case}
There exists  an $\alpha_{3}$  in  $(\beta_{2}, \pi)$ such that for $t \in (\beta_{2},\alpha_{3})$, 
 $0$ is a repelling fixed point and the  pair of  virtual cycles $C_{8, \beta_{2}^{-}}$ and $C_{8, \beta_{2}^{-}}'$  of $T_{\beta_{2}}$  have merged into one attracting cycle 
$$
C_{16,t} =\{ p_{16,1,t} , \; p_{16,2,t}, \cdots,\; p_{16,9,t}, \; p_{16,10,t}, \;\cdots, \; p_{16,16,t}\}.
$$
 (See Figure~\ref{fig9}).  Following our convention, $p_{16,1,t}$ is the lowest point of the cycle in $\Im^{-}$, $p_{16,2,t}$ is the largest in ${\mathbb R}^{+}$,  $p_{16,9,t}$ is the highest in $\Im^{+}$, and $p_{16,10,t}$ is the smallest in ${\mathbb R}^{-}$. 
 The map $T_{t}$ has no other attracting or parabolic cycles.
\end{lemma}

As we saw for $\Rt$, as $t$ approaches $\pi$, $\Rt^2(\pi/2^{-})=\Rt^2(\pi/2^{+})=0$ so we see that  $\Rt$ is defined for all $t \in (\beta_1, \pi)$.

\section{General Pattern of Cycle Doubling and Merging.}\label{gen pattern}

 Now that we have seen the period quadrupling, period splitting, and period doubling phenomena for low periods and analyzed the first examples of cycle doubling and cycle merging,   we are ready to state and prove our first  main result, Theorem~\ref{bif}, which gives the  general pattern of cycle doubling and cycle merging.   To make the statement comprehensive, we include the  period doubling case, Part a) which occurs for $n=1$ and was proved in Lemma~\ref{beta1}.  
 
\medskip
\begin{theorem}~\label{bif}
There are two sequences interleaved of parameters $\{\alpha_{n}\}_{n=1}^{\infty}$ and  $\{\beta_{n}\}_{n=1}^{\infty}$ for the tangent family 
$$
\alpha_{1}< \beta_{1}<\alpha_{2}< \beta_{2} <\alpha_{3} <\cdots
$$
$$
 < \beta_{n}<\alpha_{n}<\cdots  <\pi, 
$$
and they correspond to the following bifurcation phenomena:
\begin{itemize}
\item[\rm{a})] {\bf (Period Doubling):} At  $t=\alpha_1$, $T_t$ has $2$ parabolic cycles of period $2$.  For $t \in (\alpha_{1},\beta_{1})$, the analytic continuation of the cycles are repelling and $T_t$ has two new symmetric attracting cycles of period $4$; it has no other attracting or parabolic cycles. 
\item[\rm{b})] {\bf (Virtual Periodic Cycles):}  
     For $n \geq 1$,  the parameters $\beta_{n}$ are virtual cycle parameters.    
\begin{itemize}
\item[\rm{(i)}]
As $t \to \beta_{n}^{-}$, the  two symmetric period $2^{n+1}$ attracting cycles limit onto virtual cycles of the same period:
$$
C_{2^{n+1}, \beta_{n}^{-}} =  
 \Big\{ - i\infty,\; \beta_{n}, \; T_{\beta_{n}} (\beta_{n}), \; \cdots,\; T^{2^{n+1}-2}_{\beta_{n}} (\beta_{n}), \; (-1)^{n+1}\frac{\pi}{2}  \Big\}
$$
and 
$$
C_{2^{n+1}, \beta_{n}^{-}}' = 
\Big\{ i\infty, \; -\beta_{n}, \;T_{\beta_{n}} (-\beta_{n}),\;  \cdots, \;T^{2^{n+1}-2}_{\beta_{n}} (-\beta_{n}),\; (-1)^{n}\frac{\pi}{2} \Big\}.
$$
\item [\rm{(ii)}] As $t \to \beta_{n}^{+}$, a single symmetric  attracting cycle of period $2^{n+2}$ limits onto a virtual cycle of the same period:
$$
C_{2^{n+2}, \beta_{n}^{+}} =\Big\{ -i\infty, \; \beta_{n}, \; T_{\beta_{n}} (\beta_{n}), \; \cdots,\; 
T^{2^{n+1}-1}_{\beta_{n}} (\beta_{n}), \; (-1)^{n+1} \frac{\pi}{2},\; i\infty,\;
$$
$$
\,  \, \quad  -\beta_{n}, \;T_{\beta_{n}} (-\beta_{n}),\;  \cdots, \;T^{2^{n+1}-3}_{\beta_{n}} (-\beta_{n}), \; (-1)^{n}\frac{\pi}{2} \Big\}.
$$
\end{itemize}
\item[\rm{c})] {\bf (Cycle Merging):} For  $t \in (\beta_{n},\alpha_{n+1})$, there is a single attracting cycle of period
  $2^{n+2}$.  It is the analytic continuation of the  virtual cycle $C_{2^{n+2}, \beta_{n}^{+}}$ and is given by 
  $$
C_{2^{n+2}, t} =\Big\{ p_{1,t},\; p_{2,t}\; \cdots\; p_{2^{n+2},t} \Big\} 
$$
where $p_{1,t}$ is the lowest point of the cycle on  $\Im^{-}$, 
$p_{2,t}$ is the rightmost   on ${\mathbb R}^{+}$, $p_{2^{n+1}+1,t}$ is the highest  on $\Im^{+}$, and $p_{2^{n+1}+2, t}$ is the leftmost ${\mathbb R}^{-}$.  
This  cycle is the only attracting periodic cycle of $T_{t}$ and its multiplier  goes from $0$ to $1$ as $t$ moves from the lower to  the upper endpoint of the interval $(\beta_{n},\alpha_{n+1})$. 
\item [\rm{d})]{\bf (Parabolic Periodic Cycles):}  As $t$ approaches $\alpha_{n+1}$ from below, the limit of the attracting period $2^{n+2}$ cycle $C_{2^{n+2},t}$ is a parabolic   cycle $C_{2^{n+2}, \alpha_{n+1}}$ of the same period; its  multiplier is  $1$ and $T_t$ has no other attracting or parabolic periodic cycle.
\item [\rm{e})]{\bf (Cycle Doubling):} As $t$ moves into the interval  $(\alpha_{n+1},\beta_{n+1})$, the parabolic period $2^{n+2}$ cycle $C_{2^{n+2}, \alpha_{n+1}}$ bifurcates into two  attracting periodic cycles of the same period: that is, its analytic extension becomes repelling but two new attracting cycles of the same period, 
$C_{2^{n+2}, t}$ and $C_{2^{n+2}, t}'$ appear. The multipliers of these  new cycles are equal and go from $1$ to $0$ as $t$ increases through the interval.
 $T_{t}$ has no other attracting or parabolic periodic cycle for $t \in   (\alpha_{n+1},\beta_{n+1})$. 
  \item  [\rm{f})]{\bf (Renormalization):} For all $t \in (\beta_{n}, \pi)$, the 
   renormalizations $\Rt^n = f^{2^{n}}$ are defined on a symmetric quadruple of intervals $I_{n,t} \subset I_{n-1,t}$ bounded by pre-poles of order $n-1$ and  the poles $\pm \pi/2$ respectively.    The renormalized functions are ``tangent-like'' in that in each interval they are continuous  and strictly increasing with horizontal asymptotes.\footnote{ By abuse of notation, we also call  $\Rt^n$ a renormalization of $T_t$.}
 \end{itemize}
\end{theorem}

\begin{figure}[ht]
\centering
    \includegraphics[width=5in]{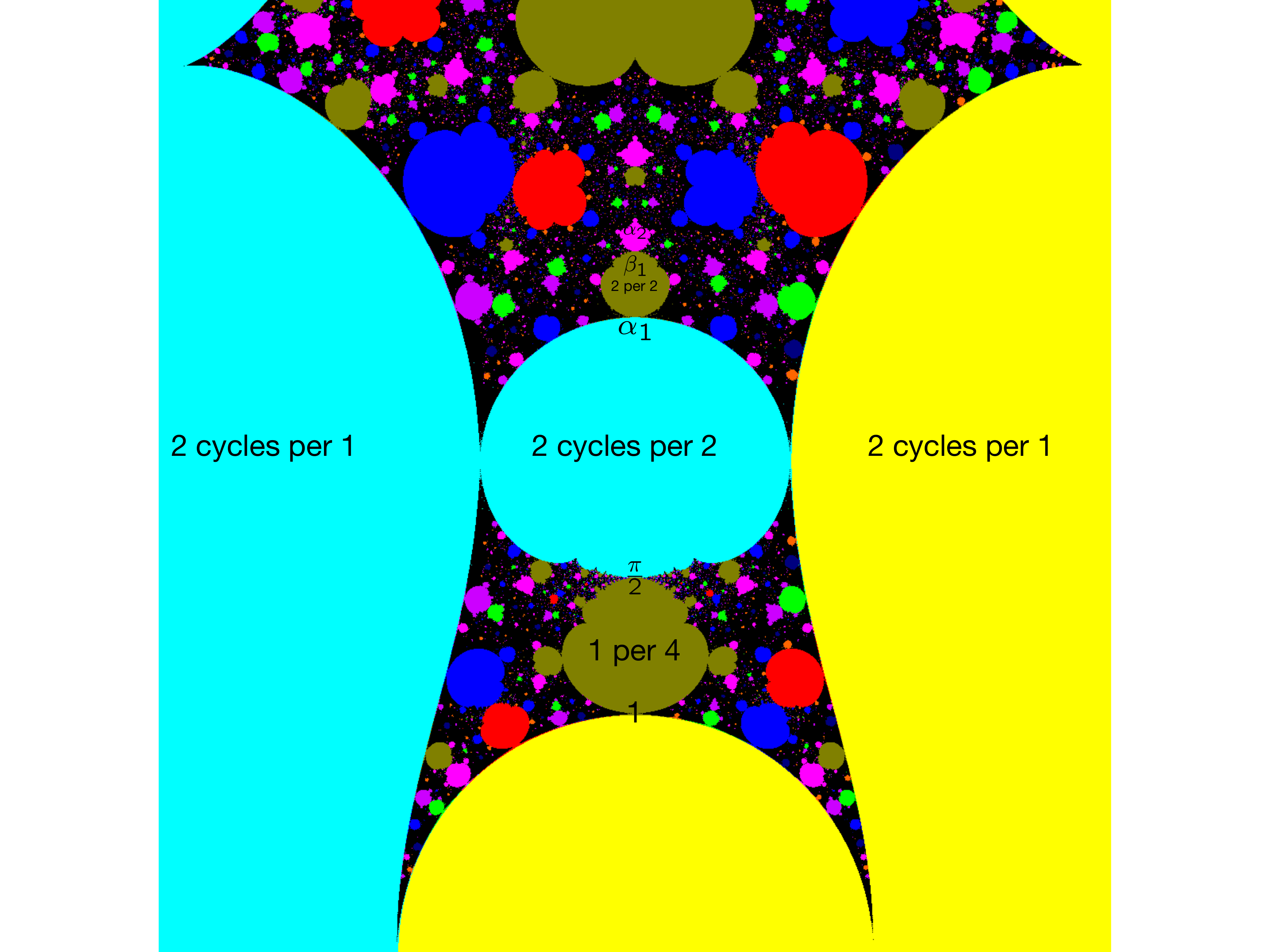}
  \caption{This is a computer generated picture illustrating Theorem~\ref{bif}. The coloring is determined by the period of the cycle attracting the asymptotic value $t$.  Thus regions with one cycle of period $N$ have the same color as regions with two cycles of period $N$.  The color coding is: 
period $1$--yellow;  period $2$--aqua;  period $3$--red;  period $4$--khaki;  \ldots,  period $8$--bright pink; \ldots .  
The range is $\{ |\Re t | < 3.15, 0< \Im t  < 3.15 \}$. 
}~\label{fig10}
\end{figure}

\begin{proof}[Proof of Theorem~\ref{bif}]
 The proof of the existence of the $\alpha_n$ and $\beta_n$ is by induction on $n$.  The uniqueness is proved in Corollaries~\ref{uniqueness} and ~\ref{uniquealpha}.
 
In \S\ref{sec:period splitting}--\S\ref{merging2} we saw that  Parts $b)-f)$ hold for $n=1,2$.  
We prove now that if the theorem holds for some $n$ then  it  holds for $n+1$. 
 We show this first for Parts $d)$ and $e)$. 

 We assume the theorem holds for $t  \leq \beta_n$ and consider $t$ in the interval $(\beta_n,\alpha_{n+1})$.  By the induction hypothesis, the virtual cycle  at $t=\beta_n^+$ has period $2^{n+2}$, is symmetric, and so attracts both asymptotic values, and  has  multiplier $0$. 
By Lemma~\ref{pos}, Corollary~\ref{uniqueness} and Proposition~\ref{monomult},  its analytic continuation is an attracting cycle of the same period whose multiplier is a positive strictly increasing function of $t$ for  $t >\beta_n$.   Hence, for some $\alpha_{n+1} >\beta_n$,  the multiplier of the cycle has reached  $1$;  thus $T_{\alpha_{n+1}}$ has a single symmetric  parabolic cycle $C_{2^{n+2}, \alpha_{n+1}}$ of   period $2^{n+2}$ whose multiplier is $1$. The uniqueness of $\alpha_{n+1}$ follows from Corollary~\ref{uniquealpha}. This shows Part $d)$ holds for $n+1$.

\begin{figure}[ht]
  \centering
  \includegraphics[width=3in]{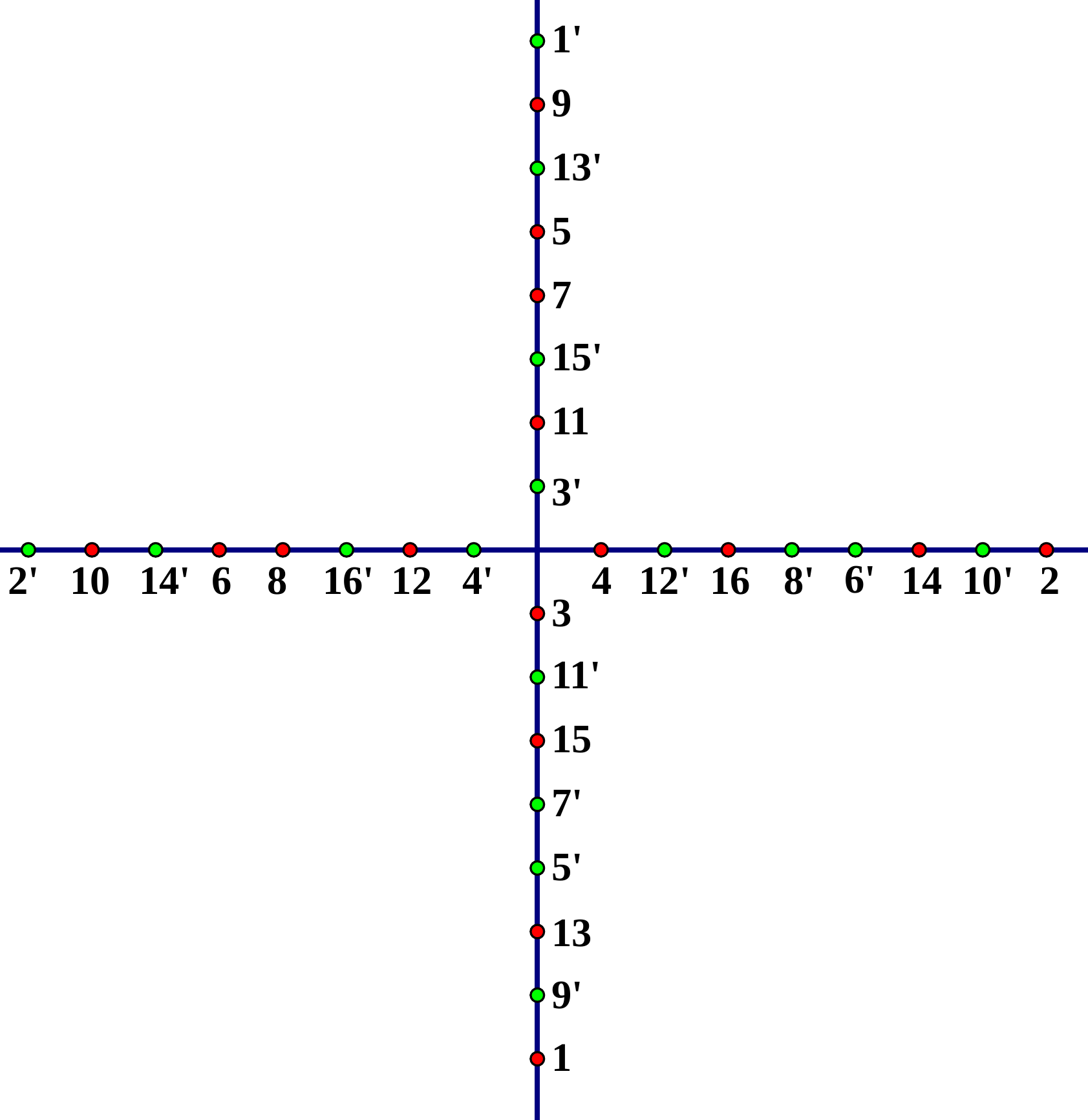}
  \caption{Two period $16$ attracting cycles}~\label{fig11}
\end{figure}

To show that Part $e)$ holds for $n+1$  we start with the parabolic cycle $C_{2^{n+2}, \alpha_{n+1}}$.  As the limit of a symmetric cycle for $t$ in the interval $(\beta_n, \alpha_{n+1})$, it attracts both asymptotic values $\pm \alpha_{n+1}$.   Since its multiplier is $+1$, we  need to use the argument of \S\ref{std doubling}. Thus, in order  to prove that as $t$ increases thorugh $\alpha_{n+1}$, the cycle becomes repelling and a new pair of period $2^{n+2}$ attracting cycles appear, we need to to show that there are  exactly two attracting  petals at each point in the   parabolic cycle $C_{2^{n+2}, \alpha_{n+1}}$;  then we can apply Lemma~\ref{parabolicbur}.   This follows directly from the argument in the proof of Lemma~\ref{alpha2},  with $2$ replaced by  $n$   so that  $\alpha_3$ becomes $\alpha_{n+1}$, $C_{8,\alpha_3}$ becomes $C_{2^{n+2}, \alpha_{n+1}}$ and so on (see Figure~\ref{fig11}).

Now we use the renormalization process to prove  that  Parts $b)$, $c)$ and $f)$ hold for $n+1$.  That is, we assume that  for all $t  \in (\beta_n, \pi)$ the $n^{th}$ renormalization  ${\mathcal R}^n$ of $f_t$ is defined and that $\beta_n$ is a virtual cycle parameter on the left boundary of an interval in which $f_t$ has a single attracting period cycle of period $2^{n+2}$ attracting both asymptotic values.   

Recall that in \S\ref{merging1} where $n=1$, the asymptotic value $\beta_1$ was the virtual image of $-i\infty$ under $T_{\beta_1}$ whereas in \S\ref{merging2} where $n=2$, the asymptotic value $\beta_2$ was the virtual image of $+i\infty$ under $T_{\beta_2}$.  This is because in the intervals where the function $f_t$ and its renormalization are both defined one is positive and one is negative.    Below, we see that  the same thing holds as we repeat the renormalization process;  $\beta_{n+1}$ is the virtual image of $(-1)^{n+1}i \infty$ under $T_{\beta_{n+1}}$.    

We used the renormalized function $\Rt$ and $\Rt^2$ and the auxilliary functions $c_1(f_t)=f_t(\pi/2^+), c_2(f_t)=f_t^2(\pi/2^+)$ and $c_1(\Rt)=\Rt(\pi/2^+), c_2(\Rt)=\Rt^2(\pi/2^{+})$   
to find the points $\beta_1$ and $\beta_2$. There are analogous functions for $\Rt^n$ which we define below.  The functions are defined as limits and the signs vary depending on the direction of the limit and the parity of $n$ so we will make some conventions about sign in the definitions below.

Define $c_m(f_t)=|f^m(\pi/2^+)|$, $m=1,2,\ldots $;  then $\pm c_m(f_t)$ are the points in the orbit of the asymptotic value $t$ of $f_t$. Inductively define
\[  c_m(\Rt) =| \Rt^m(\pi/2^{+}) |= | f_t^{2^{n}m}(\pi/2^{+} )| = c_{2^{n}m}(f_t). \]
 
 Set $\Rt^{0}=f_t$ and $a_{0,t}=b_{0,t}=\pi/2$; then   ${\mathcal R}^n(T_t)$ is defined inductively  on a set of intervals $I_{n,t}= I_{n1}^-\cup I_{n2}^-\cup I_{n2}^+ \cup I_{n1}^+$, each bounded by the pre-pole $\pm a_{n-1,t}$ or $\pm b_{n-1,t}$ of ${\mathcal R^{n-1}(T_t)}$ closest to $\pm \pi/2$,  and each containing a pre-pole $\pm a_{n,t}$ or $\pm b_{n,t}$ of ${\mathcal R^{n}(T_t)}$, as follows:
 
 \[ \Rt^n = \Rt\circ \Rt^{n-1}=f_t^{2^{n}}: I_{(n)1} ^-\to [(-1)^{n-1}c_1(\Rt^{n-1}), (-1)^{n-1}c_2(\Rt^{n-1})] \]
 \[\Rt^n = \Rt\circ \Rt^{n-1}=f_t^{2^{n}}: I_{(n)2} ^-\to [(-1)^{n}c_2(\Rt^{n-1}), (-1)^{n}c_1(\Rt^{n-1})]  \]
  \[ \Rt^n = \Rt\circ \Rt^{n-1}=f_t^{2^{n}}: I_{(n)2} ^+\to [(-1)^{n-1}c_1(\Rt^{n-1}), (-1)^{n-1}c_2(\Rt^{n-1})]  \]
\[  \Rt^n = \Rt\circ \Rt^{n-1}=f_t^{2^{n}}: I_{(n)1} ^+\to [(-1)^{n}c_2(\Rt^{n-1}), (-1)^{n}c_1(\Rt^{n-1}).]  \]
  
   Note that $c_1(\Rt^n)$ is the asymptotic value of $\Rt^n$ and the functions $\pm c_m(\Rt^n)$ define its orbits.
   When $n$ is odd    $c_{1}(\Rt^n)<c_{2}(\Rt^n)$  whereas when $n$ is even the inequalities are reversed.  The point $\beta_{n}$ is defined by solving $c_1(\Rt^{n})=a_{n,t}$ for $t$;  that is, the asymptotic value of $\Rt^{n-1}$ is a pole of $\Rt^{n-1}$ and $\beta_n$ is a virtual cycle parameter.  
   
 If there is more than one solution,   we take the largest such that  $c_{2}(\Rt)$ is continuous as our 
$\beta_{n}$.  For  $t>\beta_{n}$, 
we  have  $c_{1}(\Rt^n) <\pi/2$ for $n$ odd and $c_{2}(\Rt^n)>\pi/2$ for $n$ even.  

To show that $\Rt^{n+1}$ is defined, we need to show we can solve $c_1(\Rt^{n+1})=a_{n+1,t}$ for $t$.    This follows,  as it did in the proof of Lemma~\ref{beta1},  from the above inequalities and the fact that the branch of $c_m(\mathcal{R}_{t})$ defined for $t> \beta_n$ has as an asymptotic value (from the left) at its discontinuity $d_{n}$,   either $a_{n+1,d_{n}}< \pi/2$ or $b_{n+1,d_{n}}>\pi/2$ as   $n$ is even or odd;  the existence of the solution follows then from  the intermediate value theorem.  This proves part $f)$ of the theorem.  

Since $\beta_{n+1}$ is a virtual center parameter and $f_t$ (or $T_t$) has symmetric attracting cycles for $t \in(\alpha_{n+1}, \beta_{n+1})$, taking the limit from below, these cycles approach   virtual cycles of period $2^{n+2}$ (or $2^{n+1}$), 
  $$
C_{2^{n+2}, \beta_{n+1}^{-}} =   \Big\{ - i\infty,\; \beta_{n+1}, \; T_{\beta_{n+1}} (\beta_{n+1}), \; \cdots,\; T^{2^{n+2}-2}_{\beta_{n+1}} (\beta_{n+1}), \; (-1)^{n+2}\frac{\pi}{2}  \Big\}
$$
and 
$$
C_{2^{n+2}, \beta_{n+1}^{-}}' =   \Big\{ i\infty, \; -\beta_{n+1}, \;T_{\beta_{n+1}} (-\beta_{n+1}),\;  \cdots, \;T^{2^{n+1}-2}_{\beta_{n+1}} (-\beta_{n+1,}),\; (-1)^{n+1}\frac{\pi}{2} \Big\}.
$$

Next, taking the limit from above, we have  $\lim_{t\to \beta_{n}^+} T_{t}^{2^{n+2}-1} (t) = i\infty$ and $\lim_{-t\to \beta_{n}^+} T_{t}^{2^{n+2}-1} (t) = -i\infty$. 
This implies, as it did when $n$ was $1$,  that at $\beta_{n+1}$,  the symmetric virtual cycles $C_{2^{n+2}, \beta_{n+1}^{-}}$ and $C_{2^{n+2}, \beta_{n+1}^{-}}'$ merge into a single cycle with double the period, 
$$
C_{2^{n+3}, \beta_{n+1}^{+}} =\Big\{   i\infty,\; -\beta_{n+1}, \; T_{\beta_{n+1}} (-\beta_{n+1}), \; \cdots,\; T^{2^{n+2}-3}_{\beta_{n+1}} (-\beta_{n+1}), \; (-1)^{n+2}\frac{\pi}{2},
$$
$$
 -i\infty, \;\beta_{n+1}, \;T_{\beta_{n+1}} (\beta_{n+1}),\;  \cdots, \;T^{2^{n+2}-3}_{\beta_{n+1}} (\beta_{n+1}),\; (-1)^{n+1}\frac{\pi}{2}\Big\}.
$$
Thus  Part $b)$ for holds for $n+1$. 
 
Since $C_{2^{n+3}, \beta_{n+1}^{+}}$ is a virtual cycle, its multiplier is $0$.  As  $t$ increases through  $\beta_{n+1}$, it becomes an attracting cycle of period $2^{n+3}$ (and so an attracting cycle of period $2^{n+1}$ of $f_t$) that attracts both asymptotic values.  By Lemma~\ref{pos} and  Proposition~\ref{monomult} the multiplier of this cycle  is a positive strictly increasing function of $t$ and there is a point $t=\alpha_{n+2}$ where the multiplier reaches $1$.   Thus $T_{\alpha_{n+2}}$ has a period $2^{n+3}$ parabolic cycle. This shows  Part $c)$ holds for $n+1$ and completes the proof of the theorem.
\end{proof}

\section{Transversality}~\label{tran}

 Theorem~\ref{bif} gives us the existence of the cycle doubling parameters $\alpha_{n}$ and 
cycle merging parameters $\beta_{n}$. The uniqueness of the $\beta_{n}$ in the interval $(\alpha_{n}, \alpha_{n+1})$ is more delicate and requires a different approach:  it requires the concept of transversality.  Once we prove this uniqueness, we will use it to obtain the uniqueness of the $\alpha_n$ in the interval $(\beta_{n-1}, \beta_n)$. In a recent preprint,~\cite{LSS}, Levin, van Strien, and Shen  study  the monotonicity of the entropy function for  families of continuous folding maps by using  holomorphic motions. Here, we adapt their techniques to prove  transversality for the   family we have been studying in this paper, the functions
\[ f_t(x) = - t \tanh ( t \tan (x)), \, t \in (0, \pi) \]
of the real axis to itself and their renormalizations.  Although the maps have discontinuities at the poles and pre-poles of the tangent, as we have shown in Section~\ref{basic},    the functions have well defined right and left limits at these points.  Keeping careful track of the signs in these limits we will prove that at all the cycle merging parameters $\beta_{n}$ defined above, transversality holds.  

Roughly speaking, this means that the function defining the   asymptotic value of the $n^{th}$ renormalization,  $c_2({\mathcal R}^n_t)=c_{2^n}(t)$, is invertible in a neighborhood of the virtual cycle parameter $\beta_n$.  By choosing the one-sided limits appropriately, we will show that the derivative is actually positive at these parameters.  To state this precisely, we 
 need some notation which define  here,  and use  throughout the rest of this section.   
 
 We fix $t_0$ as  a virtual cycle parameter of order $m$ and set $c_0= \pm \pi/2$ where the sign is chosen so that  
 $f^{m-1}_{t_0} (t_{0}) =c_0$ and,  taking the appropriate directional limit,  $f^{m}_{t_{0}} (c_{0})=c_{0}$. Set $c_{i+1}=f^{i}_{t_{0}}(t_0)$ for $i=0, 1, \ldots, m-2$. Define
 $$
P=\{ c_{0}, c_{1}, \ldots, c_{m-1}\}.
$$
so that $P$ is a subset of the Riemann sphere (and is actually contained in the real axis). For  notational simplicity, set $g=f_{t_{0}}$ so that $g(c_{i})=c_{i+1}$ for $i=0, 1, \ldots, m-2$ and $g(c_{m-1})=c_m=c_{0}$.  For later use we define the constant
$$
a=\min\{ |c_{i}-c_{j}| \;|\; 0\leq i\not=j\leq m-1\}.  
$$

Set
\begin{equation}~\label{traneq}
\Phi (t) =f_{t}^{m-1}(t) -c_0. 
\end{equation}
Then {\em transversality} holds at $t_0$ if the derivative $\Phi'(t_{0}) \not=0$.

 \subsection{ Holomorphic motions, lifts of holomorphic motions and the transfer operator}

Since our proof of  transversality, like that in~\cite{LSS},  uses holomorphic motions, lifts of these motions and the transfer operator,  we recall their definitions in our context.  For a comprehensive  study of holomorphic motions, we refer the reader to~\cite{GJW}. 

\begin{definition}\label{holommot}
Suppose $P$ is a subset of the Riemann sphere $\widehat{\mathbb{C}}$ and $\Delta_{r}=\{ z\in \mathbb{C}\; |\; |z|< r\}$ 
is the disk of radius $r>0$ centered at $0$.  We set  $\Delta=\Delta_{1}$.  
A map $h (s, z): \Delta_{r}\times P\to \widehat{\mathbb{C}}$ is called a {\em holomorphic motion of $P$ over $\Delta_{r}$}  if 
\begin{itemize}
\item[(1)] $h(0,z)=z$ for all $z\in P$;
\item[(2)] for any fixed $s\in \Delta_{r}$, the map $h_{s} (\cdot): P\to \widehat{\mathbb{C}}$ is injective;
\item[(3)] for any fixed $z\in P$, the map $h_{z} (\cdot): \Delta_{r}\to  \widehat{\mathbb{C}}$ is holomrophic.
\end{itemize}
\end{definition}

 Let $E=\{k\pi+\pi/2\;|\; k\in \mathbb{Z}\}$ denote the set of poles of the 
 tangent map $T_w(z) =iw\tan z$.   They are  poles of the real map 
  $$
 f_{t} (x) = it \tan (it \tan x) =-t \tanh (t\tan x), \, t,x \in \RR,
 $$
 in the sense that although the map is discontinuous at such a point, the left and right limits are well defined. 
    
 Note that every real map $f_{t}(x)$ has a meromorphic extension 
 to the whole complex plane $\mathbb{C}$ as
 $$
F(w,z)= F_{w} (z) =T_{w}^{2}(z)=iw \tan (iw \tan z).
 $$
 We will use this notation for this extension below. 
The following lemma follows directly from the definition of $F_w$ and the facts that  the complex map $F_{w}(z)$ has asymptotic  values at $\{\pm w\}$ and is not defined at the pre-poles of $T_w(z)$.  Let 
$$
S_{w,0}=\mathbb{C}\setminus (E\cup T_{w}^{-1}(E))\quad \hbox{and}\quad S_{w,1}=\mathbb{C}\setminus \{-w, w\}
$$
be two Riemann surfaces determined by  the poles and pre-poles of $T_{w}(z)$.

\begin{lemma}\label{cov}
The map 
$$
F_{w}(z) : S_{w.0}\to S_{w,1}
$$
is a holomorphic covering map of infinite degree.
\end{lemma}

\begin{definition}
Suppose $h (s, z): \Delta_{r}\times P\to \widehat{\mathbb{C}}$ is a holomorphic motion such that $h(s,c_{0})=c_{0}$ for all $s\in \Delta_{r}$ and let $c_{1}(s) =h(s, c_{1})$.  We say another holomorphic motion $\widehat{h} (s, z): \Delta_{r}\times P\to \widehat{\mathbb{C}}$ is a {\em lift of $h$} if 
$\widehat{h}(s,c_{0})=c_{0}$ for all $s\in \Delta_{r}$, and if 
\begin{equation}~\label{lifteq}
h (s, c_{i+1}) = F_{c_{1}(s)} (\widehat{h} (s, c_i)),
\end{equation}
for all $ i=1, \ldots, m-1$ and all $s\in \Delta_{r}$; that is, the following diagram commutes

$$\xymatrix{{\Delta_r \times P} \ar[r]^{\widehat{h}}\ar[d]_{Id \times g} &\CC\ar[d]^{F_{c_{1}(s)}}\\
\Delta_r \times P\ar[r]^{h}& \CC.}$$
 \end{definition}
 
  \begin{definition}\label{transfop}
Suppose $\widehat{h}$ is a lift of the holomorphic motion $h$. If we differentiate the equation
$$
h (s, g(c_{i}))= F_{c_{1}(s)} (\widehat{h} (s, c_{i})).
$$
at $s=0$ for each $i=1, \ldots, m$, we get an $m\times m$-matrix $A$ such that 
\begin{equation}~\label{tranop}
\Big( \frac{\partial\widehat{h}(s, c_{i})}{\partial s}\Big|_{s=0}\Big)=A\Big( \frac{\partial h(s, c_{i})}{\partial s}\Big|_{s=0}\Big).
\end{equation}
The entries in the matrix $A$  depend only  on the partial derivatives of $F(w,z)$  at $s=0$ evaluated at the points of $P$.   

Computing, we have  $A=(a_{i,j})$ 
\begin{eqnarray}
\label{matrixA1} a_{i,1} = \frac{- \frac{\partial F}{\partial w}(c_1, c_i)}{\frac{\partial F}{\partial z}(c_1, c_i)}, \, \, i = 1, \ldots, m-1, \\
\label{matrixA2} a_{i,i+1} = \frac{1}{\frac{\partial F}{\partial z}(c_1, c_i)}, \, \, i = 1, \ldots, m-1, \\
\label{matrixA3} a_{i,j} = 0 \mbox{ otherwise.  }   
 \end{eqnarray}
We call $A$ the {\em Transfer Operator} associated with $g$. 
\end{definition}

The following lifting theorem is the key to  proving transversality for the tangent family. 
 
 \medskip
\begin{theorem}~\label{lift}
Set $W=\mathbb{C}\setminus E$ and suppose $h_{0} (s, z): \Delta\times P\to \widehat{\mathbb{C}}$ 
is a holomorphic motion such that $h_{0}(s,c_{0})=c_{0}$ for all $s\in \Delta$. 
Then we can find  a real number $r>0$, and a sequence of holomorphic motions 
$$
\{ h_{k} (s, z): \Delta_{r}\times P\to \widehat{\mathbb{C}}\}_{k=0}^{\infty},
$$ 
such that for all $k=0,1, \ldots$, $h_{k+1} (s,z)$ 
is a lift of $h_{k}(s,z)$  satisfying  
\begin{itemize}
\item[(i)] $c_{1,k} (s)=h_{k} (s, c_{1})\in W$ and
\item[(ii)] $\sup_{s\in \Delta_{r}, z\in P, k=0, 1, \ldots} |h_{k} (s, z)| < \infty$.
\end{itemize}
\end{theorem} 

\begin{proof}
We first prove that  for each $k\geq 0$,  the holomorphic motion $h_{k}(s,z)$ can be lifted.  To define its lift, first set $h_{k+1} (s, c_{0})=c_{0}$.  Next, by injectivity, for any $c_{i}\in P$,  $0<i \leq m-1$, 
$h_{k} (s, g(c_{i})) =h_{k}(s, c_{i+1}) \in S_{c_{1,k}(s),1}$. Since $F_{c_{1}(s)}: S_{c_{1,k}(s), 0}\to S_{c_{1,k}(s),1}$ 
is a holomorphic covering and since $\Delta$ is simply connected, the map $h_{k} (s, g(c_{i})): \Delta \times P \to S_{c_{1,k}(s), 0}$ 
can be lifted to a holomorphic covering map $h_{k+1} (s, c_{i}): \Delta \times P \to S_{c_{1,k+1}(s), 1}$ such that 
$$
 h_{k} (s, g(c_{i}))= F_{c_{1,k}(s)} (h_{k+1} (s, c_{i})) = F(c_{1,k}(s), h_{k+1}(s, c_{i})).
 $$
 We need to check injectivity for $h_{k+1}$.  It is clear that if $0<i\not=j<m-1$, then $h_{k+1} (s, c_{i})\not=h_{k+1} (s, c_{j})$ because $h_{k} (s, c_{i+1}))\not=  h_{k} (s, c_{j+1})$. We only need to check that $h_{k+1} (s, c_{m-1})\not=h_{k+1} (s, c_{j})$, for $j \neq m-1$.   Note that because 
    $h_{k} (s, g(c_{m-1})) =c_{0} \in E$,  $h_{k+1} (s, c_{m-1})$ is a pre-pole of $F_{c_{1,k}(s)}(z)$  and this pre-pole
   depends holomorphically on $s\in \Delta$.  Since $h_k(s, g(c_j)) \notin E$ for any $j \neq m-1$, the corresponding 
     $h_{k+1} (s, c_{j})$ cannot be a pre-pole and so is different from     $h_{k+1} (s, c_{m-1})$.  
 Thus $h_{k+1}(s,z): \Delta\times P\to \widehat{\mathbb{C}}$ defines a holomorphic motion which is 
 a lift of $h_{k}(s,z): \Delta\times P\to \widehat{\mathbb{C}}$.
 Since $h_{k} (s,c_{i}) \in S_{c_{1,k}(s), 0}\subset W$ for all $i=1, 2, \cdots, m-1$, we have $(i)$. 
 
Now we prove $(ii)$ by   contradiction. Suppose that for every $r<1$, 
$$
\sup_{s\in \Delta_{r}, z\in P, k=0, 1, \cdots} |h_{k} (s, z)| = \infty.
$$  
Since $P$ contains only finitely many  points,  this assumption implies that for some fixed $x\in P$ there is a  sequence $\{k_{n}\}$ of integers and a sequence of complex numbers $s_{n}\in \Delta$ such that $h_{k_{n}}(s_{n}, x) \to \infty$ as $n\to \infty$. 
Since $W$ is a hyperbolic Riemann surface and since the holomorphic functions $h_{k} (s, z) $ take values in $W$ for all $s\in \Delta$ and all $k\geq 1$, 
it follows  that $\{ h_{k}(s, x) \}$ is a normal family for $s\in \Delta$;  thus the sequence $h_{k_{n}} (s, x)$ has a subsequence, which  we again  denote by $h_{k_n}$,  that converges to a holomorphic function $h(s)$ or to the constant $\infty$. Since $h_{k_{n}} (0,x)=x$ for all $n$,  $h_{k_{n}} (s, x)$ must converge to a holomorphic function $h(s)$ in $\Delta$;  therefore, for $n$ sufficiently large, $h_{k_{n}} (s, x)$ is  bounded on compact subset of $\Delta$,  contradicting our assumption.  Therefore,    for every  $0<r<1$, there is  an $M(r)>0$, such that for all $s \in \Delta_r$ and all $z \in P$, $(ii)$ holds; that is, 
$$
|h_{k} (s, z)|\leq M, \quad \forall\; s\in \Delta_{r}, \;\; \forall\; z\in P.
$$
\end{proof}    

 \subsection{The Spectral Radius of the Transfer Operator.}
 
\begin{lemma}
The spectral radius of $A$ is less than or equal to $1$.
\end{lemma}

\begin{proof} 
Let $\mathbf{v_0}=(v_{0,1}, \ldots, v_{0,m}=v_{0,0}=0)$ be a vector in $\mathbb{R}^{m}$ such that $|v_{0,i}| \leq a/3$.
Define a holomorphic motion of $P$ over $\Delta$ by $h_{0} (s, c_i)=x+s v_{0,i}$.  Note we have $h_0(s,c_0)=c_0$.   The condition on $|v_i| $ ensures injectivity.   By  Theorem~\ref{lift}, given $r<1$, we can find a sequence of holomorphic motions $h_k$ of $P$ over $\Delta_r$ and a constant $M$ such that $\sup_{s\in \Delta_{r}, z\in P, k=0, 1, \ldots} |h_{k} (s, z)| < M$.

Since by definition,  $v_{0,i}=\frac{\partial h_0(s,c_i)}{\partial s}|_{s=0}$, inductively applying 
the transfer operator $A$,  we obtain a sequence of vectors  ${\mathbf v}_{k+1}=A {\mathbf v}_k = A^k \mathbf{v_0}$.  
By the boundedness in part $(ii)$ of Theorem~\ref{lift} and  Cauchy's Theorem,  there is a  constant $M>0$ such that 
$$
| v_{k,i}|\leq M, \quad \forall\; k>0,
$$
and hence 
\[ \| A^k {\mathbf v}_k \| \leq M, \quad \forall\; k>0.  \]

This implies that $\|A^{k}\| \leq (3M)/a$ so that  the spectral radius 
$$
\rho=\lim_{k\to \infty} \sqrt[k]{\|A^{k}\|} \leq   1.
$$
 \end{proof}

 \subsection{Non-transversality and the spectral radius of $A$}

 We saw above that the spectral radius, or largest eigenvalue of $A$ has a maximum value of $1$.   Here we show that achieving  this maximum is equivalent to non-transverality at $t_0$, that is $\Phi'(t_0)=0$ where 
$\Phi (t)$ is the function defined in (\ref{traneq}) extended as a function on $\mathbb C$. 

\medskip
\begin{lemma}~\label{eig1}
$\Phi' (t_{0}) =0$ if and only if $1$ is   an eigenvalue of $A$.
\end{lemma}

\begin{proof}
{\bf Proof of the  ``if'' statement}:  Suppose $1$ is   an eigenvalue of $A$.  This means that there is  a non-zero vector ${\mathbf v} =(v_{i})\in \mathbb{R}^{m}$ with   $\max_i |v_i | =a/3$ such that $A{\mathbf v}={\mathbf v}$. Define a holomorphic motion of $P$ over $\Delta$ by $h(s,c_{i}) =c_{i}+s v_{i}$. Suppose $\widehat{h}(s,c_{i})$ is a lift of $h$. Since $A{\mathbf v}={\mathbf v}$, we have  
$$
\widehat{h}(s,c_{i})=h(s, c_{i}) +{\rm O}(|s|^{2})= c_{i}+s v_{i}+{\rm O}(|s|^{2}).
$$
From  equation (\ref{lifteq}), for $i=1, \ldots, m-2$, we obtain
$$
F_{c_{1}(s)} (\widehat{h}(s, c_{i})) = h(s, c_{i+1}) +{\rm O}(|s|^{2}), 
$$
or equivalently, 
$$
F_{c_{1}(s)} (c_{i}+sv_{i}) = c_{i+1}+sv_{i+1} +{\rm O}(|s|^{2}).
$$
For $i=m-1$ we have
$$
F_{c_{1}(s)} (h (s, c_{m-1})) = h(s, c_{0}) +{\rm O}(|s|^{2}),
$$ 
or equivalently, 
$$
F_{c_{1}(s)} (c_{m-1}+sv_{m-1}) = c_{0} +{\rm O}(|s|^{2}).
$$ 
Using  equations~(\ref{matrixA1})--(\ref{matrixA3}), we see that $v_m=v_{0}=0$ and if  $i=1, 2, \cdots, m-1$,  then $v_i=0$ implies that $v_{i+1}=0$.  Thus,  since $\mathbf{v} \neq 0$,    
  $v_{i}\not=0$ for $i=1, 2, \cdots, m-1$.   Moreover,  from the above we conclude that 
$$
\Phi (c_{1}(s)) = F_{c_{1}(s)}^{m-1} (c_{1}(s))-  c_{0}  = {\rm O}(|s|^{2}), 
$$
and therefore that 
$
\Phi' (c_{1}(0)) c_{1}'(0)= \Phi' (c_{1}(0)) v_{1}=0.
$
 Since  $c_{1}(0)=t_{0}$,  this says $\Phi'(t_{0})=0$ as claimed.

{\bf Proof of the ``only if'' statement}:  This is relatively  easier.  Since we assume  $\Phi'(t_{0}) =\Phi' (c_{1})=0$, we have a $v_{1}\not=0$ with $|v_{1}|\leq a/3$, such that $\Phi (c_{1}+sv_{1}) ={\rm O}(|s|^{2})$.  Let $c_{1}(s) =c_{1} +sv_{1}$,  let  $v_m=v_{0}=0$ and  for $i=2, \cdots, m-1$ define 
$$
v_{i} = \frac{d F_{c_{1}(s)} (c_{i}(s))}{ds} \big|_{s=0}.
$$ 
  This is  a non-zero vector ${\mathbf v} =(v_{i})\in \mathbb{R}^{m-1}$ such that $A{\mathbf v}={\mathbf v}$. Therefore $1$ is an eigenvalue of $A$ as claimed. 
\end{proof} 

\subsection{Asymptotic invariance of lifts holomorphic motions.}

\begin{definition}
A holomorphic motion $h$ of $P$ over  $\Delta_{r}$  is called {\em asymptotically invariant of order $l$ }if there is a lift 
$\widehat{h}$ such that
$$
\widehat{h}(s,z)-h(s,z)= {\rm O}(|s|^{l+1})\quad \forall \; s\in \Delta_{r}, \; z\in P.
$$
\end{definition}
The following lemma is standard calculus.  We write out the proof to establish the notation for the proof of the following lemma. 

\begin{lemma}\label{diff}
Suppose $F(w, z): \mathbb{C}^2\to \mathbb{C}$ is holomorphic and $\sigma(s), \tau(s): \Delta_{r}\to \mathbb{C}^2$  are two holomorphic maps such that $\sigma(0)=\tau(0)=(a,b)$. Suppose further that for some $l>0$, 
$$
\sigma(s)-\tau(s) ={\rm O}(|s|^{l}).
$$
Then, writing $\sigma(s)=(\sigma_1(s), \sigma_2(s))$ and $\tau(s)=(\tau_1(s), \tau_2(s))$, we have
$$
F(\sigma(s))-F(\tau(s)) =\frac{\partial F}{\partial w}(a,b)(\sigma_1(s)-\tau_1(s))+ \frac{\partial F}{\partial z }(a,b)(\sigma_2(s)-\tau_2(s))+{\rm O}(|s|^{l+1}).
$$
\end{lemma}

\begin{proof}    The Taylor series for $F(w,z)$ is 

\begin{multline*}
F(w,z)-F((a,b))=\frac{\partial F}{\partial w}(a, b) (w-a)+ \frac{\partial F}{\partial z}(a, b) (z-b)+\frac{\partial^2 F}{\partial w^2}(a, b) (w-a)^2\\
+\frac{\partial^2 F}{\partial w\partial z}(a, b) (w-a)(z-b)+\frac{\partial^2 F}{\partial z^2}(a, b) (z-b)^2+ \text{higher order terms.}
\end{multline*}
Substituting    first   $\sigma(s)=(\sigma_1(s),\sigma_2(s))$ and then  $\tau(s)=(\tau_1(s),\tau_2(s))$ for $(w,z)$ in the Taylor series  and subtracting, we obtain 
\begin{multline*}
F(\sigma(s))-F(\tau(s))=\frac{\partial F}{\partial w}(a,b)(\sigma_1(s)-\tau_1(s))+ \frac{\partial F}{\partial z}(a,b)(\sigma_2(s)-\tau_2(s))\\
+\frac{\partial^2 F}{\partial w^2}(a,b) (\sigma_1(s)+\tau_1(s)-2a)(\sigma_1(s)-\tau_1(s))\\
+\frac{\partial^2 F}{\partial w\partial z}(a, b) [(\sigma_1(s)-a)(\sigma_2(s)-\tau_2(s))+(\tau_1(s)-b)(\sigma_1(s)-\tau_1(s))]\\
+\frac{\partial^2 F}{\partial z^2}(a,b) (\sigma_2(s)+\tau_2(s)-2b)(\sigma_2(s)-\tau_2(s))+ \text{higher order terms}\\
=\frac{\partial F}{\partial w}(a,b)(\sigma_1(s)-\tau_1(s))+ \frac{\partial F}{\partial z}(a,b)(\sigma_2(s)-\tau_2(s))+{\rm O}(|s|^{l+1})
\end{multline*}
\end{proof}

 The next lemma implies that the asymptotic order of a holomorphic motion is preserved under lifting; that is, if $h$ is holomorphic motion of asymptotically invariant of order $l$, so is its lift $\widehat{h}$.  Below, we will apply it to the sequence of lifts $\{h_k\}$ of the holomorphic motion $h_0$:  it will show that if  
  $h_{k+1}$ is a lift of $h_k$ with  $h_1- h_0={\rm O}(|s|^{l+1})$, then $h_{k+1}-h_k={\rm O}(|s|^{l+1})$. 
   
\begin{lemma}\label{liftorder}
Suppose that for some $l>0$, $h_0$ and $h_1$ are holomorphic motions of $\Delta_r \times P$ over $\mathbb C$ satisfying $$h_0(s,c_{i})-h_1(s,c_{i})={\rm O}(|s|^l).$$  If $\widehat{h}_0$ $\widehat{h}_1$ are respective lifts of the motions, then $$\widehat{h}_0(s,c_{i})-\widehat{h}_1(s,c_{i})={\rm O}(|s|^l).$$ 
\end{lemma}

\begin{proof} We use superscripts to denote the functions $c_1(s)$ for each of the motions: 
 $c_1^0(s)=h_0(s,c_1)$ and $c_1^1(s)=h_1(s,c_1)$.  
 By hypothesis,  $c_1^0(s)-c_1^1(s)={\rm O}(|s|^l)$. Since $\widehat{h}_0$ and $\widehat{h}_1$ are lifts of $h_0$ and $h_1$, we have $$h_0(s,c_{i+1})=F_{c_1^0(s)}( \widehat{h}_0(s,c_i))\ \text{and}\  h_1(s,c_{i+1})=F_{c_1^1(s)}(\widehat{h}_1(s,c_i)).$$  Subtracting and applying Lemma \ref{diff}  with $F(w,z)=F_w(z)$, we get
\begin{multline*}
h_0(s,c_{i+1})-h_1(s,c_{i+1})=F_{c_1^0(s)}( \widehat{h}_0(s,c_i))-F_{c_1^1(s)}( \widehat{h}_1(s,c_i))\\
=\frac{\partial F}{\partial w}\big|_{s=0}(c_1^0(s)-c_1^1(s))+\frac{\partial F}{\partial z}\big|_{s=0}(\widehat{h}_0(s,c_i)-\widehat{h}_1(s,c_i))+{\rm O}(|s|^{l+1})
\end{multline*}

Therefore $\widehat{h}_0(s,c_i)-\widehat{h}_1(s,c_i)={\rm O}(|s|^{l})$.
\end{proof}

The following lemma gives the construction of a new holomorphic motion $H$ from the sequence of lifts, $\{h_k\}$, of a given motion $h_0$ of asymptotic order $l$, that is asymptotically invariant of   order $l+1$.  

\begin{lemma}\label{increaseorder}
Suppose that for some $r>0$ and $l\geq 1$, we have an asymptotically invariant holomorphic motion $h$ of $P$ over $\Delta_{r}$  of order $l$.
 Then we can construct another holomorphic motion $H$ of $P$ over $\Delta_{r'}$ for some $r'>0$ which is asymptotically invariant of order $l+1$.
\end{lemma}

\begin{proof}
Take $h_{0}(s,z)=h(s,z)$. By Theorem~\ref{lift}, we can find a sequence of holomorphic motions $\{h_{k}(s,z)\}_{k=0}^{\infty}$ such that $h_{k+1}$ is a lift of $h_{k}$ satisfying $(i)$ and $(ii)$.  Consider the two means
$$
\mu_{k} (s,z) =\frac{1}{k}\sum_{i=0}^{k-1} h_{i}(s,z)\quad \hbox{and}\quad
\nu_{k} (s,z) =\frac{1}{k}\sum_{i=1}^{k} h_{i}(s,z).
$$
Both of them are uniformly bounded. Thus they form  normal families and  we can find     subsequences
$$
\mu_{k_{n}} (s,z) =\frac{1}{k_{n}}\sum_{i=0}^{k_{n}-1} h_{i}(s,z) \quad \mbox{and}\quad 
\nu_{k_{n}} (s,z) =\frac{1}{k_{n}}\sum_{i=1}^{k_{n}} h_{i}(s,z)   $$
that both converge to the same holomorphic limit  $H(s,z)$ as $k$ goes to infinity.  
 This limit is defined on $\Delta_{r}\times P$ and is holomorphic in $s$. Since $H(0,z)=z$ and since $P$ contains only a finite number of points, we can find  $0<r'\leq r$ such that for any $s\in \Delta_{r'}$, $H(s, z)$ is injective on $P$. Thus $H(s,z)$ is a holomorphic motion as well.  
   
We need to show that   $H$ is asymptotically invariant of order $l+1$.  This will follow from:

\textbf{Claim:} For each $z\in P \setminus \{c_0\}$ and any $k\geq 1$,
$$F_{\mu_k(s,c_1)}(\nu_k(s,z))=\mu_k(s,g(z))+{\rm O}(|s|^{l+2}).$$

To see this, denote the lift of $\mu_k$ by $\widehat{\mu}_k$, that is,
$$F_{\mu_k(s,c_1)}(\widehat{\mu}_k (s,z))=\mu_k(s,g(z)).$$
By the definitions of $\mu_k, \nu_k$ and $F_{\mu_k(s,c_1)}$, the claim implies that 
$$
\widehat{\mu}_k(s,z)-\nu_k(s,z)=\rm{O}(|s|^{l+2})
$$ 
for each $k$. Using an argument similar to the proof of Theorem~\ref{lift}, and, if necessary,  taking $r'$ smaller, we can assume $\{\widehat{\mu}_k(s,z)\}$ 
is a bounded sequence on $\Delta_{r'}\times P$. Taking a subsequence if  necessary, we obtain a holomorphic limit,  $\widehat{\mu}_k(s,z)\to \widehat{H}(s,z)$ on $\Delta_{r'}\times P$ so that $\widehat{H}$ is a lift of $H$ and for $ s\in \Delta_{r'}, z\in P$. 
Moreover,  the sequence $\{\widehat{\mu}_k(s,z)-\nu_k(s,z))\}$  is also a bounded sequence of holomorphic functions on $\Delta_{r'}$ so that 
$$
\widehat{H} (s,z) - H(s,z) =O(|s|^{l+2}).
$$
Therefore $H(s,z)$ is asymptotically invariant of order $l+1$.

\textbf{Proof of the claim:}  Fix $z\in P \setminus\{c_0\}$ and $k\geq 1$.   By the construction of $h_i(s,z),$ we have 
$$  F_{h_i(s,c_1)}( h_{i+1}(s,z)) =h_i(s,g(z))$$ for every $i\geq 0$. Thus
$$\mu_k(s,g(z))=\frac{1}{k}\sum_{i=0}^{k-1} F_{h_i(s,c_1)}(h_{i+1}(s,z)).$$

By  Lemma~\ref{liftorder}, we have $h_{i+1}(s,z)-h_i(s,z)={\rm O}(|s|^{l+1})$ for all $z\in P$ and $i\geq 0$. Therefore all the functions $h_i(s,z)$, $\mu_i(s,z)$, $\nu_i(s,z)$ have the same derivatives up to order $l$ at $s=0$. Applying  Lemma~\ref{diff}, we have  
\begin{multline*}F(h_i(s,c_1),h_{i+1}(s,z))-F(\mu_k(s,c_1),\nu_k(s,z))=\\
\frac{\partial F}{\partial w}(c_1, z)(h_i(c_1,z)-\mu_k(s,c_1))+\frac{\partial F}{\partial z}(c_1, z)(h_{i+1}(s,c_1)-\nu_k(s,z))+{\rm O}(|s|^{l+2}).
 \end{multline*}
 Summing from $i=0$ to  $k-1$, we obtain
$$\frac{1}{k}\sum_{i=0}^{k-1}F(h_i(s,c_1),h_{i+1}(s,z))=F_{\mu_k(s,c_1)}( \nu_k(s,z))+{\rm O}(|s|^{l+2}).$$ Thus we have the equality
$$F_{\mu_k(s,c_1)}(\nu_k(s,z))=\mu_k(s,g(z))+{\rm O}(|s|^{l+2}) $$
 which proves the claim. 
\end{proof}

\subsection{Proof of transversality}
Finally, we can state our transversality result as a theorem. Although we state it for real parameters, the proof works for all complex parameters provided we move along paths where we can define virtual pre-poles.  

\begin{theorem}\label{transversality}
The  tangent family $f_t$ is transversal at any virtual center parameter $t_{0}$ such that, taking appropriate directional limits,  $f_{t_{0}}^{m}(t_{0}) =t_0$. 
\end{theorem}

\begin{remark} In the proof of Theorem~\ref{bif} we showed that for each $n>0$  there is at least one parameter $\beta_n$ which is a solution of   
$c_1(\mathcal{R}^n_t) = c_2(\mathcal{R}^{n-1}_t)=a_{n,t}$ where $a_{n,t}$ is a pre-pole of order $n$.  That is, such that $\beta_n$ is a virtual cycle parameter. Take $t_0= \beta_n$ in Theorem~\ref{transversality}.   Transversality implies   $c_1(\mathcal{R}^n_t)$ is an invertible function at $t=\beta_n$ where limits are taken with appropriate signs along the $t$ axis.  This means that  cycle merging  actually occurs at $\beta_n$; that is, the limits of $T^{2^{n+1}-1}(t)$ have opposite signs as $t$ approaches $\beta_n$ from opposite sides. 
\end{remark}

\begin{proof}
The main point in the proof is to show that $1$ is not an eigenvalue of the transfer operator $A$ for the family $F_w$. Then Lemma~\ref{eig1} implies that 
$\Phi'(t_{0})\not=0$.  Restricting to  $w$ real implies the tangent family is transversal at $t_{0}$.

We assume $1$ is  an eigenvalue of $A$ and obtain a contradiction.
 First,  by definition, the function $\Phi(w)=F^{m-1}_w - c_0$ satisfies $\Phi(c_1)=0$, and is non-constant and holomorphic in $w$ in a neighborhood of $w=c_1$.  Therefore, for some integer $l \geq 1$ we have $a_l \neq 0$ and 

\[ \Phi(w)= a_l (w-c_1)^l + {\rm O}(|w-c_1|^{l+1}). \]

Next, we choose an eigenvector ${\mathbf v} \in \RR^m$ with $\max_i|v_i| =a/3$  such that $A{\mathbf v}={\mathbf v}$.  As in the proof of the ``if'' part of Lemma~\ref{eig1},  we define a non-degenerate holomorphic motion $h_0$ of $P$ over $\Delta$ that is asymptotically invariant of order $1$.  By Theorem~\ref{lift} we obtain a uniformly bounded sequence of lifts,  $\{h_k\}$, which are holomorphic motions of $P$ over $\Delta_{r_{1}}$ for some $0<r_{1}<1$.
By Lemma~\ref{liftorder}, all $h_k(s, c_i)$ are asymptotically invariant of order $1$.

Set $H_{0}=h_{0}$ and apply Lemma~\ref{increaseorder}  with $h=H_{0}$,   to obtain a new holomorphic motion $H_{1}$ of $P$ over $\Delta_{r_{1}}$ (taking $r_{1}$ smaller if necessary) such that $H_{1}-H_{0}={\rm O}(s^{3})$;  that is, $H_{1}$ is asymptotically invariant of order $2$.

We now set  $h_{0}=H_{1}$ in Theorem~\ref{lift} and apply it  to obtain a new sequence of lifts, $\{h_{k}\}$;   again Lemma~\ref{liftorder}, implies that all of these motions are asymptotically invariant of order $2$. They are holomorphic motions of $P$ over $\Delta_{r_{2}}$ for some $0<r_{2}<r_{1}$.    

 Repeating this process,  for each  integer $N \geq1$, we can find a non-degenerate holomorphic motion of $P$ over $\Delta_{r_{N}}$ for some $0<r_{N}<r_{N-1}<1$ which  we denote by $H_N$ that is 
 asymptotically invariant of order $N$.  
 It follows that  
$$ 
\Phi( (H_{N}(s,c_1)) = F^{N-1}_{H_N(s,c_1)} (H_N(s,c_1) )- c_0 = {\rm O}(|s|^{N+1}). $$ In particular, if $N=l+1$, this implies that $a_l=0$ giving us the required contradiction.
\end{proof}

\subsection{Positive Transversality}
In this section we improve our transitivity result for the real tangent family.  We show that, taking directional limits appropriately, not only is the derivative in question non-zero,  it is positive.    

Continuing with our notation above we have $\Phi(w)=F^{m-1}_w(w) -c_0.$ Its derivative is 
$$
\Phi'(w)=\frac{\partial F_w}{\partial {w}}+\frac{\partial F_w}{\partial z}\Big( \frac{\partial F_w}{\partial {w}}+\frac{\partial F_w}{\partial z}\Big(\cdots\Big)\Big).
$$ 
In particular,
\begin{eqnarray*}
\Phi'(t_0)&=&\frac{\partial F_{t_0}(c_{m-1})}{\partial {w}}+\frac{\partial F_{t_0}(c_{m-1})}{\partial z}\Big( \frac{\partial F_{t_0}(c_{m-2})}{\partial {w}}+\frac{\partial F_{t_0}(c_{m-2})}{\partial z}\Big(\cdots\Big)\Big)\\
          &=& \frac{\partial F_{t_0}(c_{m-1})}{\partial {w}}+\frac{\partial F_{t_0}(c_{m-1})}{\partial z} \frac{\partial F_{t_0}(c_{m-2})}{\partial {w}}+\cdots+\frac{\partial F_{t_0}(c_{m-1})} {\partial z} \cdots\frac{\partial F_{t_0}(c_{1})}{\partial z}\\
          &=&\frac{\partial F_{t_0}(c_{m-1})}{\partial {w}}+\frac{(F^{m-1}_{t_0})'(c_1)}{(F^{m-2}_{t_0})'(c_1)}\frac{\partial F_{t_0}(c_{m-2})}{\partial {w}}+(F^{m-1}_{t_0})'(c_1)
\end{eqnarray*}

Set $$L(z)=\frac{\partial F_w(z)}{\partial w}|_{w=t_0}$$  and define the 
polynomial $$P(\rho)= 1+\sum_{n=1}^{m-1} \frac{\rho^n L(c_n)}{(F_{t_0}^n)'(c_1)}.$$

Then we can write \begin{eqnarray}\label{poseqn}\Phi'(t_0)=(F_{t_0}^{m-1})'(c_1) \cdot P(1) \end{eqnarray}
The relationship between the zeros of the polynomial $P(\rho)$ and the eigenvalues of the transfer operator is summarized in the following lemma. 
\begin{lemma}\label{zeros} For any $\rho\in \mathbb{C}$, $\det(I-\rho A)=0$ if and only if $P(\rho)=0$.
\end{lemma}
\begin{proof}

When $\rho=0$, $\det (I)=P(0)=1$. Assume $\rho\neq 0$ and define  a  local deformation $(g^\rho, F_w^\rho)$  of $(g, F_w)$ as follows:\\
 
\begin{enumerate}
\item for $z$ in a neighborhood of each $x\in P\setminus{c_0}$, $F_w^\rho(z)=F_w(z)+\frac{g'(x)}{\rho}(z-x)$; 
\item $g^\rho(x)=g(x)$ for $x=c_0$, and for $z$ in a neighborhood of each $x\in P\setminus{c_0}$, $g^\rho(z)=F_{c_1}^\rho(z)$.
\end{enumerate} 
If $A^\rho$ is the transfer operator associated with the deformation $F_{w}^{\rho}$,  a simple computation gives $A_\rho=\rho A$. 

We define a map $\Phi_\rho(w)=({F_{w}^{\rho}})^{m-1}(w)-c_0$. Then $$\det(I-1 \cdot A_\rho)=0 \text{ if and only if } (\Phi_\rho)'(t_0)=0.$$ Direct computation shows that $$(\Phi_\rho)'(t_0)=\frac{P(\rho)}{\rho^{m-1}}(\Phi)'(t_0).$$

Therefore $$\det(I-\rho A)=0 \text{ if and only if } P(\rho)=0.$$

\end{proof}

\medskip
\begin{corollary}[Positive Transversality]\label{postran}
With the notation of Theorem~\ref{transversality}, for $t$ real in a neighborhood of $t_0$, 
$$
\frac{\Phi'(t_{0})}{(F_{t_{0}}^{m})'(c_{1})}>0.
$$
\end{corollary}
\begin{proof} By equation~(\ref{poseqn}) it suffices to show $P(1)>1$.  
Rewrite the polynomial as $P(\rho)=\Pi_{i=1}^{m-1}(1-\rho/\rho_i)$  where $\rho_i \neq 0$   are the zeros of $P(\rho)$.   By  Lemma \ref{zeros},   $P(\rho_i)=\det(I-\rho_i A)=0$, which implies that $1$ is an eigenvalue of $A^{\rho_i}$ and thus $1/\rho_i$ is an eigenvalue of $A$. By Lemma~\ref{eig1},  
all eigenvalues of $A$ satisfy $\{1/|\rho_i|\leq 1,1/ \rho_i\neq 1\}$.  

 Restricting to the real valued family, the eigenvalues of $A$ are all real or complex conjugate in pairs so evaluating the polynomial at $\rho=1$ in light of the above, we conclude 
 $$P(1)=\Pi_{i=1}^{m-1}(1-1/\rho_i)>0.$$
\end{proof}

 Positive transversality gives us the   uniqueness of the $\beta_{n}$. 

\medskip
\begin{corollary}[Uniqueness]~\label{uniqueness}  For each $n>1$, there is unique  parameter $t=\beta_n$ in the interval $(\alpha_n, \alpha_{n+1})$ of the renormalization sequence for the family $f_t$ where cycle merging of order $n$ occurs.  In particular, there is a unique virtual cycle parameter of order $n$ in the sequence. 
\end{corollary} 

\begin{proof} In the proof of Theorem~\ref{bif} we showed that for each $n>0$  there is at least one parameter $\beta_n$ which is a solution of   
$c_1(\mathcal{R}^n_t) = c_2(\mathcal{R}^{n-1}_t)=a_{n,t}$ where $a_{n,t}$ is a pre-pole of order $n$. More precisely, with $c_{0}$ equal to  either $\pi/2$ or $-\pi/2$, depending on $n$,  $\beta_{n}$ is a solution of  $f_{t}^{2^{n}}(t)-c_{0}=0$ in the interval $(\alpha_n, \alpha_{n+1})$. Since $f_{t}$ depends on $t$ holomorphically,  
there are only finitely many solutions.  The curve $\gamma (t)=f_{t}^{2^{n}}(t)-c_{0}$  is a smooth curve defined on $(\alpha_n, \alpha_{n+1})$ and transversality implies that at each root each root $\gamma (t)$ is  locally either strictly increasing   or strictly decreasing.  
  Positive transversality implies that it has the same direction at each root.   
Therefore, it can not have more than one root because if it did,  the directions at adjacent roots would have to be opposite.  
\end{proof}

The uniqueness of the $\alpha_n$ now follows directly from  uniqueness of the $\beta_{n}$.   

\begin{corollary}\label{uniquealpha} For each $n>1$, there is unique  parameter $t=\alpha_n$ in the interval $(\beta_{n-1}, \beta_{n})$ of the renormalization sequence for the family $f_t$ where cycle doubling of order $n$ occurs.   \end{corollary}

\begin{proof}
Let $C_{2^n,t}$ be the merged cycle for $t$ just to the right of $\beta_{n-1}$.  By Proposition~\ref{monomult}, as $t$ increases, the multiplier of $C_{2^n,t}$ increases to $+1$ at $\alpha_n$.  If the mulitplier continues to increase as $t$ increases,  we have cycle doubling into a new hyperbolic component and, again by Proposition~\ref{monomult},  the multipliers of the two new cycles decrease the  right endpoint of this component must be $\beta_n$  so that $\alpha_n$ is unique in $(\beta_{n-1}, \beta_{n})$.  

Otherwise,  as $t$ increases beyond $\alpha_n$ the multiplier of $C_{2^n,t}$ decreases again and $t$ is in a new hyperbolic component. By Proposition~\ref{monomult},  at the  right endpoint of this new component, $t^*$,  the multiplier is $0$ and is a virtual cycle parameter of order $2^n$.  Thus $t^* > \beta_{n-1}$  is a second  root of the curve $\gamma(t)$ defined in  the proof of Corollary~\ref{uniqueness} above.  This, however, contradicts the unqiueness of $\beta_{n-1}$ proved there so this case cannot occur and $\alpha_n$ is unique  in $(\beta_{n-1}, \beta_{n})$ as claimed. 
\end{proof}

\section{The Infinitely Renormalizable Tangent Map and the Strange Attractor}\label{inf ren}

Since the interleaved sequences $\{ \alpha_{n}\}_{n=1}^{\infty}$ and $\{ \beta_{n}\}_{n=1}^{\infty}$ 
are both  increasing and bounded they  have a common limit, $t_{\infty}$.   Since $t_{\infty}$ is a limit of the $\beta_n$, and we have proved that $f_{t}$ is $n$-renormalizable for $t \in (\beta_n, \pi)$, we can define $\R_{t_{\infty}}^{n} $ for all $n\geq 1$;  we say that  $f_{t_{\infty}}$ or $T_{t_{\infty}}$ is {\em infinitely renormalizable}.    In addition, since $t_{\infty}$ is a limit of the $\alpha_n$, and we have shown that $T_{t}$ has repelling periodic cycles of period $2^{n+1}$ that persist for all $t>\alpha_n$, $T_{t_{\infty}}$ has repelling periodic cycles of period $2^n$ for all $n$.  

In this section we will describe  properties of   
the orbits of the asymptotic values $\pm t_{\infty}$ under the map $T_{t_{\infty}}$.  As we have seen above, they are contained in the real and imaginary lines.   
We will give a topological description 
of the closure of the union of these two orbits which we denote by
  $C$.   We will show that $C \cap \Im$ is a perfect, uncountable,  totally disconnected and unbounded set while  $C \cap \Re$  is perfect, uncountable, totally disconnected and bounded and thus a Cantor set;   this Cantor set  
consists of two binary Cantor sets.   We argue  by analyzing the infinite sequence of  renormalizations of $T_{t_{\infty}}$. 

 As a corollary it will follow that almost every point in the real and imaginary axes is attracted by $C$ and  therefore that $T_{\infty}$ has no attracting or parabolic periodic cycles. 

 The standard  construction of a Cantor set in the real line ${\mathbb R}$ involves an infinite iterative process where at each step subintervals are removed from all the remaining intervals.   The remaining intervals are called {\em bridges} and the removed intervals are called {\em gaps}.  The Cantor set is called binary if only one gap is removed from bridge at each step.  
We give a more precise definition here which is adapted from~\cite{JCantor}. 

\medskip
\begin{definition}~\label{cantor}
 Let ${\mathcal I}=\{ {\mathcal I}_{n}\}_{n=0}^{\infty}$ be a sequence of
families of disjoint, non-empty, compact intervals:  the  {\em $n^{th}$ level bridges}. Let
${\mathcal G}=\{ {\mathcal G}_{n}\}_{n=0}^{\infty}$
be a sequence of families of disjoint, non-empty, open intervals: the {\em  $n^{th}$ level gaps}.
Let ${\mathcal CS}=\{{\mathcal I}, {\mathcal G}\}$.  

\smallskip

We call ${\mathcal CS}$ a binary Cantor system if
\begin{itemize}
\item[(i)] for each $0\leq n< \infty$ and each interval $I \in {\mathcal
I}_{n}$, there is a unique interval $G$ in ${\mathcal G}_{n}$ and two
intervals $L$ and $R$ in ${\mathcal I}_{n+1}$ which lie to the left and to the
right of $G$ such that $I=L\cup G\cup R$ (see Figure~\ref{fig12}), and
\item[(ii)] $CS= \cap_{n=0}^{\infty }\cup_{I\in {\mathcal I}_{n}}I$ is
totally disconnected. We call $CS$ the binary Cantor set  generated by 
the binary Cantor system ${\mathcal CS}$.
\end{itemize}
\end{definition}

\begin{figure}[ht]
\centering
\includegraphics[width=3in]{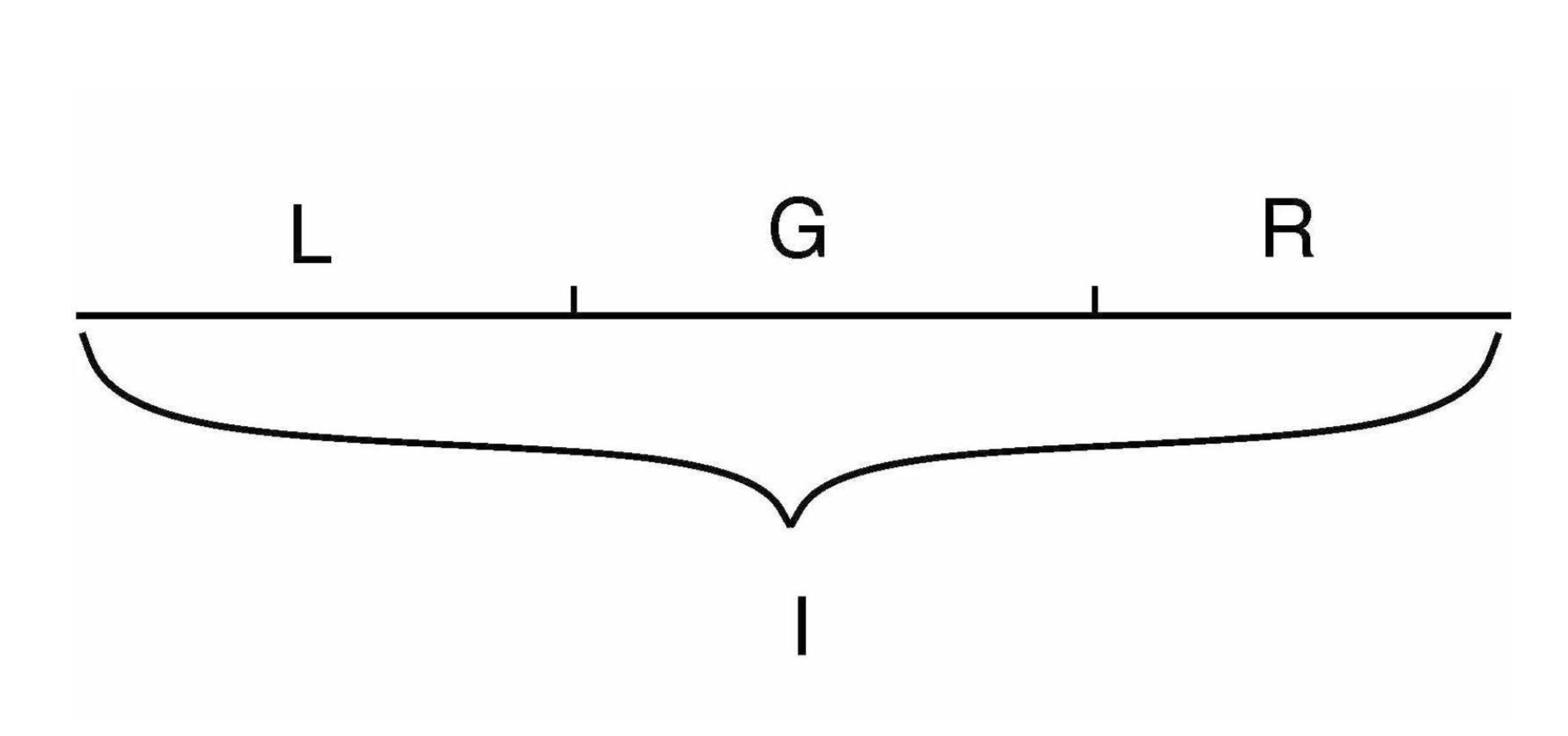}
\caption{Bridge and Gap}~\label{fig12}
\end{figure}

Suppose $T=T_{t_{\infty}}$ is the tangent map at the limit point $t_\infty$. 
Let 
$$
O=\rm{Orb} (\pm t_{\infty}) =\{ T^{n} (\pm t_{\infty})\}_{n=0}^{\infty}
$$ 
 be union of the orbits of both the asymptotic values, $\pm t_{\infty}$.  Note that by symmetry $-O=O$.
Let $C=\overline{O}$ be  the closure of the orbits. Recall that  $f=T^{2}$ maps both the real and imaginary lines to themselves.

\medskip
\begin{theorem}\label{infrenorm}
 The map $T$ is an infinitely renormalizable tangent map and 
the intersection $C_{r}={\mathbb R}\cap C$ consists of two binary Cantor sets;  $C_{r}$  is forward $f$-invariant and  $f: C_{r}\to C_{r}$ is minimal. 
The intersection $C_{i}=\Im \cap C$ is a totally disconnected, uncountable,   unbounded and perfect subset of the imaginary line $\Im$.  It is also forward $f$-invariant and minimal. 
\end{theorem}

\begin{proof}   We first prove the assertions about $C_{r}$.  The inductive construction in the proof of Theorem~\ref{bif} shows that the renormalized functions $\R^n$ converge to a limit $\R^{\infty}$.    To study the properties of the orbits of the asymptotic values  we set $\R^0=f$ and use the functions  $c_m(\R^n)=|\R^{nm}(\pi/2^{+})|$   introduced in that construction.    Since $t$ is fixed they are constants.   It will be convenient to set   $c_m=c_m(f)$ since these are the points of $O^+ \cup \RR$.  We also defined the pre-pole functions $a_n, b_n$ by the relations 
$\R^{n-1}(a_{n}) = \pm \pi/2, \R^{n-1}(b_{n})=\mp \pi/2$ where the sign depends on the parity of $n$.  These are the endpoints of the  intervals $I_n=I_n^{-} \cup I_n{+}$ that constitute the domain of $\R^n$.   

Recall  that  $I_{1}$ is divided into subintervals that are mapped by $\R$ as follows:
 $$
\R=f^{2}: \left\{ \begin{array}{ll}
                                              I_{11-}=[-b_{1},-\pi/2]&\to [-c_{1}, -c_{2}]\\
                                             I_{12-}=[-\pi/2,-a_{1}]&\to [c_{2}, c_{1}]\\
                                             I_{12+}=[a_{1}, \pi/2] &\to [-c_{1}, -c_{2}]\\
                                             I_{11+}=[\pi/2, b_{1}] &\to [c_{2}, c_{1}]
                                             \end{array}\right.
$$
 (See Figure~\ref{fig5})
 
The endpoints of the intervals in the range belong to $O \cap \RR$.  We denote these intervals  by 
\[ J_{0-}=[-c_{1}, -c_{2}]  \mbox{  and  }   J_{0+}=[c_{2}, c_{1}]. \]
     We set  ${\mathcal I}_{0\pm}=\{ J_{0\pm}\}$ and call it the set of {\em $0^{th}$-level bridges.} 

Applying $f$  we find  $c_{3}= f(-c_{2})>0$ and $c_{4}= f(c_{3})>0$,  so that 
$$
c_{2}< \pi/2<c_{4}< c_{3}<c_{1}.
$$
Now set  
\[ J_{1-}=[-c_{4}, -c_{2}], \, \, J_{1+}=[c_{2}, c_{4}], \, \, J_{11-}=[-c_{1}, -c_{3}]  \mbox{ and } J_{11+}=[c_{3}, c_{1}]. \]
   The maps 
$$
f: J_{11-}\to J_{1-}\hbox{ and } f: J_{11+}\to J_{1+}.
$$ 
are both continuous and onto. 
The interval $J_{1-}$ is divided into two intervals $J_{1-}^{l}=[-c_{4}, -\pi/2]$ and $J_{1-}^{r}=[-\pi/2, -c_{2}]$ by the pole $-\pi/2$, $f$ is continuous on each of these subintervals and
\[ f:J_{1-}^{l}  \to J_{21-}=[-c_{1}, -c_{5}]\subset J_{11-} \mbox{ and } f: J_{1-}^{r}\to J_{11+}. \] (See Figure~\ref{fig5} where these are intervals in the vertical direction and Figure~\ref{fig13} where the intervals are depicted horizontally).

Similarly,  the pole $\pi/2$ divides  
$J_{1+}$   into two intervals 
$ J_{1+}^{l}=[c_{2}, \pi/2]$ and $J_{1+}^{r}=[\pi/2, c_{4}]$   and  $f$ maps each of these subintervals continuously as follows:
$$ f:J_{1+}^{l}  \to J_{11-}  \mbox{ and } f: J_{1+}^{r}\to J_{21+}=[c_{5}, c_{1}].$$   
(See Figures~\ref{fig5} and~\ref{fig13}).

We now set 
$$
G_{0-} =(-c_{3}, -c_{4}) \hbox{ and } G_{0+} =(c_{4}, c_{3})
$$ 
so that 
$$
J_{0+}=J_{11+}\cup G_{0+}\cup J_{1+}\hbox{ and } J_{0-} = J_{11-}\cup G_{0-}\cup J_{1-}.
$$
The intervals ${\mathcal G}_{0\pm}=\{ G_{0\pm}\}$ form the {\em $0^{th}$-level  gaps} inside the {\em $0^{th}$-level  bridges} ${\mathcal I}_{0\pm}$.    The complementary intervals inside the $0^{th}$ level bridges form 
{\em the $1^{st}$-level  bridges}  ${\mathcal I}_{1\pm}=\{ J_{1\pm}, J_{11\pm}\}$. (See Figure~\ref{fig13}),

\begin{figure}[ht]
\centering
\includegraphics[width=5.5in]{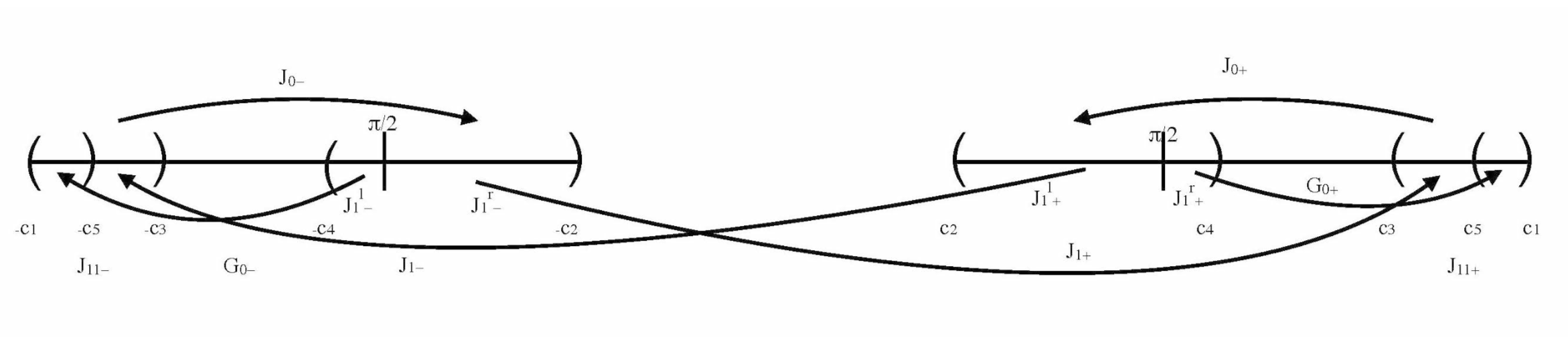}
\caption{ The $1^{st}$-level  bridges and the $0^{th}$-level gaps}~\label{fig13}
\end{figure}

We now go to the  second renormalization to get the $2^{nd}$-level  bridges and the $1^{st}$-level gaps; we have 

$$
\R^{2}=f^{2^{2}}: \left\{ \begin{array}{ll}
                                             I_{21-}=[-b_{2},-\pi/2]&\to [c_{2}, c_{4}]\\
                                             I_{22-}=[-\pi/2,-a_{2}]&\to [-c_{4}, -c_{2}]\\
                                             I_{22+}=[a_{2}, \pi/2] &\to [c_{2}, c_{2}]\\
                                             I_{21+}=[\pi/2, b_{2}] &\to [-c_{4}, -c_{2}]
                                             \end{array}\right.
$$
(See Figure~\ref{fig7})

Let 
$$ J_{2-}=[-c_{4}, -c_{8}], \, \, J_{2+}=[c_{8}, c_{4}], \, \, J_{21-}=[-c_{1}, -c_{5}]  \mbox{ and } J_{21+}=[c_{5}, c_{1}], $$ 
$$ J_{22-}=[-c_{6}, -c_{2}], \, \, J_{22+}=[c_{2}, c_{6}], \, \, J_{23-}=[-c_{7}, -c_{3}] \mbox{ and } J_{23+}=[c_{3}, c_{7}].$$
Then   $f$  is  continuous and  onto  
$$
 J_{21-}\stackrel{f}{\to} J_{22-}\stackrel{f}{\to}  J_{23+}\stackrel{f}{\to}  J_{2+} \hbox{ and } f: J_{21+}\stackrel{f}{\to} J_{22+}\stackrel{f}{\to} J_{23-}\stackrel{f}{\to} J_{2-}.
$$
As above, the pole divides    $J_{2-}$  into two subintervals, $J_{2-}^{l}=[-c_{4}, -\pi/2]$ and $J_{2-}^{r}=[-\pi/2, -c_{8}]$ and $f$ is  continuous and  onto on the subintervals
$$ f: J_{2-}^{l} \to J_{21-} \mbox{ and  }  f: J_{2-}^{r}\to J_{31+}=[c_{9},c_{1}]\subset J_{21+}.$$  Symmetrically, the pole divides 
  $J_{2+}$   into two subintervals $J_{2+}^{l}=[c_{8}, \pi/2]$ and $J_{2+}^{r}=[\pi/2, c_{4}]$  and $f$ is  continuous and  onto on the subintervals
$$ 
f: J_{2-}^{l}\to J_{31-}=[-c_{1}, -c_{9}]\subset J_{21-} \mbox{ and  } f: J_{2+}^{r}\to J_{21+}=[c_{5},c_{1}].
$$ 
(See Figure~\ref{fig7} where these are intervals in the vertical direction and Figure~\ref{fig14} where the intervals are  depicted horizontally).

Define the {\em $1^{st}$-level  gaps} as the set of intervals
$$
{\mathcal G}_{1-}=\{ G_{1-}, \; G_{11-}\}\; \hbox{ and }\; {\mathcal G}_{1+}=\{ G_{1+}, G_{11+}\}.
$$ 
inside the $1^{st}$-level  bridges where 
$$
G_{11-} =(-c_{5}, -c_{7}), \; G_{1-}= (-c_{8}, -c_{6})\;  \hbox{ and }\; G_{1+} =(c_{6}, c_{8}),\; G_{11+}=(c_{7}, c_{5}) 
$$ 
so that  
$$
J_{1-}=J_{2-}\cup G_{1-}\cup J_{22-},\;\; J_{11-}=J_{21-}\cup G_{11-}\cup J_{23-}
$$
and 
$$
J_{1+}=J_{2+}\cup G_{1+}\cup J_{22+},\;\; J_{11+}=J_{21+}\cup G_{11+}\cup J_{23+}.
$$
Now  define the {\em $2^{nd}$-level  bridges } as the set of intervals 
$$
{\mathcal I}_{2-}=\{ J_{2-}, J_{21-}, J_{22-}, J_{23-}\} \hbox{ and } {\mathcal I}_{2+}=\{ J_{2+}, J_{21+}, J_{22+}, J_{23+}\}.
$$
(See Figure~\ref{fig14}).

\begin{figure}[ht]
\centering
\includegraphics[width=5.5in]{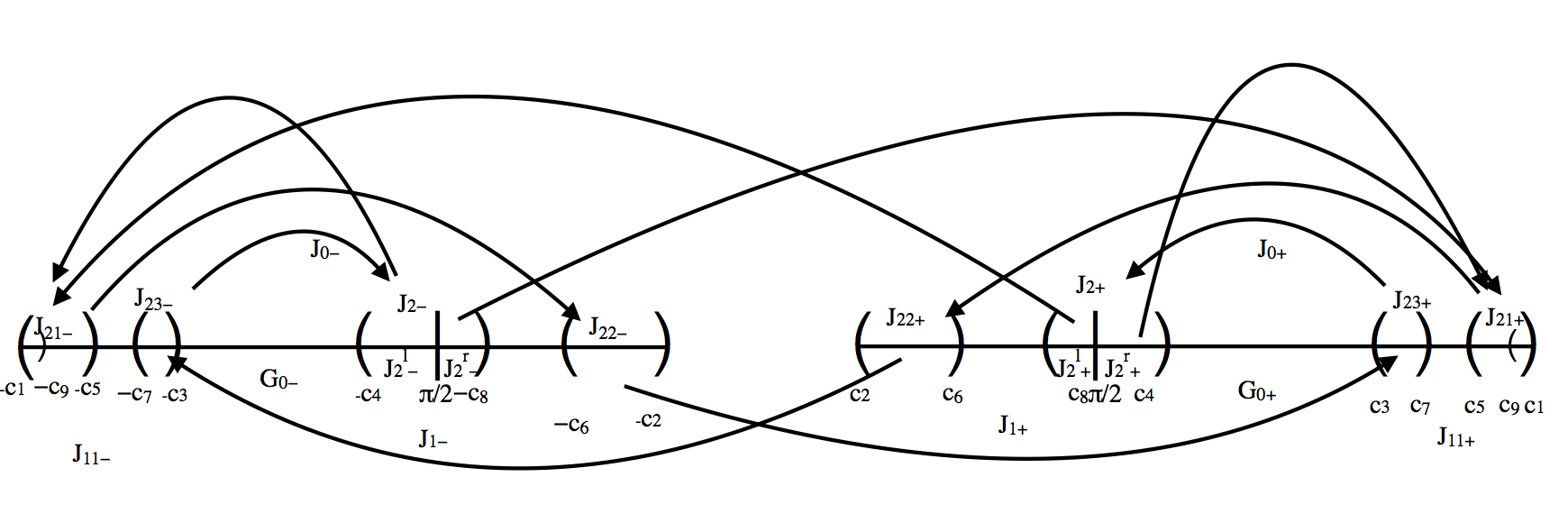}
\caption{ The $2^{nd}$-level bridges and the $1^{st}$-level  gaps}~\label{fig14}
\end{figure}
  
Using these steps as models for odd and even $n\geq 3$,  we  use the $n^{th}$ renormalization to define the $n^{th}$-level  bridges and the $(n-1)^{th}$-level  gaps.  Note that the parity of $n$ determines the orientation of the interval.  
 
We have 
$$
{\R}^{n}=f^{2^{n}}: \left\{ \begin{array}{ll}
                                             I_{n1-}=[-b_{n},-\pi/2]&\to [(-1)^n c_{2^{n-1}}, (-1)^n c_{2^{n}}]\\
                                             I_{n2-}=[-\pi/2,-a_{n}]&\to [(-1)^{n-1} c_{2^{n}}, (-1)^{n-1} c_{2^{n-1}}]\\
                                             I_{n2+}=[a_{n}, \pi/2] &\to [(-1)^n c_{2^{n-1}}, (-1)^n c_{2^{n}}]\\
                                             I_{n1+}=[\pi/2, b_{n}] &\to [(-1)^{n-1} c_{2^{n}}, (-1)^{n-1} c_{2^{n-1}}]
                                             \end{array}\right.
$$
(See Figures~\ref{fig5} and~\ref{fig7}.)

Let 
 \[ J_{n-} \mbox{ be the interval bounded by } -c_{2^{n+1}} \mbox{ and  } -c_{2^{n}} \] and let  
  \[ J_{n+} \mbox{ be the interval bounded by } c_{2^{n}} \mbox{ and  } c_{2^{n+1}}. \] 
 Which endpoint is the left one depends on the parity of $n$.   For example if $n$ is odd, $-c_{2^{n+1}}$ is to the left of  $-c_{2^{n}}$ while if $n$ is even it is to the right.

At the $n^{th}$-level we get $2^{n+1}$ subintervals.   Let $m=1, \ldots, 2^{n}$ and let 
 \[ J_{nm-} \mbox{ be the interval bounded by } -c_{m} \mbox{ and  } -c_{2^{n}+m} \] and   
  \[ J_{nm+} \mbox{ be the interval bounded by } c_{2^{n}+m} \mbox{ and  } c_{m}. \] 
 Again, which endpoint is the left one depends on the parity of $n$.  Note that  $J_{n2^{n}-}$, $J_{n0-}$  and $J_{n-}$ refer to the same interval;  $J_{n2^{n}+}$, $J_{n0+}$ and   $J_{n+}$ are the same interval.   

 Then 
$$
\cup_{m=1}^{2^{n}} J_{nm-} \subset J_{1-} \hbox{ and } \cup_{m=1}^{2^{n}} J_{nm+} \subset J_{1+}.
$$ 
If $m=2k+1$, $k=0, \ldots, 2^{n-1}-1$ then  
$$
f: J_{nm-} \to J_{n(m+1)-} \hbox{ and } f: J_{nm+} \to J_{n(m+1 )+}.  
$$

Recall that the poles $\pm \pi/2$ are contained inside the intervals  $(-c_4,-c_2)$ and $( c_4, c_2)$ respectively, and if  $m=2k$, $k=1, \ldots 2^{n-1}$, $J_{nm\pm}$ are subintervals that   lie either to the right or left of the pole.   This divides them into two groups: 

\[ \mbox{ if  } J_{n(2k)-}\subset (-\pi/2, -c_{2}) \mbox{ then }
f: J_{n(2k)-} \to J_{n(2k+1)+}   \mbox{ and } \]
\[ \mbox{ if } J_{n(2k)-}\subset (-c_{4}, -\pi/2) \mbox{  then } f: J_{n(2k)-} \to J_{n(2k+1)-}.  \]
The symmetric intervals $J_{n(2k)+}$ are also divided into two groups: 
\[ \mbox{ if  } J_{n(2k)+}\subset (c_{2}, \pi/2) \mbox{ then }
f: J_{n(2k)+} \to J_{n(2k+1)-} \mbox{ and } \]
\[ \mbox{ if } J_{n(2k)+}\subset (\pi/2, c_{4}) \mbox{ then }  f: J_{n(2k)+} \to J_{n(2k+1)+}. \]

The intervals $J_{n-}$  and $J_{n+}$ are divided by the poles they contain; because the parity of $n$ changes which endpoint is the left one, we label the intervals differently in each case.   If $n$ is odd we denote the respective subintervals as: 
 $$ J_{n-}^{l}=[-c_{2^{n+1}}, -\pi/2]  \mbox{  and  }  J_{n-}^{r}=[-\pi/2, -c_{2^{n}}] $$
 $$ J_{n+}^{l}=[c_{2^{n}}, \pi/2] \mbox{ and }  J_{n+}^{r}=[\pi/2, c_{2^{n+1}}].$$
  On each subinterval  $f$ is  continuous and maps as follows:
$$ f:J_{n-}^{l}  \to J_{(n+1)1-}=[-c_{1}, -c_{2^{n+1}+1}]\subset J_{n1-} \mbox{ and  }  f: J_{n-}^{r}\to J_{n1+}, $$  
$$ f:J_{n+}^{l}  \to J_{n1-} \mbox{ and } f: J_{n+}^{r} \to J_{(n+1)1+}=[c_{2^{n+1}+1}, c_{1}]\subset J_{n1+}.$$
If $n$ is even we denote the respective subintervals as 
 $$ J_{n-}^{l}=[-c_{2^{n}}, -\pi/2]  \mbox{  and  }  J_{n-}^{r}=[-\pi/2, -c_{2^{n+1}}]$$
 $$ J_{n+}^{l}=[c_{2^{n+1}}, \pi/2] \mbox{ and }  J_{n+}^{r}=[\pi/2, c_{2^{n}}].$$
 Then $f$ maps as follows: 
$$ f:J_{n-}^{l}  \to J_{(n)1-} \mbox{ and  }  f: J_{n-}^{r}\to J_{(n+1)1+}=[c_{2^{n+1}+1}, c_{1}]\subset J_{n1+}, $$  
$$ f:J_{n+}^{l}  \to J_{(n+1)1-}=[-c_{1}, -c_{2^{n+1}+1}]\subset J_{n1-} \mbox{ and } f: J_{n+}^{r} \to J_{n1+}.$$

Now we define the gap intervals.   Set 
\[
G_{(n-1)-} \mbox{  as the interval with endpoints  } -c_{2^{n+1}} \mbox{ and } -c_{2^{n-1}+2^{n}} 
\mbox{ and } \]
\[ G_{(n-1)+} \mbox{  as the interval with endpoints  } c_{2^{n-1}+2^{n}}   \mbox{ and } c_{2^{n+1}}
\]
where again which endpoint is the left one depends on the parity of $n$. 
For $k=1, \ldots, 2^{n-1}-1$,  inside the interval  $J_{(n-1)k-}$ bounded by $-c_{k}$ and $-c_{2^{n-1}+k}$, we define $G_{(n-1)k-}$ as the subinterval bounded by 
$-c_{2^{n}+k}$ and $-c_{2^{n}+2^{n-1}+k}$.  Symmetrically, for any $J_{(n-1)k+}$ bounded by $c_{k}$ and $c_{2^{n-1}+k}$, 
we define $G_{(n-1)k+}$ as the subinterval bounded by 
$c_{2^{n}+k}$ and $c_{2^{n}+2^{n-1}+k}$.  Then  
$$
J_{(n-1)k-} =J_{nk-}\cup G_{(n-1)k-}\cup J_{n(2^{n-1}+k)-} 
$$
and 
$$
J_{(n-1)k+} =J_{nk+}\cup G_{(n-1)k+}\cup J_{n(2^{n-1}+k)+}.
$$ 
The {\em $(n-1)^{th}$-level  gaps} are the collection of intervals 
$$
{\mathcal G}_{(n-1)-}=\{ G_{(n-1) k-}\}_{k=0}^{2^{n-1}-1} \hbox{ and } {\mathcal G}_{(n-1)+}=\{ G_{(n-1) k+}\}_{k=0}^{2^{n-1}-1}.
$$
 The {\em $n^{th}$-level  bridges} are the collection of complementary intervals 
$$
{\mathcal I}_{n-}=\{ J_{nk-}\}_{k=0}^{2^{n}-1} \hbox{ and } {\mathcal I}_{n+}=\{ J_{nk+}\}_{k=0}^{2^{n}-1}.
$$

Since  $t_{\infty} > \beta_n$  for all $n$, we can use   $\R^n$  to define bridges and gaps at all levels.  In the limit we have

$$
{\mathcal I}_{-}=\{ {\mathcal I}_{n-}\}_{n=0}^{\infty} \hbox{ and } {\mathcal I}_{+}=\{ {\mathcal I}_{n+}\}_{n=0}^{\infty}
$$
and  
$$
{\mathcal G}_{-}=\{ {\mathcal G}_{n-}\}_{n=0}^{\infty} \hbox{ and } {\mathcal G}_{+}=\{ {\mathcal G}_{n+}\}_{n=0}^{\infty}.
$$
 These  define two Cantor systems 
$$
{\mathcal CS}_{-}=({\mathcal I}_{-}, {\mathcal G}_{-}) \hbox{ and } {\mathcal CS}_{+}=({\mathcal I}_{+}, {\mathcal G}_{+}).
$$
 Let 
$$
C_{-} = \cap_{n=0}^{\infty} \cup_{k=0}^{2^{n}-1} J_{nk-} \hbox{ and } C_{+} = \cap_{n=0}^{\infty} \cup_{k=0}^{2^{n}-1} J_{nk+}.
$$
These are both  binary Cantor sets.

From our construction, we see that 
$$
C_{r}=C_{-}\cup C_{+}=\overline{\{ f^{n}(\pm t_{\infty}) \}_{n=0}^{\infty}}={\mathbb R}\cap C
$$ 
is forward $f$-invariant and the map $f: C_{r}\to C_{r}$ is minimal. 

 Now $C_{+}\subset (0, \pi)$ contains $\pi/2$ and $C_{-}=-C_{+}\subset (-\pi, 0)$ contains $-\pi/2$.   Since $T$ is  odd and one to one on $(0, \pi)$,  
the image $C_{i}=T(C_{-})=T(C_{+})=\Im \cap C$ is a totally disconnected, perfect, and uncountable, unbounded subset in the imaginary line $\Im$.   It is also $f$-forward invariant and minimal. 

This completes the proof of  Theorem~\ref{infrenorm}.
\end{proof}

\section{Appendix}\label{app 2}
This appendix contains  the proof of Lemma~\ref{parabolicbur}. 
\begin{lemma1}[Lemma~\ref{parabolicbur}]
Suppose $f(z)=z+a_nz^n+o(z^n)$ is an analytic function defined on some neighborhood  of $0 \in \mathbb C$. 
\begin{enumerate}
\item   Suppose  $\lambda$ lies  inside a small disk, inside and tangent to the unit circle at the point $1$.   
Then $g_{\lambda} (z)= \lambda f(z)$ has one attracting fixed point $0$ and $(n-1)$ repelling fixed points  counted with multiplicity,  in a small neighborhood of $0$.
\item  Suppose  $\lambda$ lies  inside a small disk, outside and tangent to the unit circle at the point $1$. Then $g_{\lambda}(z)=\lambda f(z)$ 
has one repelling fixed point $0$ and $(n-1)$ attracting fixed points counted with multiplicity, in a small neighborhood of $0$.
\end{enumerate}
\end{lemma1}

\begin{proof}
By Rouch\'e's theorem, $\lambda f(z)=z$ and $\lambda (z+a_{n}z^n)=z$ have the same number of solutions in a small neighborhood of $0$ and  by continuity, the 
corresponding multipliers are close to each other. 
So without loss of generality, suppose $f(z)=z+a_n z^n$. Then solutions  of $g_{\lambda}(z)=\lambda f(z) =z $ are $0$, with   multiplier  $\lambda$, and all solutions of $a_n\lambda z^{n-1}+(\lambda-1)=0$, each with  multiplier $\lambda+ n(1-\lambda).$ 
Consider two disks 
\[  
D_1=\{z\ :\ |z-n|<n-1\} \quad \mbox{and}\quad D_2=\{z\ : \ |z-\frac{n}{n-1}|<\frac{1}{n-1}\}.  \]
The map $\lambda \mapsto n-(n-1)\lambda$ takes the unit disk $\Delta=\{ z\:\  |z|=1\}$ to the disk $D_1$ and takes the disk $D_2$ to the unit disk $\Delta$.
It follows that if  $|\lambda | <1$,  $0$ is attracting and all the other fixed points are  repelling,  all have the same multiplier, and this multiplier is in the disk $D_1$.
Similarly, if $\lambda \in D_2$, $0$ is repelling and all the other fixed points are attracting, all have the same multiplier, and this multiplier is in the disk $D_2$.    By Rouch\'e's theorem $\lambda f(z)=z$  and $\lambda f(z)=g_{\lambda }(z)+o(z^n)$ have the same number of solutions as $g_{\lambda }(z)$ in a small neighborhood of $0$. 
\end{proof}

\vspace*{20pt}
\noindent Tao Chen, Department of Mathematics, Engineering and Computer Science,
Laguardia Community College, CUNY,
31-10 Thomson Ave. Long Island City, NY 11101.
Email: tchen@lagcc.cuny.edu

\vspace*{5pt}
\noindent Yunping Jiang, Department of Mathematics, Queens College of CUNY,
Flushing, NY 11367 and Department of Mathematics, CUNY Graduate
School, New York, NY 10016
Email: yunping.jiang@qc.cuny.edu

\vspace*{5pt}
\noindent Linda Keen, Department of Mathematics,  CUNY Graduate
School, New York, NY 10016, 
Email: LINDA.KEEN@lehman.cuny.edu; linda.keenbrezin@gmail.com


\begin{thebibliography}{12}

\bibitem{ACT}  A. Arneodo, P. Coullet and C. Tresser,  A possible new mechanism for the onset of turbulence.
Physics Letfers,  Volume {\bf 81A}, number 4, 19 January 1981, 197-201.

\bibitem{CT}  P. Coullet and C. Tresser, It\`erations  D'\`endomorphismes  et  Groupe De  Renormalisation. 
Journal  de Physique Colloque C5,  Suppl\`ement  au  no. 8,  tome  39 ,  ao\^ut  1978, page C5-25. 

\bibitem{CK} T. Chen and L. Keen,  Dynamics of Generalized Nevanlinna Functions, arXiv:1805.10974. 

\bibitem{DK} R. Devaney and L. Keen, Dynamics of tangent. {\it In Dynamical Systems, Proceedings, University of Maryland}, Springer-Verlag Lecture Notes in Mathematics, {\bf1342} (1988), 105-111.

\bibitem{Dou1} A. Douady, Syst\'emes dynamiques holomorphes. {\it Ast\'erisque}, {\bf 105-106} (1983), 39-64.

\bibitem{Dou2} A. Douady, chirugie sur les applications holomorphes. In {\it Proc. Int'l. Congress of Mathematicians}, AMS (1986), 724-738. 

\bibitem{FK} N. Fagella and L. Keen, Dynamics of purely meromorphic functions of bounded type. arXiv:1702.06563.

\bibitem{Fe1}  M. Feigenbaum, Quantitative universality for a class of non-linear transformations. J. Stat. Phys. {\bf 19} (1978), 25-52.

\bibitem{Fe2} M. Feigenbaum, The universal metric properties of non-linear transformations. J. Stat. Phys. {\bf 21} (1979), 669-706.

\bibitem{GJW}   F. Gardiner, Y. Jiang, and Z. Wang, Holomorphic motions and related topics. Geometry of Riemann Surfaces, London Mathematical Society Lecture Note Series, No. 368, 2010, 166-193.

\bibitem{JCantor} Y. Jiang, Geometry of Cantor Systems. Transactions of AMS,  Volume {\bf 351} (1999), Number 5, 1975-1987.

\bibitem{JBook} Y. Jiang, {\it  Renormalization and Geometry in One-Dimensional and Complex Dynamics}.
{\it Advanced Series in Nonlinear Dynamics}, Vol. {\bf 10} (1996) World Scientific Publishing Co. Pte. Ltd., River Edge, NJ.

\bibitem{K} L. Keen, Complex and real dynamics for the family $\lambda \tan z$.  {\it Proceedings of  the Conference on Complex Dynamic}, RIMS Kyoto University, 2001.

\bibitem{KK1} L. Keen and J. Kotus, Dynamics of the family of $\lambda \tan z$.  
{\it Conformal Geometry and Dynamics}, Volume {\bf 1} (1997), 28-57.

\bibitem{KK2} L. Keen and J. Kotus, On period doubling and Sharkovskii type ordering for the family $\lambda \tan z$. {\it Value Distribution Theory and Complex Dynamics} (Hong Kong, 2000), 51--78, Contemp. Math., 303, Amer. Math. Soc., Providence, RI, 2002.

\bibitem{LSS} G. Levin, S. van Strien, and W. Shen,  Monotonicity of entropy and positively oriented transversality for families of interval maps. arXiv:1611.10056v1.

\bibitem{McM} C.T. McMullen, {\it Complex Dynamics and Renormalization}. {\it Annals of Math. Studies}, {\bf 135},  Princeton Univ. Press, 1994.

\bibitem{MSBook} W. de Melo and S. van Strien, {\it One-Dimensional Dynamics}. Springer-Verlag, Berlin, Heidelberg,
1993. 

\bibitem{Mil} J. Milnor,  {\it Dynamics in One Complex Variable: Introductory Lectures}. Vieweg, 2nd Ed. 2000. 

\bibitem{MilA}   J. Milnor, {\it On the concept of attractor.}  Comm. Math. Phys. Vol. {\bf 99}, no. 2 (1985), 177-195.

\end{thebibliography}
\end{document}